\documentclass[11pt,reqno]{amsart}
\usepackage{amsmath,amsfonts,amssymb,bbm}
\usepackage{a4wide}
\usepackage[active]{srcltx}
\usepackage{dsfont}
\usepackage{xfrac}
\usepackage{xcolor}
\usepackage{mathrsfs}
\usepackage{mathtools}
\usepackage{tikz}
\usepackage{cancel}
\usepackage{pgfplots}
\usetikzlibrary{intersections}
\usepackage{tikz-3dplot}
\usetikzlibrary{shadings,intersections,patterns}
\usepackage{pgfplots}
\usepgfplotslibrary{fillbetween}
\usetikzlibrary{patterns}
\usetikzlibrary{intersections, calc, angles}
\usepackage[outline]{contour}
\contourlength{0.05em}
\usepackage{makecell}
\usepackage{verbatim}
\usepackage[shortlabels]{enumitem}
\usepackage{float}
\usepackage[unicode,hypertexnames=false,colorlinks=true,linkcolor=blue,citecolor=blue]{hyperref}
\usepackage{geometry}

\usepackage[
backend=biber,
style=alphabetic,
sorting=nyt,
maxbibnames=50,
giveninits=true,
maxalphanames=10
]{biblatex}
\usepackage{csquotes}
\renewbibmacro{in:}{%
	\ifentrytype{article}{}{}}
\DeclareFieldFormat[article]{volume}{\textbf{#1}}
\DeclareFieldFormat[article]{title}{\textit{#1}}
\DeclareFieldFormat[article]{journaltitle}{#1}
\DeclareFieldFormat[article]{number}{\textbf{#1}}
\renewbibmacro*{volume+number+eid}{%
	\printfield{volume}%
	\setunit*{\textbf{\addcolon}}%<---- was \setunit*{\adddot}%
	\printfield{number}\setunit{\addcomma\space}%
	\printfield{eid}}

\usepackage{cfr-lm}
\usepackage[osf]{libertine}
\usepackage{hyperref}
\usepackage[capitalise, nameinlink, noabbrev]{cleveref} %Leave it as the last package
\hypersetup{colorlinks=true,
	linkcolor=blue,
	citecolor=red,
}

\hypersetup{colorlinks,breaklinks,
	linkcolor=[rgb]{0,0,0},
	citecolor=[rgb]{0,0,0},
	urlcolor=[rgb]{0,0,0}}

\usepackage{esint}

\newcommand{\be}{\begin{equation}}
	\newcommand{\ee}{\end{equation}}
\theoremstyle{plain}

\newtheorem{theorem}{Theorem}[section]
\newtheorem{proposition}[theorem]{Proposition}

\newtheorem{lemma}[theorem]{Lemma}
\newtheorem{definition}[theorem]{Definition}

\newtheorem{remark}[theorem]{Remark}

\newcommand{\R}{\mathbb{R}}

\newcommand{\N}{\mathbb{N}}

\newcommand{\eps}{\varepsilon}

\newcommand{\HH}{\mathcal{H}}

\newcommand{\bea}{\begin{equation*}\begin{aligned}}
		\newcommand{\eea}{\end{aligned}\end{equation*}}

\newcommand{\res}{\mathop{\hbox{\vrule height 7pt width .5pt depth 0pt
			\vrule height .5pt width 6pt depth 0pt}}\nolimits}

\DeclareMathOperator{\dist}{dist}

\DeclareMathOperator{\diver}{div}

\DeclareMathOperator{\sgn}{sgn}

\crefname{subsection}{subsection}{subsections}

\numberwithin{equation}{section}

\title{The singular set in the thin Stefan problem}

\subjclass[2020]{Primary 35R35; Secondary 35K85, 35B65, 28A78}
\keywords{Free boundaries, thin Stefan problem, parabolic Signorini problem, singular set, generic regularity}

\author[M. Carducci]{Matteo Carducci}\thanks{}
\address {Matteo Carducci \newline \indent
	Classe di Scienze, Scuola Normale Superiore \newline \indent
	Piazza dei Cavalieri 7, 56126 Pisa, Italy}
\email{\href{mailto:matteo.carducci@sns.it}{matteo.carducci@sns.it}}

\author[R. Colombo]{Roberto Colombo}\thanks{}
\address {Roberto Colombo \newline \indent
	EPFL SB MATH\newline \indent
	Institute of Mathematics, Station 8, CH-1015 Lausanne, Switzerland}
\email{\href{mailto:roberto.colombo@epfl.ch}{roberto.colombo@epfl.ch}}

\author[C. Torres-Latorre]{Clara Torres-Latorre}\thanks{}
\address{Clara Torres-Latorre\newline \indent Instituto de Ciencias Matemáticas\newline \indent Consejo Superior de Investigaciones Científicas\newline \indent
	C/ Nicolás Cabrera, 13-15, 28049 Madrid, Spain.}
\email{\href{mailto:clara.torres@icmat.es}{clara.torres@icmat.es}}

\pgfplotsset{compat=1.18}
\begin{document}
	
	\newgeometry{margin=2cm}
	\begin{abstract} 
		We study the singular set of the free boundary for solutions of the thin Stefan problem in $\mathbb{R}^{n+1}\times \mathbb{R}$, which models the melting of a thin sheet of ice over water. After a suitable transformation, we reduce to studying the parabolic Signorini (thin obstacle) problem with a time-monotonicity condition.
        First, we prove that the free boundary is locally the graph of a differentiable melting function $t=\tau(x)$ defined on the $n$-dimensional thin space, and that the singular set coincides with the critical set of $\tau$. Next, we show that the singular set has parabolic Hausdorff dimension at most $n$. We then establish generic regularity up to dimension $n+1=4$, proving that the free boundary is regular for almost every time. Finally, under smoothness assumptions on the obstacle, we show that the singular set is contained in a $C^\infty$-hypersurface outside a set of parabolic dimension at most $n-1$. 

Our analysis combines ideas developed for the classical Stefan problem with several new ingredients specific to the thin setting, including frequency formulas for second blow-ups, epiperimetric inequalities, higher order expansions, and a fine stratification of the singular set.
	\end{abstract}
	
	\maketitle
	\tableofcontents
	
	\section{Introduction}
	   We study a free boundary problem that models the melting of a thin sheet of ice over a body of water. The ice occupies a thin layer, such as the surface of a lake, and the melted region expands as time evolves. The interface separating the still frozen part of the sheet from the melted region therefore moves along the thin set. In analogy with the classical Stefan problem \cite{lc,stefan}, we refer to this problem as the \emph{thin Stefan problem}. 
	After a suitable transformation of the temperature, the problem reduces to the study of time-monotone solutions of the parabolic thin obstacle problem
	\be\label{eq:parabolic-signorini-problem}\begin{cases}
		\partial_t u-\Delta u=0 &\quad \text{in }Q_1^+:=Q_1\cap\{x_{n+1}>0\},\\
		\min\{-\partial_{n+1}u,u-\varphi\}=0&\quad \text{on }Q_1':=Q_1\cap\{x_{n+1}=0\},\\
		u(x',x_{n+1},t)=u(x',-x_{n+1},t)&\quad \text{in }Q_1:=B_1\times(-1,0],\\
		\partial_t u>0&\quad \text{in }Q_1^+,\\
	\end{cases}\ee
	 with an obstacle $\varphi:B_{1}'\to \R$ determined by the initial conditions. 
     Throughout the paper, we will assume that
     \be\label{eq:ostacolo}\varphi\in C^{k,\gamma}(B_{1}'),\quad k\ge 4,\quad \gamma\in (0,1),\quad \ell:=k+\gamma.\ee
    The contact set and the free boundary are respectively
	$$\Lambda := \left\{(x',0,t)\in Q_{1}': u(x',0,t)=\varphi(x')\right\}\quad\text{and}\quad  \Gamma :=\partial_{Q_{1}'}\Lambda.$$ 
	In the physical model, $\Lambda$ corresponds to the region that remains covered by ice, while $\Gamma$ is the corresponding melting front. The time-monotonicity of $u$ in \eqref{eq:parabolic-signorini-problem} reflects the progressive and irreversible melting of the ice. A non-zero obstacle $\varphi$ allows spatially varying initial ice density, while a spatially uniform initial ice layer admits a zero obstacle formulation; 
   we refer to \cref{subsect:ice_melting} for a detailed derivation of \eqref{eq:parabolic-signorini-problem} from the physical model.
	
    The free boundary problem \eqref{eq:parabolic-signorini-problem} also arises naturally in diffusion through
	semipermeable membranes, where $\varphi$ represents a spatially varying exterior pressure; see \cite[Section I.2.2]{dl76} and \cite[Section 1.4]{fr21}. In this context, the time monotonicity assumption corresponds to a monotone loading regime. Related models include the boundary Stefan problem introduced in \cite{ac10} and the parabolic Signorini problem arising in boundary heat control studied in \cite{allenshi2016}.
	
	\smallskip
	The parabolic thin obstacle problem is also known as the parabolic Signorini problem, analogously to its well-known stationary counterpart originating in elastostatics \cite{sig33, sig59}, which has been the subject of considerable mathematical interest in the past decades \cite{ac04, acs08, gp09, psu12, survey}. Existence and regularity of solutions were studied in \cite{ls67, dl76, ath82, ura85, au88, au96}. The free boundary was first considered in the time-monotone setting in \cite{ath84}, where it was shown to be locally a Lipschitz graph $t = \tau(x',0)$. The general case (i.e.~without time-monotonicity) was later investigated systematically in \cite{dgpt17}, where the authors established the optimal parabolic $H^{3/2,3/4}$-regularity of solutions and studied structure and regularity of the free boundary. In particular, they proved that the free boundary can be decomposed into a relatively open \emph{regular set} and its complement, which we call the \emph{singular set}\footnote{In \cite{dgpt17}, the term \textit{singular points} refers to points where the contact set has zero density. Here, instead, singular points denote all free boundary points that are not regular, consistent with the terminology commonly adopted in the Stefan problem.}.
    Around regular points, the free boundary is a $C^{1,\alpha}$-hypersurface \cite{dgpt17,petrosyan2019boundedness}. Further developments in the regularity theory for the parabolic Signorini free boundary can be found in \cite{shi18,smooth-parabolic-signorini,fr21}.

    \smallskip
	A classical question in obstacle-type problems is that of \emph{generic regularity}. This goes back to Schaeffer in the 1970s, who constructed examples of smooth obstacles for which the free boundary has large singular sets \cite{sch77}. Nevertheless, he conjectured that for \textit{generic} solutions the free boundary should be smooth \cite{sch74}. 
    In recent years, generic regularity and bounds on the size of singular sets have been established in several obstacle-type free boundary problems.
    For the classical elliptic obstacle problem, it was proved in \cite{mon03,frs20} that in low dimensions, and for monotone one-parameter families of solutions, the free boundary is smooth for almost every value of the parameter. Analogous results were subsequently obtained for the Signorini problem and the fractional obstacle problem \cite{fr21,ft23,cc24}. In the parabolic setting, time itself plays the role of the parameter, so that generic regularity means almost everywhere \textit{in time}. This has been established for the parabolic Signorini problem in dimension $n + 1 = 2$ in \cite{fr21}, and for the Stefan problem in low dimensions in \cite{frs24,col2025}; see also \cite{supercooled} for the supercooled Stefan problem. Related generic regularity results also appear in the theory of minimal surfaces \cite{generic8,generic9-10,generic11}, one-phase free boundary problems \cite{fy23,fy25}, and Coulomb gases \cite{cf24}.

\smallskip

In this paper, we develop a fine analysis of the singular set in the thin Stefan problem, addressing two closely related questions: how large this set can be in space-time, and how often singularities can occur.
We establish bounds on its parabolic Hausdorff dimension and describe its fine structure in arbitrary dimension.
As a consequence of this analysis, we prove generic regularity results in dimensions $n+1\le 4$.

Our analysis combines ideas developed for the classical Stefan problem and the elliptic Signorini problem with several new ingredients in the parabolic thin setting. These include frequency formulas for second blow-ups, epiperimetric inequalities, cleaning estimates, and higher order expansions leading to a fine stratification of the singular set. For further details about the main ideas of the proof, we refer to \cref{subsect:strategy}.

	\subsection{Main results}\label{subsect:main}
	We now state our main results on the regularity and structure of the free boundary of the thin Stefan problem, formulated as \eqref{eq:parabolic-signorini-problem}. 
    Solutions of \eqref{eq:parabolic-signorini-problem} are assumed to satisfy the optimal regularity proved in \cite{dgpt17} (see \eqref{eq:reg-sol}).
    Based on the results in that work, free boundary points can be divided into regular points and singular points. At regular points, solutions detach from the obstacle with optimal rate given by the parabolic distance to the power $3/2$, whereas singular points are characterized by a higher order of detachment. We refer to \cref{subsec:frequency} and \cref{subsubsect:frequencies} for the precise definitions. 
    It is well-known that the free boundary is a $C^{1,\alpha}$-hypersurface around regular points \cite{dgpt17,petrosyan2019boundedness}.
Throughout the paper, we denote by $\Sigma$ the singular set and by $\pi_t(\Sigma):=\{t\in(-1,0]:\exists\,x\text{ such that }(x,t)\in\Sigma\}$ the set of singular times. We also denote by $\text{dim}_{\text{par}}$ the Hausdorff dimension with respect to the parabolic distance $\text{dist}_{\text{par}}((x_1,t_1),(x_2,t_2)):=\sqrt{|x_1-x_2|^2+|t_1-t_2|}$.
	
    Our first result concerns the space-time regularity of the entire free boundary $\Gamma$ and gives a characterization of singular points in terms of the melting-time function.
	
	\begin{theorem}\label{theorem:structure-spacetime-intro}
		Let $u$ be a solution of \eqref{eq:parabolic-signorini-problem}, with obstacle $\varphi$ satisfying \eqref{eq:ostacolo}. Then,
        the free boundary is locally the graph of a $C^1$ function $t=\tau(x',0)$, and the singular points are precisely the critical points of $\tau$.
	\end{theorem}
	 We refer to \cref{theorem:structure-spacetime} for a more precise statement.
After characterizing singular points in terms of the melting-time function, we next quantify the size of the singular set in space-time. The following result provides a first quantitative bound on its parabolic Hausdorff dimension.
	
	\begin{theorem}\label{t:thm0}
		Let $u$ be a solution of \eqref{eq:parabolic-signorini-problem}, with obstacle $\varphi$ satisfying \eqref{eq:ostacolo}. Then $$\text{dim}_{\text{par}}(\Sigma)\le n.$$
	\end{theorem}
	This bound is stronger than what follows from the $C^1$-graph representation alone, since an $n$-dimensional $C^1$ space-time hypersurface may have parabolic Hausdorff dimension $n+1$.
In dimension $n+1=2$, it also recovers the generic regularity result in \cite[Theorem 1.4]{fr21}. Indeed, the bound on the parabolic dimension implies that the set of singular times $\pi_t(\Sigma)$ has Hausdorff dimension at most $1/2$, and therefore the free boundary is regular for almost every time.
	
	\smallskip 
    
    Our next two results are the main contributions of the paper and concern generic regularity and finer structure of the singular set.
    The first improves on the result in \cite[Theorem 1.4]{fr21}, by showing generic regularity up to dimension $n+1=4$. 
	\begin{theorem}\label{t:thm1}
		Let $u$ be a solution of \eqref{eq:parabolic-signorini-problem} in dimension $n+1\le 4$, with obstacle $\varphi$ satisfying \eqref{eq:ostacolo}. Then $$\mathcal{H}^1(\pi_t(\Sigma))=0.$$  
		As a consequence, \,$\Gamma\cap\{t=t_0\}$ is an $(n-1)$-dimensional manifold of class $C^{1,\alpha}$ for almost every time $t_0\in(-1,0]$. Moreover, if $\varphi\equiv0$, the manifold is $C^\infty$. 
	\end{theorem}
	
	We refer to \cref{prop10} for a more precise result that quantifies the size of the set of singular times in low dimensions. In particular, if the obstacle is smooth, $\dim_{\mathcal{H}}(\pi_t(\Sigma)) = 0$ in dimension $n+1=2$, and $\dim_{\mathcal{H}}(\pi_t(\Sigma)) \le \frac{1}{2}$ in dimension $n+1=3$. Moreover, in dimensions $n+1 \geq 5$, we obtain that $\Sigma\cap\{t = t_0\}$ has Hausdorff dimension at most $n-3$ for almost every time $t_0\in (-1,0]$. 
	\smallskip
	
	Our second key contribution is the following splitting of the singular set into a highly regular part that is extremely rare in time, and a remainder of lower parabolic dimension.
	\begin{theorem}\label{thm:n-1andtime-like}
		Let $u$ be a solution of \eqref{eq:parabolic-signorini-problem}, with obstacle $\varphi\in C^\infty(B_1')$. Then, there exists a subset $\Sigma^\infty\subseteq \Sigma$ such that $$\text{dim}_{\text{par}}(\Sigma\setminus \Sigma^\infty)\le n-1,\qquad \text{dim}_{\mathcal{H}}(\pi_t(\Sigma^\infty))=0,$$ 
		and $\Sigma^\infty$ is locally contained in an $n$-dimensional manifold of class $C^\infty$.
	\end{theorem}
	We refer to \cref{prop:general-kgamma} for a more general result in the case of $C^{k,\gamma}$ obstacles. 

    \smallskip
    
    Our main results have natural counterparts in the theory of the classical Stefan problem; see \cite[Theorems 1.1, 1.2, and 1.3]{frs24} and \cite[Theorems 1.1 and 1.2]{col2025}. Indeed, after the Duvaut transformation \cite{duvaut,figalli}, the classical Stefan problem reduces to the classical parabolic obstacle problem with a time-monotonicity assumption. A further connection arises in the analysis of singularities of the classical Stefan problem, where second blow-up limits at points of the top stratum solve the parabolic thin obstacle problem with zero obstacle; see \cite[Proposition 6.7]{frs24}.
	\subsection{A thin ice melting model.}\label{subsect:ice_melting}
	In this subsection, we introduce a simplified model for the melting of a thin layer of ice on the surface of a lake and show how, after a suitable transformation, it naturally leads to the parabolic Signorini problem \eqref{eq:parabolic-signorini-problem}.
	We identify the water region with an open set $\Omega \subseteq \R^{n+1}_+:=\R^{n+1}\cap \{x_{n+1}>0\}$,
	where the coordinate $x_{n+1}$ measures depth, so that the surface
	of the lake is $\Pi:= \partial \Omega \cap \{x_{n+1}=0\}$.
	We assume that the temperature of the water is everywhere at or above
	the freezing point. The state of the system at time $t \in [-1,0]$
	can be described by two functions:
	\begin{itemize}
		\item the temperature of the water
		$\theta:\overline\Omega\times [-1,0]\to [0,\infty)$;
        
        \smallskip
		\item the surface density of the ice
		$\mu:\Pi\times [-1,0]\to [0,\infty)$.
	\end{itemize}
	We assume that the water is calm, so that heat is transferred only by conduction, that the ice remains at the melting temperature, and that heat exchange with the air above the surface is negligible.
	All physical constants are normalized to $1$. See \cref{fig:thin-ice-model} for an illustration of the setup of our model.
	
	With these assumptions, the system evolves as follows. Denoting by $\mathbf{H}:=\Delta-\partial_t$ the heat operator, the temperature solves the heat equation
	\begin{equation*}
		\mathbf{H}\theta=0\qquad \text{in }\Omega \times (-1,0],
	\end{equation*}
	with initial and boundary conditions
	$$
		\begin{cases}
		    \theta(x,-1)= \theta_{0}(x)
		 &\qquad\forall x\in \overline{\Omega},\\
         \theta(x,t) = \theta_{b}(x,t)&\qquad
		 \forall x\in \partial \Omega \cap \R^{n+1}_{+},
		\quad \forall t\in (-1,0].
		\end{cases}$$
	On the surface, wherever ice is present, i.e.~$\mu>0$, the water in contact with it remains at the melting temperature, i.e.~$\theta=0$. Equivalently,
	\begin{equation}\label{eq:thin-ice-signorini-condition}
		\theta\mu=0\qquad \text{on }  \Pi \times [-1,0].
	\end{equation}
	Moreover, the heat flux from below is entirely used to melt the ice, giving the Stefan condition
	\begin{equation}\label{eq:phase-transition-conditions}
		-\partial_{t}\mu = \partial_{n+1}\theta
		\qquad \text{on } \Pi \times (-1,0].
	\end{equation}
	In particular, the ice density $\mu$ is non-increasing. Indeed, if
	$\theta = 0$ at a point in $\Pi$, then
	$\partial_t\mu = -\partial_{n+1}\theta \leq 0$ since
	$\theta \geq 0$ in $\Omega$. On the other hand, if $\theta > 0$, by continuity
	of $\theta$ we have $\theta > 0$ in a neighborhood, and
\eqref{eq:thin-ice-signorini-condition} implies $\mu = 0$ there. In particular, $\partial_t\mu = 0$. In other words,
	$$\min\{-\partial_t\mu,\theta\} = 0 \qquad \text{on } \Pi \times (-1,0].$$
	We complete the system with an initial condition for the ice density,
	$$\mu(x,-1) = \mu_{0}(x) \qquad \forall x\in \Pi,$$
	where $\mu_0\ge0$ is a given, sufficiently regular function,
	not necessarily constant, satisfying $\mu_0\theta_0 = 0$ on $\Pi$.
	
	We now consider the Duvaut transform \cite{duvaut} of $\theta$: we introduce
	the function $v:\overline\Omega \times [-1,0] \to [0,\infty)$ defined as
	\begin{equation*}
		v(x,t):= \int_{-1}^{t}\theta(x,s)\,ds.
	\end{equation*}
	Note that $v\ge 0$ and $\partial_{t}v=\theta\ge 0$.
	Since $\mathbf{H}\theta=0$, in $\Omega \times (-1,0]$ we deduce that
	\begin{equation*}
		\mathbf{H}v(x,t)
		= \int_{-1}^{t}\Delta\theta(x,s)\,ds-\theta(x,t)
		= \int_{-1}^{t}\partial_{s}\theta(x,s)\,ds-\theta(x,t)
		= -\theta_{0}(x).
	\end{equation*}
	For simplicity, we assume that the water is initially at the melting temperature, namely $\theta_{0}\equiv 0$.
	Consequently,
	\begin{equation*}
		\mathbf{H}v=0 \qquad \text{in }\Omega\times (-1,0].
	\end{equation*}
	Similarly, using \eqref{eq:phase-transition-conditions} on
	$\Pi \times (-1,0]$ we get
	\begin{equation*}
		\partial_{n+1}v(x,t)
		= \int_{-1}^{t}\partial_{n+1}\theta(x,s)\,ds
		= -\int_{-1}^{t}\partial_{s}\mu(x,s)\,ds
		= \mu_0(x)-\mu(x,t).
	\end{equation*}
	Moreover, $v\mu=0$ in $\Pi\times(-1,0]$. Indeed, if $v(x,t)>0$,
	then $\theta(x,s)>0$ for some $s<t$. By
	\eqref{eq:thin-ice-signorini-condition}, this implies that $\mu(x,s)=0$, and since $\mu\ge0$ is non-increasing in time, we also have $\mu(x,t)=0$. Therefore
	\begin{equation*}
		v(\mu_0-\partial_{n+1}v)=0 \qquad \text{on }\Pi\times (-1,0].
	\end{equation*}
	Finally, we choose a time-independent function
	$\psi:\overline\Omega\to\R$ such that
	\begin{equation*}
		\Delta\psi=0\quad \text{in } \Omega,\qquad
		\partial_{n+1}\psi=\mu_0\quad \text{on } \Pi,
	\end{equation*}
	and consider the transformation
	\begin{equation*}
		u(x,t):= v(x,t)-\psi(x),\qquad \varphi(x'):= -\psi(x',0).
	\end{equation*}
	Then, in $\Omega\times (-1,0]$, $u$ is caloric and
	$\partial_{t}u=\partial_{t}v\ge 0$. Moreover, on
	$\Pi\times (-1,0]$ we have $u-\varphi=v\ge0$ and
	$-\partial_{n+1}u=\mu_0-\partial_{n+1}v= \mu\ge0$.
	In particular, $u$ solves
	\begin{equation*}
		\begin{cases}
			\mathbf{H}u=0\qquad
			&\text{in $\Omega\times (-1,0]$},\\
			\min\{u-\varphi,-\partial_{n+1}u\}=0\qquad
			&\text{on $\Pi\times (-1,0]$},\\
			\partial_{t}u\ge 0\qquad
			&\text{in $\Omega\times (-1,0]$}.
		\end{cases}
	\end{equation*}
	Once non-trivial heating has started in a connected component of $\Omega$, the strong maximum principle gives $\partial_tu=\theta>0$ there.
	After localizing in space and time, and extending $u$ evenly across $\{x_{n+1}=0\}$,
	we recover precisely \eqref{eq:parabolic-signorini-problem}. 
    The contact set $\Lambda=\{u = \varphi\}\cap \Pi$
	corresponds to the points where ice has always been present or
	just melted, and hence the free boundary
	$\Gamma=\partial_{\Pi}\{u = \varphi\}$ coincides with the melting front
	on the surface of the lake.
	
	\begin{figure}[htbp]
		\centering
		\includegraphics[width=0.75\linewidth]{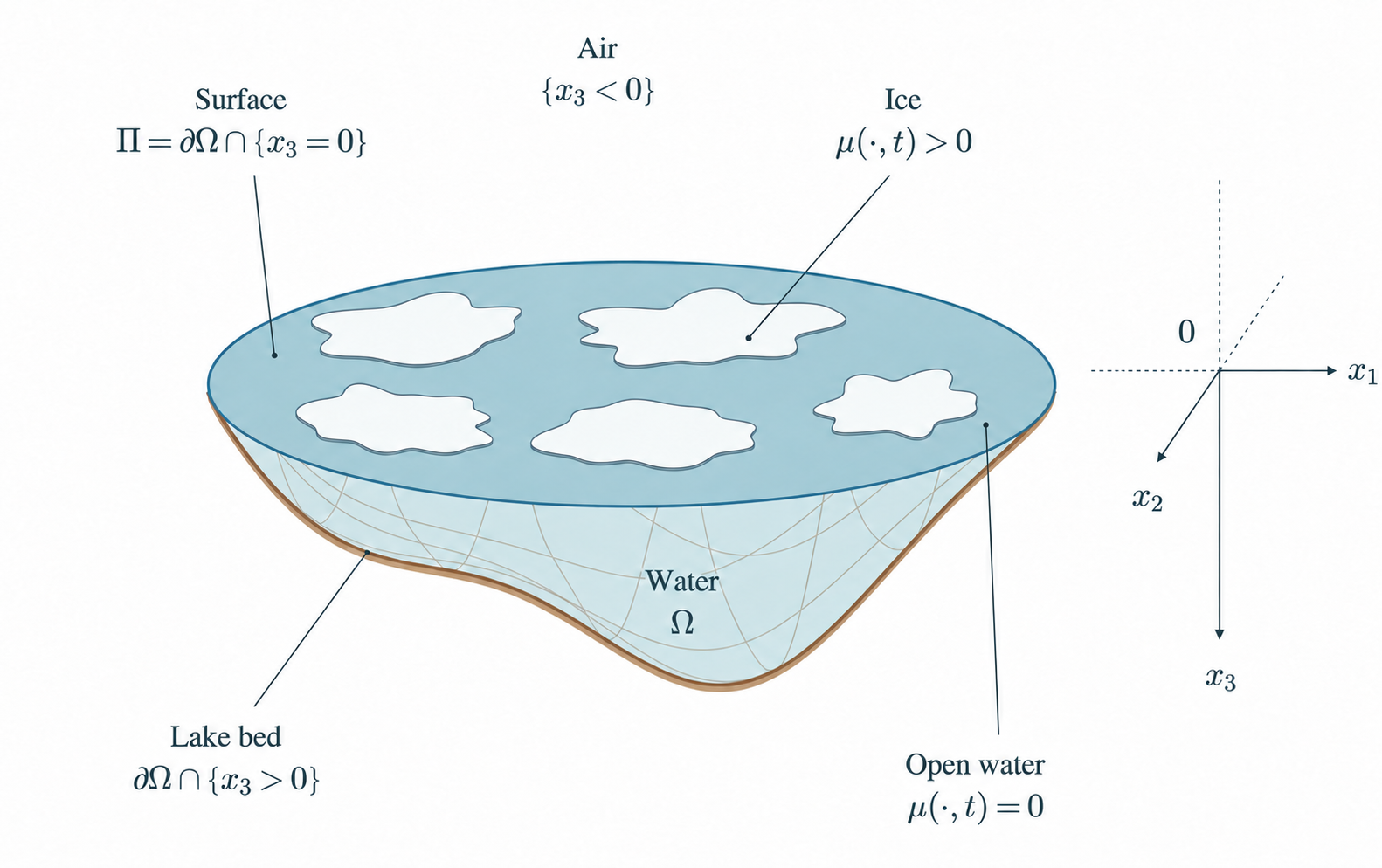}
		\caption{Schematic representation of the thin ice melting model.
			The water occupies the region $\Omega \subseteq \{x_3>0\}$,
			with the depth coordinate $x_3$ oriented downwards.
			On the flat surface $\Pi=\partial\Omega\cap\{x_3=0\}$,
			ice sheets of negligible thickness, where $\mu(\cdot,t)>0$,
			coexist with open water, where $\mu(\cdot,t)=0$.
			As the ice melts, its boundary may develop singularities.
			Our results show, however, that such singular behavior can occur only at a negligible set of times, while at almost every time the melting front is regular.}
		\label{fig:thin-ice-model}
	\end{figure}
	
	In the physical dimension $n+1=3$, if the initial ice density $\mu_0$ is of class $C^{3,\gamma}$, then $\varphi$ satisfies
	\eqref{eq:ostacolo} and our results show that the melting front is locally
	a $C^{1,\alpha}$-curve on the surface of the lake, for almost every time. 
	If $\mu_0$, and thus $\varphi$, is smooth, the exceptional set of times has Hausdorff
	dimension at most $1/2$. When the initial ice density $\mu_0$ is constant, we may choose $\psi(x)=\mu_0x_{n+1}$, so that $\varphi\equiv0$. In this case, our result implies that the melting front is in fact $C^\infty$ for almost every time. Finally, the local melting time $\tau(x',0)$ is a $C^1$ function of the position, and singular free boundary points correspond precisely to its critical points.
	
	\subsection{Strategy of the proofs}\label{subsect:strategy}
	Our proofs build on the general strategy introduced by Figalli, Ros-Oton and Serra for generic regularity in obstacle-type problems \cite{frs20, frs24}. The key idea of this strategy is to use the strict time-monotonicity to represent the singular set as a graph $t=\tau(x',0)$, and to reduce dimensional estimates for the singular set to two ingredients: a bound on the Hausdorff dimension of its spatial projection, and a quantitative control on how fast singular points can approach each other in the time direction. More precisely, for a subset $E$ of the free boundary, we introduce the notation (see \cref{def:cleaned}) $$E\in\text{Clean}(\beta),$$ which means, roughly speaking, that no point of $E$ can lie in a region of the form $t>t_0+C|x-x_0|^\beta$. Locally, this corresponds to a $\beta$-Hölder estimate of $\tau$ on $\pi_x(E)$. Combining this cleaning estimate with a bound on the spatial projection $$\dim_{\mathcal H}(\pi_x(E))\le \nu$$ we deduce estimates on the dimension of the time slices and of the set of exceptional times, through the geometric measure theory arguments of \cite{frs20,frs24} (see also \cite{ffr09}). Precisely, for almost every time, the singular set has dimension at most $\nu - \beta$ (or it is empty if $\nu < \beta$); see \cref{geometricmeastheory}. Thus, we need to find suitable decompositions of the singular set for which both the spatial dimension and the cleaning rate can be controlled.

We start by classifying singular points according to their frequency, i.e.~the homogeneity of their blow-ups. The general theory of the parabolic Signorini problem developed in \cite{dgpt17} gives the gap between the regular frequency $3/2$ and frequency $2$, while the gap between frequencies $2$ and $3$ was proved in \cite{col2025} (see also \cite{franceschini2025solutions} for the elliptic case). The strict time-monotonicity yields an additional restriction: a singular point cannot have frequency strictly larger than $3$ (see \cref{prop:onlyle3}). We therefore obtain
$$\Sigma=\Gamma_2\cup\Gamma_3,$$
where we denote by $\Gamma_\lambda$ the set of free boundary points of frequency $\lambda$; see \cref{subsec:frequency}. Consequently, the analysis naturally splits into quadratic and cubic points.
\subsubsection{Quadratic points}
	
	At quadratic points, we start from the uniqueness theory for even frequency developed in \cite{dgpt17}; see also \cite{shi18} for quantitative rates of convergence. Time monotonicity forces the quadratic blow-up to be independent of time, so that
    $$\widetilde u(x,t) = p_2(x) + o(|x|^2+|t|),$$
    where $\widetilde u$ is the translation of $u-\varphi$ centered at $(x_0,t_0)\in\Gamma_2$, with a correction term to deal with $\varphi$ not being caloric; see \eqref{eq:utilde}.
In \cref{sec:2nd_blowup_quadratic}, we develop a second blow-up analysis introducing a modified Poon frequency function for $\widetilde u-p_2$, and proving a new monotonicity formula, in the spirit of \cite{fj21,cc24}. A barrier argument similar to the one in \cite[Lemma 5.8]{frs24} then shows that the frequency of the second blow-up is again exactly $2$ (see Step 1 in \cref{prop:second-blow-up}). Consequently, every second blow-up is a parabolically $2$-homogeneous caloric polynomial, and can be explicitly characterized in terms of the first blow-up $p_2$ (see \cref{lem:characterization-second-blow-up}).
In \cref{sec:cleaning_quadratic}, adapting some ideas from \cite{col2025} to the thin setting, this yields the quadratic cleaning
$$\Gamma_2\in\operatorname{Clean}(2).$$
Finally, in \cref{sect:stratification_quadratic}, we combine these results with a Whitney-type stratification argument, as in \cite{dgpt17}. Precisely, if $\Gamma_2^j$ denotes the set of quadratic points whose blow-up has a $j$-dimensional invariant set with $j \in \{0,\ldots,n-1\}$, we prove that $\Gamma_2^j$ is locally contained in a $j$-dimensional $C^1$-manifold. In particular,
$$\pi_x(\Gamma_2)\text{ is $(n-1)$-rectifiable}
\qquad\text{and}\qquad
\dim_{\mathcal H}(\pi_x(\Gamma_2))\le n-1.$$ 
This concludes the analysis of quadratic points. 

\subsubsection{Cubic points}	
	The analysis of cubic points is more delicate and contains some of the main new ingredients of the paper. First, in \cref{sec:epiperimetric}, we establish uniqueness of blow-ups at cubic points through an epiperimetric inequality approach, in the spirit of \cite{shi18}. 
    Precisely, we prove an expansion of the form 
\be\label{eq:expans}\widetilde u(x,t) = p_3(x,t) + O\left(|x|^{3+\alpha}+|t|^{(3+\alpha)/2}\right),\ee
	where $p_3$ is the even reflection of a cubic polynomial vanishing on $\{x_{n+1} = 0\}$. 
    
    More generally, we prove uniqueness with a quantitative rate of convergence at every odd integer frequency in the parabolic Signorini problem (without assuming time-monotonicity). This provides a counterpart of the uniqueness theory developed in \cite{dgpt17,shi18} for even frequencies. We refer to \cref{prop:rate-for-w} and \cref{prop:odd-gap} for the precise statements. 
    Related uniqueness results at odd frequencies were previously obtained for the elliptic Signorini problem in \cite{sy23,cv24} and, in a parabolic setting, for cubic second blow-ups of the Stefan problem in \cite{col2025}.
    
    As a consequence of the expansion \eqref{eq:expans}, in \cref{sec-cleaning-cubic}, we prove a superquadratic cleaning estimate for cubic points. Precisely, we show that $$\Gamma_3\in\text{Clean}(2+\eta)$$ for some $\eta>0$. 
    The fact that the cleaning exponent is strictly larger than $2$ has an important consequence: the trace of $\widetilde u$ on $\{x_{n+1} = 0\}$ vanishes in a suitable backward cylinder. Hence, in this region the Signorini condition disappears and the equation becomes a standard heat equation. This allows us to bootstrap the cubic expansion and construct a polynomial approximation of any order allowed by the regularity of the obstacle.
Precisely, recalling that $\varphi\in C^{k,\gamma}(B_1')$, in \cref{sec-expansion-cubic} we prove an improvement of flatness lemma inspired by \cite{frs24},  which gives the higher order expansion
$$\widetilde u(x,t)=P_{\le k}(x,t)+O(|x|^{k+\alpha}+|t|^{(k+\alpha)/2})\qquad\text{for }|x|^2\lesssim|t|,$$
where $\alpha\in(0,\gamma)$ and $P_{\le k}$ is the even reflection of a caloric polynomial of degree $\le k$ vanishing on $\{x_{n+1}=0\}$.
    
    In \cref{sec-time-like-cubic}, we use the higher order expansion and Whitney extension theorem to stratify the set of cubic points as
$$\Gamma_3=\bigcup_{j=0}^n\Gamma_3^j,$$
according to the dimension $j$ of the kernel of the normal derivative of the cubic blow-up. The Whitney construction produces a $C^{k-1,\alpha}$ function $T$ such that, locally,
\be\label{eq:inclgamma3}\Gamma_3\subseteq\{t=T(x',0),\ \nabla_{x'} T(x',0)=0\},\ee
and whose Hessian encodes the geometry of the cubic blow-up. 
For $j<n$, the Hessian has positive rank in at least one tangential direction, and the implicit function theorem gives
$$\dim_{\mathcal H}(\pi_x(\Gamma_3^j))\le j\le n-1\qquad\text{for }j\in\{0,\ldots,n-1\}.$$
    The top stratum $j=n$ is more delicate. At these \emph{time-like cubic points}, the cubic blow-up has the form
$$p_3(x,t)=a|x_{n+1}|(6t+x_{n+1}^2),\quad a>0,$$
and is independent of the tangential variables. Hence the first blow-up does not provide any dimensional reduction in the spatial variables.

We therefore need to inspect the successive terms in the higher order expansion of the normal derivative. If any of the higher order terms has a non-trivial dependence on $x'$ at $t = 0$, its spine has dimension at most $n-1$ and a Whitney implicit function argument again gives a dimensional reduction. If no tangential dependence appears up to the order permitted by the obstacle, the associated time graph is flat up to that higher order, which yields an improved cleaning estimate. In this way we split the time-like stratum into a lower dimensional part, whose spatial projection has dimension at most $n-1$, and a part $\Gamma_3^{n,\ge k}\subseteq\Gamma_3^n$ satisfying
$$\Gamma_3^{n,\ge k}\in \text{Clean}(k+\alpha-1).$$
For a $C^\infty$ obstacle, the points that remain time-like at every order satisfy $$\Gamma_3^{n,\infty}\in\operatorname{Clean}(\infty).$$
The set $\Sigma^\infty$ in \cref{thm:n-1andtime-like} is precisely given by $\Sigma^\infty:=\Gamma_3^{n,\infty}$.
 \subsubsection{Conclusion of the proof.}
Once these results are established, the main results follow from the geometric measure theory lemmas of \cref{geometricmeastheory}. 
The following tables summarize, for $C^{k,\gamma}$ obstacles and $C^{\infty}$ obstacles respectively, the key dimensional bounds and cleaning results for each piece of the singular set, according to the splitting described above. 
	\renewcommand{\arraystretch}{1.2}
	\begin{table}[H]
		\centering
		\begin{tabular}{|c|c|c|c|}
			\hline
			Set & $\text{dim}_\HH(\pi_x(\cdot))$ & Cleaning exponent & $\text{dim}_{\text{par}}(\cdot)$ \\
			\hline
			$\Gamma_{2}$ & $n-1$ & $2$& $n-1$\\
			\hline
			$\Gamma_{3} \setminus \Gamma_3^{n,\ge k}$ & $n-1$ & $2+\eta$& $n-1$\\
			\hline
			$\Gamma_{3}^{n,\ge k}$ & $n$ & $k+\alpha - 1$ & $n$\\
			\hline
		\end{tabular}
		\caption{The splitting for $C^{k,\gamma}$ obstacles.}
		\vspace{0.2cm}
	\end{table}
	\renewcommand{\arraystretch}{1.2}
	\begin{table}[H]
		\centering
		\begin{tabular}{|c|c|c|c|}
			\hline
			Set & $\text{dim}_\HH(\pi_x(\cdot))$ & Cleaning exponent & $\text{dim}_{\text{par}}(\cdot)$ \\
			\hline
			$\Gamma_{2}$ & $n-1$ & $2$& $n-1$\\
			\hline
			$\Gamma_{3} \setminus \Gamma_3^{n,\infty}$ & $n-1$ & $2+\eta$& $n-1$ \\
			\hline
			$\Gamma_{3}^{n,\infty}$ & $n$ & $\infty$ & $n$\\
			\hline
		\end{tabular}
		\caption{The splitting for $C^{\infty}$ obstacles.}
		\vspace{0.2cm}
	\end{table}
Note that in the borderline case $n+1=4$, the quadratic stratum has both spatial dimensional bound and cleaning exponent equal to $2$, so the basic Sard-type estimate is not sufficient to conclude that it is generically empty. For that, we need to use additionally the $(n-1)$-rectifiability of $\pi_x(\Gamma_2)$. Thanks to the refined GMT argument of \cite{ffr09,col2025}, we thus obtain
$$\mathcal H^1(\pi_t(\Gamma_2))=0\qquad\text{for $n+1\le4$}.$$
Moreover, \eqref{eq:inclgamma3} allows us to apply Sard's theorem to the critical values of $T$, yielding generic regularity for cubic points in every dimension, whenever the obstacle $\varphi\in C^{k,\gamma}(B_1')$ is sufficiently regular. More precisely, in \cref{prop:sard}, we prove $$\mathcal H^1(\pi_t(\Gamma_3))=0\qquad\text{for $n+1\le k$}.$$ 
Thus, in higher dimensions, the main obstruction to generic regularity comes from the quadratic points. 
This is the counterpart of what happens in the classical Stefan problem: there, generic regularity can be proved for the top stratum $\Sigma_{n-1}$ in every dimension\footnote{It is sufficient to combine Step 1 in the proof of \cite[Theorem 1.2]{col2025} with Sard's argument used in the proof of \cref{prop:sard}.}, while the bottleneck is given by the lower-dimensional strata $\Sigma_{0},\ldots,\Sigma_{n-2}$.
	 \subsection{Organization of the paper} 
In \cref{sec:preliminaries}, we introduce the notation and preliminary results used throughout the paper. The analysis of quadratic points begins in \cref{sec:2nd_blowup_quadratic}, where we prove a Poon-type monotonicity formula for the second blow-up. The second blow-up analysis is then leveraged in \cref{sec:cleaning_quadratic} to establish quadratic cleaning for $\Gamma_2$. Finally, \cref{sect:stratification_quadratic} is devoted to the stratification of $\Gamma_2$, concluding the analysis at quadratic points. In \cref{sec:spacetime_reg}, we establish the space-time regularity of the free boundary proving \cref{theorem:structure-spacetime-intro}. The following four sections are then dedicated to the analysis of cubic points. In \cref{sec:epiperimetric}, we prove an epiperimetric inequality to show uniqueness of blow-ups and polynomial rate of convergence at cubic points. Building on this, in \cref{sec-cleaning-cubic} we establish superquadratic cleaning for $\Gamma_3$. The cleaning result is then used in \cref{sec-expansion-cubic} to obtain higher order expansions at cubic points in backward cylinders. This allows us, in \cref{sec-time-like-cubic}, to prove generic regularity for $\Gamma_3$ in every dimension, provided the obstacle is sufficiently regular, and establish a fine stratification of cubic points. Finally, in \cref{sec:proof_main}, we combine the results obtained in the previous sections to conclude the proof of our main theorems.
	  
	\section{Preliminaries}\label{sec:preliminaries}
	
	\subsection{Notation}\label{section-notations}
	Throughout the paper we will use the following notation.
	\begin{itemize}
		\item For any cylinder $\mathcal{C} = \Omega\times (t_{1},t_{2}]\subseteq \R^{n+1}\times \R$ we denote its parabolic boundary by $$\partial_{\text{par}}\mathcal{C}:= (\overline{\Omega}\times \{t_{1}\})\cup(\partial\Omega\times (t_{1},t_{2}]).$$

        \smallskip
		\item The parabolic dimension $\text{dim}_{\text{par}}(E)$ of a set $E\subseteq \R^{n+1}\times \R$ is its Hausdorff dimension with respect to the parabolic metric $\text{dist}_{\text{par}}((x_{1},t_{1}),(x_{2},t_{2})):= \sqrt{|x_1-x_2|^2+|t_1-t_2|}$.

        \smallskip
		\item We denote by $\pi_{x}$ and $\pi_{t}$ the projections on the spatial and temporal variables, respectively.

        \smallskip
		\item For $x\in \R^{n+1}$, we write $x=(x',x_{n+1})$. We let $\R^{n+1}_{\pm}:= \R^{n+1}\cap \{\pm x_{n+1}>0\}$, $\R^{n+1}_{0}:= \R^{n}\times \{0\}$, and we denote $E^{\pm}:=E\cap \{(x,t): \pm x_{n+1}>0\}$ and $E':=E\cap \{x_{n+1}=0\}$, for any $E\subseteq \R^{n+1}\times\R$. Moreover, when convenient, we identify the thin space-time with $\R^n\times\R$.

        \smallskip
		\item $Q_{r}(x_{0}, t_{0}):= B_{r}(x_{0})\times (t_{0}-r^{2},t_{0}]$ is the parabolic cylinder, $Q_{r}^{\pm}(x_{0}, t_{0}):= B_{r}^{\pm}(x_{0})\times (t_{0}-r^{2},t_{0}]$ are the parabolic half cylinders, and $Q_{r}'(x_{0}, t_{0}):= B_{r}'(x_{0})\times (t_{0}-r^{2},t_{0}]$ is the thin parabolic cylinder. 
		Similarly, we set $\mathscr{C}_\rho(x_0,t_0):=B_\rho(x_0)\times(t_0-\rho^2,t_0+\rho^2)$.
		When $(x_0,t_0)=(0,0)$, we drop the dependence on $(x_0,t_0)$.

        \smallskip
		\item $S_{r}:= \R^{n+1}\times (-r^{2},0]$ is the parabolic strip, $S_{r}^{\pm}:= \R^{n+1}_{\pm}\times (-r^{2},0]$ are the parabolic half strips, and $S_{r}':= (\R^{n}\times \{0\})\times (-r^{2},0]$ is the thin parabolic strip. The infinite strip is $S_{\infty}:= \R^{n+1}\times (-\infty,0]$, and $S_{\infty}^{\pm}, S_{\infty}'$ are defined accordingly. 

        \smallskip
		\item For any $m\in \N$, the parabolic Sobolev space $W^{2m,m}_{2}$ denotes the space
		of those functions in $L^2$, whose distributional derivative $\partial^\alpha_x\partial^j_t u$ belongs to $L^2$, for $|\alpha|+2j\le 2m$. Similarly, for any $\ell>0$, the parabolic H\"{o}lder space $H^{\ell,\ell/2}$ is the space of functions $u$ such that $\partial^\alpha_x\partial^j_t u$ is continuous for $|\alpha|+2j\le  \ell$, it is $(\ell-|\alpha|-2j)$-H\"{o}lder in space for $|\alpha|+2j\in (\ell-1,\ell)$, and $(\ell-|\alpha|-2j)/2$-H\"{o}lder in time for $|\alpha|+2j\in (\ell-2,\ell)$. We refer to \cite[Chapter 2.2]{dgpt17} for more details. We use $W^{1,2}$ for the ordinary Sobolev space in the space-time variables.

        \smallskip
		\item We write $\mathbf{H}:=\Delta-\partial_t$ for the heat operator.

        \smallskip
		\item We define the backward heat kernel $G(x,t):= (-4\pi t)^{-(n+1)/2}e^{|x|^{2}/4t}\mathbbm{1}_{(-\infty,0)}(t)$ on $\R^{n+1}\times \R$.

        \smallskip
		\item We define $Z:= x\cdot\nabla +2t\partial_{t}$, the generator of the parabolic scaling. 

        \smallskip
		\item We denote by $\mathcal{L}:=\Delta-\frac x2\cdot\nabla$ the Ornstein-Uhlenbeck operator. We also use $\mathcal{L}_\lambda:=\Delta-\frac x2\cdot\nabla +\frac\lambda2$ for its $\lambda/2$-shift.

        \smallskip
		\item For a function $v$, we write $\partial_{n+1}v(x',0,t):=\lim_{h\downarrow 0}h^{-1}(v(x',h,t)-v(x',0,t))$ for the one-sided normal derivative at the hyperplane, whenever this limit exists. We also denote by $\nabla_{x'}v$ the gradient of $v$ with respect to $x'=(x_1,\ldots,x_n)$, whenever it exists.

        \smallskip
		\item Unless otherwise specified, inequalities involving weak derivatives are understood almost
		everywhere, and differential inequalities are understood in the sense of distributions.
		Suprema and infima of quantities defined only almost everywhere mean essential suprema
		and infima.
        
        \smallskip
        \item For backward cylinders, $C_c^\infty(Q_r)$ denotes functions
smooth up to the terminal time whose support is a compact
subset of $Q_r$. No vanishing at the terminal time is required.
		
	\end{itemize}
	
	\subsection{Basic properties of solutions} 
	In this subsection we recall some properties of general solutions of the parabolic Signorini problem (i.e.~\eqref{eq:parabolic-signorini-problem} without the monotonicity assumption $\partial_tu>0$), referring the reader to \cite{dgpt17} for a detailed analysis and the proofs. In \cref{subsubsec:time-monotonicity} we then record a consequence of strict time monotonicity. 
	
	\subsubsection{Optimal regularity.}
	Throughout the paper, a solution $u:Q_{1}\to \R$ of \eqref{eq:parabolic-signorini-problem}, with obstacle $\varphi$ satisfying \eqref{eq:ostacolo} is
	assumed to belong to the parabolic regularity class
	\begin{equation}\label{eq:reg-sol}
		u\in W^{2,1}_{2}(Q_{1}^{+})\cap H^{3/2,3/4}(Q_{1}^{+}\cup Q_{1}'),
	\end{equation}
	that is, in the parabolic half cylinder, $u$ has two square integrable derivatives in space and one in time, and moreover, $u$ is $3/4$-H\"{o}lder in time and  $\nabla u$ is $1/2$-H\"{o}lder in space and $1/4$-H\"{o}lder in time, up to the hyperplane (see the notation in \cite[Chapter 2.2]{dgpt17}). General solutions of the parabolic thin obstacle problem according to \cite[Definition 4.1]{dgpt17} have this optimal regularity in a smaller cylinder (see \cite[Theorem 9.1]{dgpt17}). 
	
	The contact set and the free boundary, defined respectively as
	\begin{equation}\label{eq:def-freeboundary}
		\Lambda(u):= \left\{(x',0,t)\in Q_{1}': u(x',0,t)=\varphi(x')\right\},\qquad  \Gamma(u):=\partial_{Q_{1}'}\Lambda(u)
	\end{equation}
	are closed subsets of the thin parabolic cylinder $Q_{1}'$. When clear from the context, we will typically drop the dependence on $u$, writing $\Lambda$ and $\Gamma$ for $\Lambda(u)$ and $\Gamma(u)$. 
	
	\subsubsection{Reduction to zero obstacle.}\label{sec-subtraction-obstacle}
	Next, we describe how we shall reduce \eqref{eq:parabolic-signorini-problem} to a similar problem with zero obstacle and right-hand side suitably controlled in terms of the regularity of $\varphi$ (see \cite[Proposition 4.4]{dgpt17}). Given $(x_0,t_0)\in \Gamma$, we consider $q_{k}^{x_{0}}:\R^{n}\to \R$, the $k$-th order Taylor polynomial of $x'\mapsto \varphi(x_0'+x')$ at the origin,
	and we let $\widetilde q_k^{x_0}:\R^{n+1}\to \R$ be its even harmonic extension to $\R^{n+1}$:
	\begin{equation*}
		q_{k}^{x_{0}}(x'):= \sum_{|\alpha|\le k}\frac{D^{\alpha}\varphi(x_{0}')}{\alpha!}(x')^\alpha,\qquad \widetilde q_k^{x_0}(x):= \sum_{j=0}^{\lfloor k/2\rfloor}\frac{(-1)^{j}x_{n+1}^{2j}}{(2j)!}(\Delta_{x'})^{j}q_{k}^{x_{0}}(x').
	\end{equation*}
	We define
	\be\label{eq:utilde} \widetilde u^{x_0,t_0}(x,t):=u(x_0+x,t_0+t)-\varphi(x_0'+x')+q_k^{x_0}(x')-\widetilde q_k^{x_0}(x',x_{n+1}).
	\ee 
	
	Later on, $\widetilde{u}^{x_{0},t_{0}}$ will be multiplied by a suitable spatial cutoff function. This is needed because the functionals appearing in the monotonicity formulas are defined in the full parabolic strip $S_{1}$.
	Without loss of generality, assume $(x_{0},t_{0})=(0,0)\in\Gamma$, and let
	\begin{equation}\label{eq:subtraction-obstacle+cutoff}
		w(x,t):= \zeta(x)\widetilde{u}^{0,0}(x,t), 
	\end{equation}
	extended by zero for $x\in \R^{n+1}\setminus B_{1}$.
	Here $\zeta\in C^{\infty}_{c}(B_{15/16},[0,1])$ is a cutoff function, even in $x_{n+1}$, and $\zeta\equiv 1$ in $B_{7/8}$. 
	Then $w\in W^{2,1}_{2}(S_{1}^{+})\cap H^{3/2,3/4}(S_{1}^{+}\cup S_{1}')$ has compact spatial support and solves
	\begin{equation}\label{eq:w-solves}
		\begin{cases}
			\mathbf{H}w= f\qquad &\text{in $S_{1}^{+}$},\\
			\min\{-\partial_{n+1}w,w\}=0\qquad &\text{on $S_{1}'$},
		\end{cases}
	\end{equation}
	where $f=\zeta \mathbf{H}\widetilde{u}^{(0,0)}+2\nabla \zeta \cdot \nabla \widetilde{u}^{(0,0)}+\Delta \zeta \widetilde{u}^{(0,0)}$ which we extend evenly with respect to $x_{n+1}$. Note that $f$ does not depend on time in $Q_{7/8}$. If $\ell=k+\gamma$ is the regularity of the obstacle as in \eqref{eq:ostacolo}, one can show that $f$ is estimated by 
	\begin{equation}\label{controlf}
        \begin{aligned}
        |f(x,t)|
        &\le C_f (|x|+\sqrt{-t})^{\ell-2}
        && \forall(x,t)\in S_1,\\
        |\partial_x^\alpha\partial_t^j f(x,t)|
        &\le C_f (|x|+\sqrt{-t})^{\ell-2-|\alpha|-2j}
        && \forall(x,t)\in B_{7/8}\times(-1,0],\quad |\alpha|+2j\le k-2,
        \end{aligned}
    \end{equation}
	for some $C_{f}>0$ depending only on $n$, $k$, $\gamma$, $\lVert u\rVert_{L^{\infty}(Q_{1})}$, $\lVert \nabla u\rVert_{L^{\infty}(Q_{1})}$, and $\lVert \varphi\rVert_{C^{k,\gamma}(B_{1}')}$.
	
	Clearly, the same procedure can be applied around any free boundary point $(x_{0},t_{0})\in \Gamma$, noting that $u(x_{0}+r\cdot, t_{0}+r^{2}\cdot)$ is a solution of \eqref{eq:parabolic-signorini-problem} with obstacle $\varphi(x_{0}'+r\cdot)$ and $(0,0)$ belongs to its free boundary, for sufficiently small $r>0$.
	
	\subsubsection{Quantitative time monotonicity.}\label{subsubsec:time-monotonicity} To conclude this section, we record in the following lemma a straightforward consequence of the monotonicity assumption $\partial_{t}u>0$ in \eqref{eq:parabolic-signorini-problem}.
	\begin{lemma}[Monotonicity property]\label{lemma:monotonicity-ass}
		Let $u$ be a solution of \eqref{eq:parabolic-signorini-problem}, with obstacle $\varphi$ satisfying \eqref{eq:ostacolo}. Then, there exists $c>0$, depending also on $u$, such that 
		$$\partial_{t}u\ge c|x_{n+1}|\qquad\text{in }Q_{3/4}.$$
	\end{lemma}
	\begin{proof}
		First, for $h>0$ sufficiently small, by the strong maximum principle applied to the nonnegative function $u_h(x,t):=h^{-1}(u(x,t)-u(x,t-h))$, which is caloric in $Q_{15/16}^+$, we obtain 
        $$u_h\ge c\qquad\text{in }Q_{7/8}\cap\{|x_{n+1}|>1/2\},$$ for some $ c>0$.
		Let $\psi: Q_{7/8}\to [0,\infty)$ be the barrier function satisfying 
		$$
		\begin{cases}
			\mathbf{H}\psi=0 &\qquad \text{in }Q_{7/8}\setminus Q_{7/8}',\\
			\psi=0&\qquad \text{on } Q_{7/8}'\cup(\partial_{\text{par}} Q_{7/8}\cap \{|x_{n+1}|\le 1/2\}),\\
			\psi=1&\qquad \text{on }\partial_{\text{par}} Q_{7/8}\cap \{|x_{n+1}|>1/2\}.
		\end{cases}
		$$
		Then, by the maximum principle, $u_h\ge c\psi$ in $Q_{7/8}$, and thus, the parabolic Hopf lemma yields
		$$u_h\ge c|x_{n+1}|\qquad\text{in }Q_{3/4},$$ for some smaller $c>0$. Passing to the limit as $h\downarrow 0$, we conclude the proof. 
	\end{proof}
	\subsection{Splitting of the free boundary} We next describe how points in the free boundary $\Gamma$ can be classified according to the value of their \lq\lq frequency'', which measures the asymptotic rate of detachment of $u$ from the obstacle around a given free boundary point. 
	
	\subsubsection{The frequency function.}\label{subsec:frequency} Here we recall the definition of the frequency function, its monotonicity formula, and the corresponding blow-up analysis. We refer the reader to \cite[Chapters 6 and 7]{dgpt17} for further details. Let $G$ be the backward heat kernel defined above. For any function $v:S_{r}^{+}\to \R$ we define the quantity 
	\begin{equation*}
		H(r,v):= \frac{1}{r^{2}}\int_{S_{r}^{+}}v^{2}G.
	\end{equation*}
	Assume $(0,0)\in \Gamma$, and let $w$ be as in \eqref{eq:subtraction-obstacle+cutoff}. For any $\sigma\in (0,1)$, one can find $C>0$ such that the generalized frequency function
	\begin{equation*}
		\Phi^{\sigma}(r,w):= \frac{1}{2}re^{Cr^{\sigma}}\frac{d}{dr}\log \max\left\{H(r,w), r^{2(\ell-\sigma)}\right\}+2\left(e^{Cr^{\sigma}}-1\right)
	\end{equation*}
	is non-decreasing for $r\in (0,1)$. Here $\ell=k+\gamma$ is the regularity of the obstacle, as in \eqref{eq:ostacolo}. In particular, there exists $\Phi^{\sigma}(0^+,w):= \lim_{r\downarrow 0}\Phi^{\sigma}(r,w)$, and in fact this does not depend on the choice of the cutoff function $\zeta$ in \eqref{eq:subtraction-obstacle+cutoff}. Moreover, $\Phi^{\sigma}(0^+,w)\le \ell-\sigma$, and $\Phi^{\sigma}(0^+,w)= \min\{\Phi^{\sigma_0}(0^+,w),\ell-\sigma\}$ for any $0<\sigma_0\le \sigma$. We may therefore define the frequency of $u$ at $(0,0)$ as the quantity
	\begin{equation*}
		\kappa_{u}(0,0):= \sup_{\sigma\in (0,1)}\Phi^{\sigma}(0^+, w)\le \ell.
	\end{equation*}
	The importance of this definition lies in the following fact: if $\kappa=\kappa_{u}(0,0)<\ell$, the rescalings
	\begin{equation*}
		w_{r}(x,t):= \frac{w(rx, r^{2}t)}{H(r,w)^{1/2}}
	\end{equation*}
	converge as $r\downarrow 0$, up to subsequences in $H^{\alpha,\alpha/2}(Q_{R}^{+}\cup Q_{R}')$ for any $R>0$ and $\alpha\in (0,3/2)$, to a non-trivial global solution of the parabolic Signorini problem $q:S_\infty\to\R$ with zero obstacle, namely a solution of
	\be\label{eq:parabolic-signorini-problem-global}
	\begin{cases}
		\mathbf{H}q=0 &\qquad \text{in }S_\infty^+,\\
		\min\{-\partial_{n+1}q,q\}=0&\qquad \text{on }S_\infty',\\
		q(x',x_{n+1},t)=q(x',-x_{n+1},t)&\qquad \text{in }S_\infty.
	\end{cases}\ee
	Moreover, $q$ is parabolically $\kappa$-homogeneous, namely $q(\lambda x,\lambda^2t)=\lambda^\kappa q(x,t)$. 
	
	Of course, the same blow-up analysis can be performed around any point $(x_{0},t_{0})\in \Gamma$, leading to the definition of the frequency $\kappa_{u}(x_{0},t_{0})$. In particular, we may introduce the set of free boundary points with a given frequency
	\begin{equation}\label{eq:def-gammalambda}
		\Gamma_{\lambda}(u):= \left\{(x_{0},t_{0})\in \Gamma(u): \kappa_{u}(x_{0},t_{0})=\lambda\right\}\qquad \forall \lambda\in[0,\ell),
	\end{equation}
	and the set of free boundary points with frequency larger than or equal to a given value
	\begin{equation*}
		\Gamma_{\ge \lambda}(u):= \left\{(x_{0},t_{0})\in \Gamma(u): \kappa_{u}(x_{0},t_{0})\ge \lambda\right\}\qquad \forall \lambda \in \R.
	\end{equation*}
	When clear from the context, we will drop the dependence on $u$, writing $\Gamma_{\lambda}$ and $\Gamma_{\ge\lambda}$ for $\Gamma_{\lambda}(u)$ and $\Gamma_{\ge \lambda}(u)$. 
	\subsubsection{Classification of admissible frequencies.}\label{subsubsect:frequencies} The analysis carried out in \cite[Chapter 8]{dgpt17} shows that not all frequencies are admissible for points in the free boundary of a solution to the parabolic Signorini problem. More precisely, 
	\begin{equation}\label{eq:basic-admissible-frequencies}
		\Gamma_{\lambda}=\emptyset\qquad \forall\lambda\not \in \{3/2\}\cup [2,\infty).
	\end{equation}
	The lowest admissible homogeneity of blow-ups $3/2$ determines the optimal $H^{3/2,3/4}$-regularity of solutions. Moreover, $\Gamma_{3/2}$ is a $C^{1,\alpha}$-hypersurface in the space-time variable (see \cite{dgpt17,petrosyan2019boundedness}), relatively open in $\Gamma$. For this reason, $\Gamma_{3/2}$ is called the set of regular free boundary points. We call its complement $\Sigma:= \Gamma\setminus \Gamma_{3/2}$ the singular set. 
	The free boundary $\Gamma$ can then be decomposed into the two disjoint subsets of regular and singular points: 
	\begin{equation*}
		\Gamma= \Gamma_{3/2}\cup \Sigma.
	\end{equation*}
	Later, based on the strategy of \cite{franceschini2025solutions} for the elliptic thin obstacle problem, in \cite[Theorem 3.1]{col2025} the author proved that 
	\begin{equation}\label{eq:giacomo-classification}
		\Gamma_{\lambda}=\emptyset\qquad \forall\lambda\in (2,3).
	\end{equation}
	Then, \eqref{eq:basic-admissible-frequencies} and \eqref{eq:giacomo-classification} together imply that $\Sigma\subseteq \Gamma_{2}\cup \Gamma_{\ge 3}$. Note that all the classification results of admissible frequencies discussed above hold true for general solutions of the parabolic thin obstacle problem, independently of the monotonicity assumption $\partial_{t}u> 0$ in \eqref{eq:parabolic-signorini-problem}. Their proof is based on the analysis of homogeneous solutions of the global system in \eqref{eq:parabolic-signorini-problem-global}.
	
	Based on the monotonicity assumption, the following lemma proves that the admissible frequencies for the singular set are actually only $2$ and $3$.
	\begin{lemma}\label{prop:onlyle3}
		Let $u:Q_{1}\to \R$ be a solution of \eqref{eq:parabolic-signorini-problem}, with obstacle $\varphi$ satisfying \eqref{eq:ostacolo}. Then
		\begin{equation*}
			\Sigma=\Gamma_2\cup \Gamma_3.
		\end{equation*}
	\end{lemma}
	\begin{proof}
		By \eqref{eq:basic-admissible-frequencies} and \eqref{eq:giacomo-classification} it suffices to show that $\Gamma_{\ge\lambda}=\emptyset$ for every $\lambda>3$. Without loss of generality, we assume by contradiction that $(0,0)\in\Gamma_{\ge \lambda}$ for some $\lambda>3$.
		Then, by \cite[Lemma 9.3]{dgpt17}, for every $r$ small enough and for some $C>0$, we have 
		$$\|\widetilde u\|_{L^\infty(Q_r)}\le CH(2r,w)^{1/2}+Cr^{\ell},$$ 
		where $w=\zeta \widetilde u$ and $\widetilde u:=\widetilde u^{0,0}$ are defined in \eqref{eq:utilde}, \eqref{eq:subtraction-obstacle+cutoff}. 
		We consider the following two cases. If $\lambda<\ell$, then by using the monotonicity formula in \cite[Theorem 6.3]{dgpt17}, a standard argument (see \cite[Lemma 7.4]{dgpt17}) gives $H(2r,w)^{1/2}\le Cr^\lambda$. Similarly, if $\lambda\ge\ell$, then for every $\eps>0$ we can find $C_\eps>0$ such that $H(2r,w)^{1/2}\le C_\eps r^{\ell-\eps}$. Since $\ell>4$, in all cases, we have  \be\label{eq:controlu}\|\widetilde u\|_{L^\infty(Q_r)}\le Cr^{\kappa}\qquad \text{for some }\kappa>3.\ee
		Moreover, by \cref{lemma:monotonicity-ass},
		$$\partial_t\widetilde u=\partial_t u \ge c|x_{n+1}|\qquad\text{in }Q_{1/2}.$$
		As a consequence, for any sufficiently small $r>0$ we find
		$$\widetilde u(0,r,0)-\widetilde u(0,r,-r^2)=\int_{-r^2}^{0}\partial_t\widetilde u(0,r,s)\,ds\ge c r^3.$$
		On the other hand, \eqref{eq:controlu} gives $|\widetilde u(0,r,0)|\le Cr^\kappa$ and $|\widetilde u(0,r,-r^2)|\le Cr^\kappa$. This would give $2Cr^{\kappa}\ge cr^3$ for sufficiently small $r$, contradicting $\kappa>3$.
	\end{proof}

	\subsection{Geometric measure theory tools}\label{geometricmeastheory}
	In this section, we record a series of fundamental lemmas of geometric measure theory, mostly borrowed from \cite{ffr09,frs20,frs24,col2025}, which will allow us to estimate Euclidean and parabolic Hausdorff dimensions of different subsets of the free boundary. 
	
	We start by defining the notion of \lq\lq cleaning rate'' for a subset of space-time.
	\begin{definition}\label{def:cleaned}
		Consider a set $E\subseteq\R^{n}\times \R$,	and let $\beta >0$ be a given number. We say that $E$ is cleaned up with rate $\beta$, and write $$E\in\text{Clean}(\beta),$$ if, for all $(x_0,t_{0})\in E$, there are $\rho >0$ and $C>0$ such that
			\begin{equation*}
				\left\{(x,t)\in B_{\rho}(x_{0})\times \R:t>t_{0}+C|x-x_0|^{\beta}\right\}\cap E = \emptyset .
			\end{equation*}
			Moreover, we say that $E\in\text{Clean}(\infty)$, if $E\in\text{Clean}(\beta)$ for every $\beta>0$.
		\end{definition}
		Once we know that a set $E$ belongs to a given cleaning class, and that its spatial projection has a suitable Hausdorff dimension, we can then prove a generic bound on the size of its time slices.
		\begin{lemma}[\protect{\cite[Corollary 7.8]{frs20}}]
			\label{lemma:gmt}
			Consider a set $E\subseteq\R^n\times \R$ and let $E_{t}:=\{x\in \R^{n}:(x,t)\in E\}$ for all $t\in \R$. 
			Let \,$0\le \nu \le n$ and $\beta>0$, and assume that the following hold:
			\begin{itemize}
				\item[i)] $\text{dim}_{\mathcal{H}} \left(\pi_x(E) \right)\le \nu$;
                
				\smallskip
				\item[ii)] $E\in\text{Clean}(\beta).$
			\end{itemize}
			Then:
			\begin{itemize}
				\item [a)] if $\beta >\nu$, $\text{dim}_{\mathcal{H}}\left(\pi_t(E)\right)=\text{dim}_{\mathcal{H}}\left(\left\{t: E_{t}\neq \emptyset\right\}\right)\le \nu / \beta$;
                
				\smallskip
				\item [b)] if $\beta \le \nu$, $\text{dim}_{\mathcal{H}}\left(E_{t}\right)\le \nu - \beta$, for almost every $t\in \R$.
			\end{itemize}
		\end{lemma}
		Under similar assumptions, we may also get a bound on the parabolic dimension of a given set.
		\begin{lemma}[\protect{\cite[Lemma 8.9]{frs24}}]
			\label{lemma:par-gmt}
			Consider a set $E\subseteq \R^{n}\times \R$, and let $0\le \nu \le n$. Assume that the following hold:
			\begin{itemize}
				\item[i)] $\text{dim}_{\mathcal{H}} \left(\pi_x(E)\right) \le \nu$;
                
				\smallskip
				\item[ii)] $E\in\text{Clean}(2).$
			\end{itemize}
			Then $\dim_{\rm{par}}(E)\le\nu.$
		\end{lemma}
		
		Finally, we state the following stronger version of \cref{lemma:gmt}, under the additional assumption that $\pi_{x}(E)$ is $m$-rectifiable (with respect to the Euclidean distance). The proof we give here for the reader's convenience is just a reformulation in our notation of the argument contained in \cite[Section 6, proof of Theorem 1.1]{col2025}, which is based on \cite[Lemma A.3]{ffr09}.
		\begin{lemma}\label{lemma:colombo}
			Consider a set $E\subseteq \R^{n}\times \R$, and let $m\in \{1,\dots,n\}$ and $\beta>1$. Suppose that $\pi_x(E)$ is $m$-rectifiable and $E\in\text{Clean}(\beta)$. Then
			$$\mathcal{H}^{\frac{m}{\beta}}(\pi_t(E))=0.$$
		\end{lemma}
		\begin{proof}
			For every $x_{0}\in\pi_x(E)$, we let $\tau(x_{0})\in \R$ be the unique number such that $(x_{0},\tau(x_{0}))\in E$. Note that the function $\tau:\pi_{x}(E)\to \R$ is well-defined thanks to the cleaning assumption on $E$, and that moreover, we have $\pi_t(E)=\tau(\pi_x(E))$.
			We also localize the choice of constants $\rho, C>0$ in \cref{def:cleaned}, by setting 
			$$E_N:=\left\{x_0\in\pi_x(E):\left\{(x,t):\ |x-x_0|\le N^{-\beta},\ t>\tau(x_0)+N|x-x_0|^\beta\right\}\cap E=\emptyset\right\},$$
			so that $\pi_x(E)=\cup_{N\in\N}E_N$. For every $N\in \N$,  $E_N$ is $m$-rectifiable, and, by the cleaning result, we get $$|\tau(x)-\tau(y)|\le N|x-y|^\beta\quad\forall x,y\in E_{N}, \,|x-y|< N^{-\beta}.$$ 
			We may then apply \cite[Lemma 6.4]{col2025} to every element of a countable partition of $E_{N}$ made of sets with diameter smaller than $N^{-\beta}$ and deduce that $\mathcal{H}^{\frac{m}{\beta}}(\tau(E_N))=0$. Since this holds for every $N$, the conclusion follows. 
		\end{proof}
		\subsection{Two useful estimates} Finally, we record here some  useful parabolic estimates that will be needed often in the sequel. We first recall the following $L^\eps$-$L^\infty$ estimate for nonnegative subsolutions. The result is a consequence of the more general \cite[Theorem 4.16]{wan92}; see also \cite[Lemma 3.6]{frs24}.
		\begin{lemma}\label{lemma:wang-start}
			Let $v\in L^{\eps}(Q_{1})$ be a nonnegative function such that $\mathbf{H}v\ge g$ in $Q_{1}$, where $g\in L^\infty(Q_1)$. Then, for every $\eps>0$, there is $C_{\eps}$ depending only on $n$ and $\eps$ such that 
			$$\sup_{Q_{1/2}} v\le C_\eps \Big(\|v\|_{L^\eps(Q_1)}+ \|g\|_{L^\infty(Q_1)}\Big).$$ 
			The analogous estimate holds when $\mathbf{H}$ is replaced by $\mathcal{L}_\lambda-\partial_{\tau}$, where $\mathcal{L}_{\lambda}=\Delta-\frac x2\cdot\nabla +\frac\lambda2$ is the shifted Ornstein-Uhlenbeck operator, with $C_{\eps}$ depending also on $\lambda$.
		\end{lemma}
		We also recall the following parabolic version of the standard $L^1$-$L^1_{\text{weak}}$ estimate of
		Calderón and Zygmund for elliptic PDEs. The result follows from \cite[Theorem 3.5]{frs24}; see also \cite{jones1964class}.
		\begin{lemma}\label{lemma:calderon-zygmund}
			Let $v\in L^{1}(Q_{1})$ be a nonnegative function with $\partial_{t}v \in L^{1}(Q_1)$, such that $\mathbf{H}v\ge g$ in $Q_1$, where $g\in L^1(Q_1)$. Then $$\sup_{\theta>0}\,\theta\left|\{|\partial_t v|>\theta\}\cap Q_{1/2}\right|\le C\Big(\|v\|_{L^1(Q_1)}+\|g\|_{L^1(Q_1)}\Big),$$ for some $C>0$ depending only on $n$.
		\end{lemma}
		\begin{proof}
			Take a cutoff function $\eta\in C^{\infty}_{c}(Q_{1}, [0,1])$ such that $\eta\equiv 1$ in $Q_{3/4}$. Consider the measure $\mu:= \mathbf{H}v-g\ge 0$ and observe that 
			\begin{align*}
				\mu(Q_{3/4})&\le \int_{Q_{1}}\eta \,d\mu
				= \int_{Q_{1}}\Big(v(\Delta \eta +\partial_{t}\eta)-g\eta\Big)-\int_{B_{1}}v(\cdot,0)\eta(\cdot,0)\le C \left(\lVert v\rVert_{L^{1}(Q_{1})}+\lVert g\rVert_{L^{1}(Q_{1})}\right). 
			\end{align*}
			Then, $|\mathbf{H}v|(Q_{3/4})\le C( \lVert v\rVert_{L^{1}(Q_{1})}+\lVert g\rVert_{L^{1}(Q_{1})})$. Let $v_\eps:=v\ast \rho_\eps$, where $\rho_\eps$ is a standard mollifier in space-time, and the convolution is taken only over past times.
			By \cite[Theorem 3.5]{frs24}, we get
			$$\sup_{\theta>0}\,\theta \left|\{|\partial_t v_\eps|>\theta\}\cap Q_{1/2}\right|\le C\Big(\|v_\eps\|_{L^1(Q_{2/3})}+\|\mathbf{H}v_\eps\|_{L^1(Q_{2/3})}\Big)\le C\Big(\|v\|_{L^1(Q_1)}+\|g\|_{L^1(Q_1)}\Big),$$ for some $C>0$ depending only on $n$. Since $\partial_t v\in L^1(Q_1)$, the derivatives $\partial_t v_\eps$ converge to $\partial_t v$ in $L^1$ on inner cylinders and, along a subsequence, almost everywhere. Applying Fatou's lemma to the superlevel sets at each fixed $\theta>0$ gives the desired weak-type estimate.
		\end{proof}
		
		\section{The second blow-up at quadratic points}\label{sec:2nd_blowup_quadratic}
		In this section we consider the set of quadratic points $\Gamma_{2}$ in the free boundary of a solution $u$ of \eqref{eq:parabolic-signorini-problem}. The main result here is \cref{prop:second-blow-up}, which states that after subtracting the first quadratic blow-up, a sequence of rescalings converges to a second blow-up, which is a quadratic caloric polynomial. In order to pursue this second blow-up analysis, we will need a new monotonicity formula for a modified Poon frequency functional, which we obtain in \cref{prop:mon-formula-second-bu}.   
		
		\subsection{The first blow-up at quadratic points} First, we recall a result about the strong convergence of the solution to the unique first blow-up at quadratic points, together with the classification of quadratic blow-ups. Let $\mathcal{P}_2^+$ be the closed convex cone of admissible stationary quadratic polynomials
		\bea\label{eq:def-first-blow-ups}
		\mathcal{P}_2^+:=\left\{p:\R^{n+1}\to \R \ \mbox{polynomial}:\  \Delta p=0,\ \nabla p\cdot x=2p, \ p(x',0)\ge0,\ p(x',x_{n+1})=p(x',-x_{n+1})\right\}.
		\eea
		The following proposition follows from \cite[Theorem 12.6]{dgpt17} (which treats more generally even frequency points), after noting that, under the monotonicity assumption $\partial_{t}u> 0$, quadratic blow-ups must be independent of time.
		\begin{proposition}\label{prop:p2-indep-time}
			Let $u$ be a solution of \eqref{eq:parabolic-signorini-problem}, with obstacle $\varphi$ satisfying \eqref{eq:ostacolo}, and let $(0,0)\in \Gamma_{2}$. Then, there exists a unique non-zero $p\in \mathcal{P}_{2}^{+}$ such that
			\begin{equation*}
				\lVert \widetilde u-p\rVert_{L^{\infty}(Q_{r})}= o(r^{2})\qquad \text{as $r\downarrow 0$}, 
			\end{equation*} where $\widetilde u:=\widetilde u^{0,0}$ is defined in \eqref{eq:utilde}.
		\end{proposition}
		\begin{proof}
			By \cite[Theorem 12.6]{dgpt17}, there exists a quadratic parabolically homogeneous polynomial $p:S_{\infty}\to \R$ with  $\mathbf{H}p=0$, $p(x',0,t)\ge0$, and $p(x',x_{n+1},t)=p(x',-x_{n+1},t)$, such that 
			$$\widetilde u(x,t)=p(x,t)+ o(|x|^{2}+|t|).$$ 
			Since $\partial_{t}\widetilde{u}=\partial_{t}u\ge0$, we must have $\partial_{t}p\ge 0$. Writing $p(x,t)=at+q(x)$ for $a\ge 0$ and $q$ a $2$-homogeneous polynomial, since $p$ is nonnegative on the thin space, we must have $0\le p(0,t)=at$, for $t<0$. Therefore, $a=0$, so $p$ does not depend on $t$.
		\end{proof}
		
		We refer the reader to \cite[Theorem 1.2 and Section 5]{shi18} for a quantitative logarithmic convergence rate in \cref{prop:p2-indep-time} for the blow-up at even frequency points of the parabolic thin obstacle problem.

		\subsection{The monotonicity formula for the second blow-up}
		Before stating the monotonicity formula for the second blow-up, we introduce notation for a series of quantities that will appear in the computations below (see also the definitions in \cite[Chapter 6]{dgpt17}).
			
			For a sufficiently regular function $v:S_{1}^{+}\to \R$ and any $r\in (0,1)$ we let
			\begin{equation*}
				H(r,v):= \frac{1}{r^{2}}\int_{S_{r}^{+}}v^{2}G,\qquad I(r,v):= \frac{1}{r^{2}}\int_{S_{r}^{+}}|t||\nabla v|^{2}G,\qquad N(r,v):= \frac{2I(r,v)}{H(r,v)}.
			\end{equation*}
			For a fixed $\beta\in(0, \ell)$, we also define the regularized quantities 
			\begin{equation*}
				\widetilde{H}(r,v):= H(r,v)+r^{2\beta},\qquad \widetilde{I}(r,v):= I(r,v)+\frac{\beta}{2}r^{2\beta},\qquad \widetilde{N}(r,v):= \frac{2\widetilde{I}(r,v)}{\widetilde{H}(r,v)}.
			\end{equation*}
			For any $\kappa\in \R$, we let
			\begin{equation*}
				\mathcal{W}_{\kappa}(r,v):= \frac{1}{r^{2\kappa}}\left(I(r,v)-\frac{\kappa}{2}H(r,v)\right).
			\end{equation*}
			Finally, extending $v$ evenly with respect to $x_{n+1}$ to the full $S_{1}$, we let
			\begin{equation}\label{eq:def-F}
				F(r,v):= 
				\frac{1}{2r^{2}}\int_{S_{r}}|t|v\mathbf{H}vG =\frac{1}{r^{2}}\int_{S_{r}^{+}}|t|v\mathbf{H}v G +\frac{1}{r^{2}}\int_{S_{r}'}|t|v\partial_{n+1}v G.
			\end{equation}
			\begin{proposition}\label{prop:mon-formula-second-bu}
				Let $u$ be a solution of \eqref{eq:parabolic-signorini-problem}, with obstacle $\varphi$ satisfying \eqref{eq:ostacolo}.
				Suppose that $(0,0)\in \Gamma_{2}$, and let $p\in \mathcal{P}_{2}^+$. We denote by $v:=w-p$, where $w:=\zeta\widetilde u^{0,0}$ and $\widetilde u^{0,0}$ are defined in \eqref{eq:utilde}, \eqref{eq:subtraction-obstacle+cutoff}. Then, for every $\beta> 2$ such that $\ell+2-2\beta>0$, there exist constants $r_{0}\in (0,1)$ and $C>0$ for which the modified Poon frequency function
				\begin{equation*}
					\widehat{N}(r,v):= \left(\widetilde{N}(r,v)+Cr^{\ell+2-2\beta}\right)e^{Cr^{\ell-\beta}}
				\end{equation*}
				is non-decreasing in $(0,r_{0})$. More precisely, we have
				\begin{equation}\label{eq:mon-formula-second-bu}
					\frac{d}{dr}\widehat{N}(r,v)\ge \frac{1}{r}\frac{F(r,v)^{2}}{\widetilde{H}(r,v)^{2}}\qquad \text{for a.e. } r\in (0,r_{0}).
				\end{equation} 
				In particular, the limits $\widehat{N}(0^{+},v)=\widetilde{N}(0^{+},v)$ exist. 
			\end{proposition}
			\begin{proof}
				Recall that $w=\zeta \widetilde{u}^{(0,0)}$ solves \eqref{eq:w-solves}, with right-hand side $f$ satisfying \eqref{controlf}. Given $p\in \mathcal{P}_2^+$, we let $v=w-p$. We divide the proof into three steps.
				
				\smallskip
				\noindent \textit{Step 1: Preliminary estimates.}
				By \eqref{controlf} one gets the following integral estimate
				\begin{equation}\label{eq:bound-integral-f-rhs}
					\int_{S_{r}^{+}}t^{2}f^{2}G\le Cr^{2\ell+2}.
				\end{equation}
				Now, since $p\ge 0$ and $\partial_{n+1}w\le 0$ on $S_{1}'$, we have $v\partial_{n+1}v=-p\partial_{n+1}w \ge 0$. Therefore, \eqref{eq:def-F}, the Cauchy-Schwarz inequality and \eqref{eq:bound-integral-f-rhs} give
				\begin{equation}\label{eq:bound-below-Fv}
					F(r,v)\ge \frac{1}{r^{2}}\int_{S_{r}^{+}}|t|vfG \ge -Cr^{\ell}H(r,v)^{1/2}.
				\end{equation}

				We will use the following equations, understood for almost every $r$, which can be obtained by adapting the formulas derived in \cite[Chapter 6]{dgpt17}:
				\begin{equation}\label{eq:derivatives-H-and-I1}
					H'(r,v)= \frac{4}{r}I(r,v)+\frac{4}{r^{3}}\int_{S_{r}^{+}}|t|vfG+\frac{4}{r^{3}}\int_{S_{r}^{'}}|t|v\partial_{n+1}vG=\frac4 r\Big( I(r,v)+F(r,v)\Big)=\frac{2}{r^{3}}\int_{S_{r}^{+}}v(Zv)G,
				\end{equation}
				\begin{equation*}\label{eq:derivatives-H-and-I2}
					I'(r,v)\ge \frac{1}{r^{3}}\int_{S_{r}^{+}}(Zv)^{2}G-\frac{2}{r^{3}}\int_{S_{r}^{+}}|t|(Zv)fG-\frac{4}{r^3}\int_{S_r'}|t|v\partial_{n+1}vG.
				\end{equation*}
				From these, adding the regularization terms, we deduce
				\begin{equation}\label{eq:derivatives-H-and-I-tilde1}
					\widetilde{H}'(r,v)= \frac{4}{r}\left(\widetilde{I}(r,v)+F(r,v)\right)=\frac{2}{r^{3}}\int_{S_{r}^{+}}v(Zv)G+2\beta r^{2\beta-1},
				\end{equation}
				\begin{equation}\label{eq:derivatives-H-and-I-tilde2}
					\widetilde{I}'(r,v)\ge \frac{1}{r^{3}}\int_{S_{r}^{+}}(Zv)^{2}G-\frac{2}{r^{3}}\int_{S_{r}^{+}}|t|(Zv)fG-\frac{4}{r^3}\int_{S_r'}|t|v\partial_{n+1}vG+\beta^{2}r^{2\beta-1}.
				\end{equation}
				
				Finally, we recall the following facts for the quantity $\mathcal{W}_{2}$ (see \cite[Chapter 13]{dgpt17}). Firstly, since $p\in \mathcal{P}_{2}^{+}$, integration by parts yields
				\begin{equation*}
					\mathcal{W}_{2}(r,v)=\mathcal{W}_{2}(r,w).
				\end{equation*}
				Secondly, since $(0,0)\in \Gamma_{2}$, then $r\mapsto \mathcal{W}_{2}(r,w)+Cr^{2(\ell-2)}$ is non-decreasing, and $\mathcal{W}_{2}(0^{+}, w)=0$. 
				
				\smallskip
				\noindent \textit{Step 2: Computation of $\widetilde{N}'(r,v)$.}  Differentiating $\widetilde{N}(r,v)$ in $r$, and using the last equality in \eqref{eq:derivatives-H-and-I1}, we get
				\begin{align*}
					\frac{1}{2}\widetilde{H}(r,v)^{2}\widetilde{N}'(r,v)&=\widetilde{I}'(r,v)\widetilde{H}(r,v)-\widetilde{H}'(r,v)\widetilde{I}(r,v)\\
					&=\widetilde{I}'(r,v)\widetilde{H}(r,v)-\widetilde{H}'(r,v)\left(\frac1{2r^2}\int_{S_r^+}v(Zv)G+\frac{\beta}{2}{r^{2\beta}}\right)+\widetilde{H}'(r,v)F(r,v).
				\end{align*}
				Then, by \eqref{eq:derivatives-H-and-I-tilde1}, \eqref{eq:derivatives-H-and-I-tilde2} and the Cauchy-Schwarz inequality, we obtain
				\begin{equation}\label{eq:main-cancellation-mon-formula}
					\begin{aligned}
						\frac{1}{2}\widetilde{H}(r,v)^{2}\widetilde{N}'(r,v)&\ge \left(\frac{1}{r^{3}}\int_{S_{r}^{+}}(Zv)^{2}G+\beta^{2}r^{2\beta-1}\right)\left(\frac{1}{r^{2}}\int_{S_{r}^{+}}v^{2}G+r^{2\beta}\right)\\
						&\qquad-\left(\frac{2}{r^{3}}\int_{S_{r}^{+}}v(Zv)G+2\beta r^{2\beta-1}\right)\left(\frac{1}{2r^2}\int_{S_{r}^{+}}v(Zv)G+\frac{\beta}{2}r^{2\beta}\right)\\
						&\qquad\qquad-\left(\frac{2}{r^{3}}\int_{S_{r}^{+}}|t|(Zv)fG+\frac{4}{r^3}\int_{S_r'}|t|v\partial_{n+1}vG\right)\widetilde{H}(r,v)\\
						&\qquad\qquad\qquad +F(r,v)\left(\frac{4}{r}\widetilde{I}(r,v)+\frac{4}{r}F(r,v)\right) \\   
						&\ge \frac{4}{r}F(r,v)^{2}+\frac{2}{r}\left(2F(r,v)\widetilde{I}(r,v)-A(r,v)\widetilde{H}(r,v)\right),
					\end{aligned}
				\end{equation}
				where we defined
				$$A(r,v):=\frac{1}{r^2}\int_{S_r^+}|t|(Zv)fG+\frac{2}{r^2}\int_{S_r'}|t|v\partial_{n+1}vG.$$
				
				Since $Zp=2p$, we have $Zv-2v=Zw-2w$. Therefore, by \eqref{eq:def-F} and $\mathbf{H}v=f$ in $S_1^+$,
				we may rewrite
				$$A(r,v)= 2F(r,v)+\frac{1}{r^{2}}\int_{S_{r}^{+}}|t|(Zw-2w)fG.$$
				Inserting this equation in \eqref{eq:main-cancellation-mon-formula} we obtain
				\begin{equation}\label{eq:remainders-mon-formula}
					\begin{aligned}
						\widetilde{N}'(r,v)&\ge \frac{8}{r}\frac{F(r,v)^{2}}{\widetilde {H}(r,v)^{2}}+\frac{8}{r}\frac{F(r,v)(\widetilde{I}(r,v)-\widetilde{H}(r,v))}{\widetilde{H}(r,v)^{2}}-\frac{4}{r^{3}}\frac{1}{\widetilde{H}(r,v)}\int_{S_{r}^{+}}|t|(Zw-2w)fG\\
						&=: \frac{8}{r}\frac{F(r,v)^{2}}{\widetilde {H}(r,v)^{2}}+(I)+(II).
					\end{aligned}
				\end{equation}
				
				\smallskip
				\noindent \textit{Step 3: Estimates of the remainder terms.} We now bound from below the two remainder terms $(I)$, $(II)$ from \eqref{eq:remainders-mon-formula}.
				
				To bound $(I)$ we consider separately the two cases $F(r,v)\ge 0$ and $F(r,v)<0$. Suppose first that $F(r,v)\ge 0$.
				By the considerations made in \textit{Step 1}, we have $\mathcal{W}_{2}(r,v)=\mathcal{W}_2(r,w)\ge -Cr^{2(\ell-2)}$. Using also $\beta\ge 2$ we get
				$$\widetilde{I}(r,v)-\widetilde{H}(r,v)=I(r,v)-H(r,v)+\left(\frac{\beta}{2}-1\right)r^{2\beta}\ge r^{4}\mathcal{W}_{2}(r,v)\ge -Cr^{2\ell}.$$ Young's inequality and $\widetilde{H}(r,v)\ge r^{2\beta}$ then give
				$$(I)=\frac{8}{r}\frac{F(r,v)(\widetilde{I}(r,v)-\widetilde{H}(r,v))}{\widetilde{H}(r,v)^{2}}\ge -\frac{C}{r}\frac{F(r,v)}{\widetilde{H}(r,v)}\frac{r^{2\ell}}{\widetilde{H}(r,v)}\ge -\frac{7}{r}\frac{F(r,v)^{2}}{\widetilde{H}(r,v)^{2}}-Cr^{4(\ell-\beta)-1}.$$
				Suppose now that $F(r,v)<0$. By \eqref{eq:bound-below-Fv} we have 
				$$(I)= \frac{4}{r}\frac{F(r,v)}{\widetilde{H}(r,v)}\left(\widetilde{N}(r,v)-2\right)\ge -\frac{C}{r}\frac{r^{\ell}}{\widetilde{H}(r,v)^{1/2}}\widetilde{N}(r,v)\ge -Cr^{\ell-\beta-1}\widetilde{N}(r,v),$$
				where we used $\widetilde{H}(r,v)\ge r^{2\beta}$ again. In all,
				\begin{equation}\label{eq:bound-first-remainder-mon-formula}
					(I)\ge -Cr^{\ell-\beta-1}\widetilde{N}(r,v)-\frac{7}{r}\frac{F(r,v)^{2}}{\widetilde{H}(r,v)^{2}}-Cr^{4(\ell-\beta)-1}.
				\end{equation}
				
				Next, we estimate $(II)$. The Gaussian Sobolev estimates contained in \cite[Section 5]{dgpt17}, the growth estimate in \cite[Lemma 13.3]{dgpt17}, and $\ell\ge 4$ give
				\begin{equation}\label{eq:bound-Gaussian-Sobolev-mon-formula}
					\frac{1}{r^2}\int_{S_{r}^{+}}(Zw-2w)^{2}G\le C\left(H(2r,w)+r^{2\ell-4}\right)\le Cr^{4}.
				\end{equation}
				Then, the Cauchy-Schwarz inequality combined with \eqref{eq:bound-integral-f-rhs}, \eqref{eq:bound-Gaussian-Sobolev-mon-formula} and $\widetilde{H}(r,v)\ge r^{2\beta}$ yields
				\begin{equation}\label{eq:bound-second-remainder-mon-formula}
					(II)\ge -\frac{C}{r^{2}}\frac{1}{\widetilde{H}(r,v)}\left(\frac{1}{r^{2}}\int_{S_{r}^{+}}(Zw-2w)^{2}G\right)^{1/2}\left(\int_{S_{r}^{+}}t^{2}f^{2}G\right)^{1/2}\ge -Cr^{\ell+1-2\beta}.
				\end{equation}
				In conclusion, putting \eqref{eq:remainders-mon-formula} together with \eqref{eq:bound-first-remainder-mon-formula} and \eqref{eq:bound-second-remainder-mon-formula} we get
				\begin{equation}\label{eq:mon-formula-differential-form}
					\begin{aligned}
						\widetilde{N}'(r,v)&\ge \frac{1}{r}\frac{F(r,v)^{2}}{\widetilde{H}(r,v)^{2}}-Cr^{\ell-\beta-1}\widetilde{N}(r,v)-Cr^{4(\ell-\beta)-1}-Cr^{\ell+1-2\beta}\\
						&\ge \frac{1}{r}\frac{F(r,v)^{2}}{\widetilde{H}(r,v)^{2}}-Cr^{\ell-\beta-1}\widetilde{N}(r,v)-Cr^{\ell+1-2\beta}.
					\end{aligned}    
				\end{equation}
				Note that the absorption in the second inequality is legitimate, since under our hypotheses $2\beta\le\ell+2\le 3\ell-2$ and so $4(\ell-\beta)-1\ge \ell+1-2\beta$. Integrating \eqref{eq:mon-formula-differential-form} we get the desired monotonicity formula. 
			\end{proof}
			
			As a consequence of the monotonicity formula proved in \cref{prop:mon-formula-second-bu}, we obtain the following growth estimates for the second blow-up analysis.
			From now on, we fix a number $\beta$ that satisfies the assumptions 
			\begin{equation*}
				\beta>2,\qquad \ell+2-2\beta>0.
			\end{equation*}
			\begin{lemma}\label{lemma:asymptotic-second-bu}
				Let $u$ be a solution of \eqref{eq:parabolic-signorini-problem}, with obstacle $\varphi$ satisfying \eqref{eq:ostacolo}.
				Suppose that $(0,0)\in \Gamma_{2}$, and let $p\in \mathcal{P}_{2}^+$. We denote by $v:=w-p$, where $w:=\zeta\widetilde u^{0,0}$ and $\widetilde u^{0,0} $ are defined in \eqref{eq:utilde}, \eqref{eq:subtraction-obstacle+cutoff}. There is $r_{0}\in (0,1)$ for which the following hold.
				\begin{itemize}
					\item [i)] Suppose that $0<r<R<r_{0}$ are such that $\underline{\lambda}\le \widetilde{N}(\rho,v)\le \overline{\lambda}$ for every $\rho \in [r,R]$. Then, for every $\eps>0$  
					\begin{equation*}
						c\left(\frac{R}{r}\right)^{2\underline{\lambda}}\le \left(\frac{\widetilde{H}(R,v)}{\widetilde{H}(r,v)}\right)\le C_{\eps}\left(\frac{R}{r}\right)^{2\overline{\lambda}+\eps},
					\end{equation*} where $c>0$ and $C_\eps>0$ are independent of $r$ and $R$.

                    \smallskip
					\item [ii)] For any $\underline{\lambda}\le \widetilde{N}(0^{+},v)$, the following Monneau-type monotonicity formula holds
					\begin{equation*}
						\frac{d}{dr}\left(\frac{\widetilde{H}(r,v)}{r^{2\underline{\lambda}}}e^{\overline{C}r^{\ell+2-2\beta}}\right)\ge 0\qquad \text{for a.e. }r\in (0,r_{0}),
					\end{equation*}where $\overline C>0$ is independent of $r$.

                    \smallskip
					\item [iii)] We have $\widetilde{N}(0^{+},v)\in [2,\beta]$. Moreover, if $\widetilde{N}(0^{+},v)<\beta$, then $N(0^{+}, v)=\widetilde{N}(0^{+},v)$.
				\end{itemize}
			\end{lemma}
			\begin{proof}
				We first prove point $i)$, starting from the first inequality. By \eqref{eq:bound-below-Fv} and \eqref{eq:derivatives-H-and-I-tilde1} we get
				\begin{align*}
					\frac{d}{d\rho}\log\left(\widetilde{H}(\rho,v)\right)= \frac{2}{\rho}\widetilde{N}(\rho,v)+\frac{4}{\rho}\frac{F(\rho,v)}{\widetilde{H}(\rho,v)}\ge \frac{2\underline{\lambda}}{\rho}-C\rho^{\ell-1}\frac{H(\rho,v)^{1/2}}{\widetilde{H}(\rho,v)}\ge  \frac{2\underline{\lambda}}{\rho}-C\rho^{\ell-1-\beta}\qquad \forall \rho\in (r,R).
				\end{align*}
				Integration and $\ell>\beta$ yield
				$$\log\left(\frac{\widetilde{H}(R,v)}{\widetilde{H}(r,v)}\right)\ge 2\underline{\lambda}\log\left(\frac{R}{r}\right)-\frac{C}{\ell-\beta}(R^{\ell-\beta}-r^{\ell-\beta})\ge 2\underline{\lambda}\log\left(\frac{R}{r}\right)-C.$$
				Exponentiating we get the desired lower bound.
				
				To obtain the second inequality, we  first note that by the Cauchy-Schwarz inequality and the monotonicity formula \eqref{eq:mon-formula-second-bu},
				\begin{align*}
					\int_{r}^{R}\frac{1}{\rho}\frac{F(\rho,v)}{\widetilde{H}(\rho,v)}d\rho&\le \left(\int_{r}^{R}\frac{1}{\rho}\frac{F(\rho,v)^{2}}{\widetilde{H}(\rho,v)^{2}}d\rho\right)^{1/2}\left(\int_{r}^{R}\frac{1}{\rho}d\rho\right)^{1/2}\\
					&\le \left(\widehat{N}(R,v)-\widehat{N}(r,v)\right)^{1/2}\log\left(\frac{R}{r}\right)^{1/2}\le C\log\left(\frac{R}{r}\right)^{1/2}.
				\end{align*}
				Therefore, as before, since
				\begin{equation*}
					\frac{d}{d\rho}\log\left(\widetilde{H}(\rho,v)\right)= \frac{2}{\rho}\widetilde{N}(\rho,v)+\frac{4}{\rho}\frac{F(\rho,v)}{\widetilde{H}(\rho,v)}\le \frac{2\overline{\lambda}}{\rho}+\frac{4}{\rho}\frac{F(\rho,v)}{\widetilde{H}(\rho,v)}\qquad \forall \rho \in (r,R),
				\end{equation*}
				after integration we get, for any $\eps>0$,
				$$\log\left(\frac{\widetilde{H}(R,v)}{\widetilde{H}(r,v)}\right)\le 2\overline{\lambda}\log\left(\frac{R}{r}\right)+C\log\left(\frac{R}{r}\right)^{1/2}\le (2\overline{\lambda}+\eps)\log\left(\frac{R}{r}\right)+C_{\eps},$$
				where $C_{\eps}>0$ is some constant depending on $\eps$. The desired upper bound is obtained after exponentiation.
				
				Next, we prove point $ii)$. By \Cref{prop:mon-formula-second-bu}, we get $\widehat{N}(r,v)\ge \widehat{N}(0^{+}, v)= \widetilde{N}(0^{+}, v)\ge \underline{\lambda}$ which rewrites as
				\begin{equation}\label{eq:bound-below-frequency-proof-monneau}
					\widetilde{N}(r,v)-\underline{\lambda}\ge \underline{\lambda}\left(e^{-Cr^{\ell-\beta}}-1\right)-Cr^{\ell+2-2\beta}\ge -C\left(r^{\ell-\beta}+r^{\ell+2-2\beta}\right)\ge -Cr^{\ell+2-2\beta},
				\end{equation}
				where in the last step we used the assumption $\beta> 2$. On the other hand, using \eqref{eq:bound-below-Fv} we get
				\begin{equation}\label{eq:bound-below-second-remainder-proof-monneau}
					\frac{4}{r}\frac{F(r,v)}{\widetilde{H}(r,v)}\ge -Cr^{\ell-1}\frac{H(r,v)^{1/2}}{\widetilde{H}(r,v)}\ge -Cr^{\ell-1-\beta}.
				\end{equation}
				Therefore, combining \eqref{eq:bound-below-frequency-proof-monneau} and \eqref{eq:bound-below-second-remainder-proof-monneau} we obtain
				\begin{align*}
					\frac{d}{dr}\left(\frac{\widetilde{H}(r,v)}{r^{2\underline{\lambda}}}e^{\overline{C}r^{\ell+2-2\beta}}\right)&=\frac{\widetilde{H}(r,v)}{r^{2\underline{\lambda}}}e^{\overline{C}r^{\ell+2-2\beta}}\left(\frac{\widetilde{H}'(r,v)}{\widetilde{H}(r,v)}-\frac{2\underline \lambda}{r}+C_{1}r^{\ell+1-2\beta}\right)\\
					&=\frac{\widetilde{H}(r,v)}{r^{2\underline{\lambda}}}e^{\overline{C}r^{\ell+2-2\beta}}\left(\frac{2}{r}(\widetilde{N}(r,v)-\underline \lambda)+\frac{4}{r}\frac{F(r,v)}{\widetilde{H}(r,v)}+C_{1}r^{\ell+1-2\beta}\right)\\
					&\ge \frac{\widetilde{H}(r,v)}{r^{2\underline{\lambda}}}e^{\overline{C}r^{\ell+2-2\beta}}\left((C_{1}-C)r^{\ell+1-2\beta}-Cr^{\ell-1-\beta}\right),
				\end{align*} where $C_{1}=(\ell+2-2\beta) \overline{C}$. If now we choose $\overline{C}$ sufficiently large, the right-hand side above is nonnegative. This concludes the proof of point $ii)$.
				
				Finally, we prove point $iii)$. Since both $w$ and $p$ have quadratic growth at the origin, we must have 
				\begin{equation}\label{eq:naive-bounds-Htildev}
					r^{2\beta}\le \widetilde{H}(r,v)\le Cr^{4}.
				\end{equation}
				Suppose by contradiction that $\widetilde{N}(0^{+}, v)<2$. Then we can find $R>0$, $\overline{\lambda}>0$ and $\eps>0$ such that $2\overline{\lambda}+\eps<4$ and $\widetilde{N}(\rho,v)\le \overline{\lambda}$ for all $\rho \in (0,R)$. The second inequality from point $i)$ then gives 
				\begin{equation*}
					\widetilde{H}(r,v)\ge c r^{2\overline{\lambda}+\eps}\qquad \forall r\in (0,R),
				\end{equation*}
				which contradicts the upper bound in \eqref{eq:naive-bounds-Htildev}. Similarly, if $\widetilde{N}(0^{+},v)>\beta$, we could choose $R>0$ and $\underline{\lambda}>\beta$ such that $\widetilde{N}(\rho,v)\ge \underline{\lambda}$ for all $\rho \in (0,R)$. The first inequality in point $i)$ would give
				\begin{equation*}
					\widetilde{H}(r,v)\le  C r^{2\underline{\lambda}}\qquad \forall r\in (0,R),
				\end{equation*}
				contradicting the lower bound in \eqref{eq:naive-bounds-Htildev}. This proves $\widetilde{N}(0^{+},v)\in [2,\beta]$.
				
				Suppose now that $\widetilde{N}(0^{+},v)<\beta$. Choosing $R>0$, $\overline{\lambda}$ and $\eps>0$ such that $2\overline{\lambda}+\eps<2\beta$ and $\widetilde{N}(\rho,v)\le \overline{\lambda}$ for all $\rho \in (0,R)$, by point $i)$ we get 
				\begin{equation*}
					\frac{r^{2\beta}}{\widetilde{H}(r,v)}\le C r^{2\beta-2\overline{\lambda}-\eps}\qquad \forall r\in (0,R),
				\end{equation*}
				where the right-hand side goes to zero as $r\downarrow 0$. Since $\widetilde{H}(r,v)=H(r,v)+r^{2\beta}$, we deduce that $\widetilde{H}(r,v)/H(r,v)$ converges to $ 1$ as $r\downarrow 0$. In particular, the second term in  
				\begin{equation*}
					N(r,v)= \widetilde{N}(r,v)-(\beta-\widetilde{N}(r,v))\frac{r^{2\beta}}{H(r,v)}
				\end{equation*}
				vanishes as $r\downarrow 0$, and we conclude $N(0^{+}, v)=\widetilde{N}(0^{+},v)$.
			\end{proof}
			
			\begin{remark}\label{rmk:doubling}
				Under the assumptions of \cref{lemma:asymptotic-second-bu}, in the proof above we obtained that if $\widetilde{N}(0^{+},v)<\beta$, then $\widetilde{H}(r,v)/H(r,v)\to 1$ as $r\downarrow 0$. One can then apply point $i)$ of \cref{lemma:asymptotic-second-bu} to deduce the doubling property
				\be\label{lemmafinal2}C^{-1}\le\frac{H(2r, v)}{H(r, v)}\le C\qquad \forall r\in (0,r_{0}).\ee
			\end{remark}

			\subsection{The second blow-up analysis}\label{subsection-second-blow-up}
			In this section, we prove existence and characterization of second blow-ups at quadratic points. To this end, we need a uniform Sobolev bound on the sequence of rescalings of $v:=w-p$. We then prove the following two lemmas.
			\begin{lemma}\label{lemma:uniform-bounds-second-blow-up}
				Let $u$ be a solution of \eqref{eq:parabolic-signorini-problem}, with obstacle $\varphi$ satisfying \eqref{eq:ostacolo}. Suppose that $(0,0)\in \Gamma_{2}$, and let $p\in \mathcal{P}_{2}^+$. We denote by $v:=w-p$, where $w:=\zeta\widetilde u^{0,0}$ and $\widetilde u^{0,0} $ are defined in \eqref{eq:utilde}, \eqref{eq:subtraction-obstacle+cutoff}.
				Then $$\|v\|_{L^{\infty}(Q_{1/2})}+\|\nabla v\|_{L^2(Q_{1/2})}+\|\partial_t v\|_{L^\infty(Q_{1/2})}\le C\left(\|v\|_{L^2(Q_1)}+\|f\|_{L^{\infty}(Q_1)}\right),$$ for some $C>0$ depending on $n$.  
			\end{lemma}
			\begin{proof}
				We divide the proof into three steps. 
				
				\smallskip
				\noindent \textit{Step 1: Estimate of \,$\lVert v\rVert_{L^{\infty}(Q_{1/2})}$.} Extending $w$ and $f$ evenly with respect to $x_{n+1}$, from $\mathbf{H}p=0$ and \eqref{eq:w-solves} we get
				\begin{equation}\label{eq:formula-Hv}
					\mathbf{H}v=\mathbf{H}w=f-\mu, \qquad \mu:= -2\partial_{n+1}w \mathcal{H}^{n+1}\res S_{1}',
				\end{equation}
				where $\mu\ge 0$ is a Radon measure concentrated on $S_{1}'\cap \{w=0\}$ due to the Signorini conditions. 
				In particular,
				$$\mathbf{H}v=f\ge -\|f\|_{L^\infty(Q_1)}\quad\text{in }Q_1\setminus\left(S_{1}'\cap\{w=0\}\right)\qquad\text{and}\qquad v=-p\le0\quad\text{on }S_{1}'\cap\{w=0\}.$$
				As a consequence, since $v^{+}=\max\{v,0\}$, we deduce $\mathbf{H}v^+\ge -\|f\|_{L^\infty(Q_1)}$ in the viscosity sense. Similarly, because $\mathbf{H}(-v)\ge -f\ge -\|f\|_{L^\infty(Q_1)}$ in $Q_{1}$ and $v^{-}=\max\{-v,0\}$, we obtain $\mathbf{H}v^-\ge -\|f\|_{L^\infty(Q_1)}$. Then, the $L^2$-$L^\infty$ estimate for subsolutions from \cref{lemma:wang-start} gives $$\|v\|_{L^\infty(Q_{1/2})}\le\|v^+\|_{L^\infty(Q_{1/2})}+\|v^-\|_{L^\infty(Q_{1/2})}\le
				C\left(\|v\|_{L^2(Q_{1})}+\|f\|_{L^\infty(Q_{1})}\right).
				$$
				
				\smallskip
				\noindent \textit{Step 2: Estimate of \,$\lVert \nabla v\rVert_{L^{2}(Q_{1/2})}$.} Since on the hyperplane $p$ is nonnegative and $\min\{-\partial_{n+1}w,w\}=0$, by \eqref{eq:formula-Hv} we have
				\begin{equation*}
					v\mathbf{H}v= vf+p\mu\ge vf\qquad \text{in $Q_{1}$}.
				\end{equation*}
				The bound on the gradient follows from a Caccioppoli-type inequality. Take $\psi\in C^{\infty}_{c}(Q_{1}, [0,1])$ such that $\psi\equiv 1$ on $Q_{1/2}$. Then, integrating by parts,
				\begin{align*}
					\int_{Q_{1}}\psi^{2}vf\le \int_{Q_{1}}\psi^{2}v\mathbf{H}v
					&=  -\int_{Q_{1}}\psi^{2}|\nabla v|^{2}-2\int_{Q_{1}}\psi v\nabla \psi \cdot \nabla v -\frac{1}{2}\int_{B_{1}}\psi(\cdot, 0)^{2}v(\cdot, 0)^{2}+\frac{1}{2}\int_{Q_{1}}v^{2}\partial_{t}\psi^{2}\\
					&\le -\frac{1}{2}\int_{Q_{1}}\psi^{2}|\nabla v|^{2}+ C\int_{Q_{1}}v^{2},
				\end{align*}
				where in the last step we used Young's inequality. Rearranging terms we get $$\|\nabla v\|_{L^2(Q_{1/2})}\le C\left(\| v\|_{L^2(Q_{1})}+\| f\|_{L^{\infty}(Q_{1})}\right).$$
				
				\smallskip
				\noindent \textit{Step 3: Estimate of \,$\lVert\partial_{t}v\rVert_{L^{\infty}(Q_{1/2})}$.} 
				By the monotonicity assumption $\partial_{t}u\ge 0$ and the fact that $p$ is time-independent we deduce $\partial_{t}v=\partial_{t}w\ge 0$. We claim that 
				\be\label{eq:time-derivative-subcaloric}\mathbf{H}(\partial_{t}w)\ge 0\qquad\text{in }Q_{3/4}.\ee
				Given \eqref{eq:time-derivative-subcaloric}, we can get the desired bound by combining the parabolic Krylov-Safonov $L^{1}$-$L^{\infty}$ estimate in \cref{lemma:wang-start} with \textit{Step 1}:
				\bea\begin{aligned}\|\partial_t v\|_{L^\infty(Q_{1/2})}&\le C\|\partial_t v\|_{L^{1}(Q_{2/3})}=C\int_{B_{2/3}}(v(\cdot, 0)-v(\cdot, -4/9))\le C\|v\|_{L^\infty(Q_{3/4})}\\&\le C\left(\|v\|_{L^2(Q_{1})}+\|f\|_{L^{\infty}(Q_{1})}\right).
				\end{aligned}\eea
				It only remains to prove \eqref{eq:time-derivative-subcaloric}. For sufficiently small $h>0$, define the backward difference quotient
				\begin{equation*}
					w_h(x,t):=\frac{w(x,t)-w(x,t-h)}h.
				\end{equation*}
				It is defined throughout $Q_{3/4}$. Since $f$ is independent of time in this cylinder, $\mathbf{H} w_h=0$ in $Q_{3/4}^+$. On the thin space, if $w(x',0,t)>0$, the Signorini conditions give
				\begin{equation*}
					\partial_{n+1}w_h(x',0,t)=-\frac{\partial_{n+1}w(x',0,t-h)}h\ge 0.
				\end{equation*}
				If instead $w(x',0,t)=0$, time monotonicity and nonnegativity of the thin trace give $w(x',0,t-h)=0$. Thus $w_h(x',0,t)=0$, and $w_h\ge 0$ in the upper half cylinder implies $\partial_{n+1}w_h(x',0,t)\ge 0$ in this case as well. The even extension therefore satisfies $\mathbf{H} w_h\geq0$ in $Q_{3/4}$. Passing to the limit as $h\downarrow 0$ proves the claim. This concludes the proof.
			\end{proof}
			\begin{lemma}\label{lemma:controlvconH}
				Let $u$ be a solution of \eqref{eq:parabolic-signorini-problem}, with obstacle $\varphi$ satisfying \eqref{eq:ostacolo}. Suppose that $(0,0)\in \Gamma_{2}$, and let $p\in \mathcal{P}_{2}^+$. We denote by $v:=w-p$, where $w:=\zeta\widetilde u^{0,0}$ and $\widetilde u^{0,0} $ are defined in \eqref{eq:utilde}, \eqref{eq:subtraction-obstacle+cutoff}. Assume that $\widetilde{N}(0^{+},v)<\beta$. Then, there are $C>0$ and $r_{0}\in (0,1)$ such that
				$$\frac{1}{C}H(r, v)^{1/2}\le \|v(r\cdot,r^2\cdot)\|_{L^2(Q_1)}\le CH(r, v)^{1/2}\qquad \forall r\in (0,r_{0}).$$
			\end{lemma}
			\begin{proof}
				Taking \cref{lemma:asymptotic-second-bu} and \cref{rmk:doubling} into account, we obtain that there is $r_{0}\in (0,1)$ sufficiently small such that 
				\begin{equation}\label{eq:unif-bounds-HvsL2}
					N(r,v)\le M\quad \text{and}\quad \frac{H(s,v)}{H(r,v)}\le C\left(\frac{s}{r}\right)^{4}\qquad \forall 0<s<r<r_{0}.
				\end{equation}
				Calling $v_{r}(x,t):= v(rx,r^{2}t)/H(r,v)^{1/2}$, we have 
				\begin{equation}\label{eq:normalization-HvsL2}
					H(1,v_{r})=1\quad \text{and}\quad I(1,v_{r})=\frac{N(r,v)}{2}\le \frac{M}{2}\qquad \forall r\in (0,r_{0}).
				\end{equation}
				Up to further reducing $r_{0}$, we wish to prove the following uniform two-sided bounds:
				\begin{equation*}
					\frac{1}{C}\le \lVert v_{r}\rVert_{L^{2}(Q_{1})}\le C\qquad \forall r\in (0,r_{0}).
				\end{equation*}
				
				\smallskip
				\noindent \textit{Step 1: The lower bound.} We know that the square integral of $v_{r}$ over $S_{1}^{+}$ with Gaussian weight is normalized. To prove a lower bound for $\lVert v_{r}\rVert_{L^{2}(Q_{1})}$ we need to show that a given fraction of such mass lives in a region where $G$ is bounded above. We start by proving the following uniform Gaussian bound:
				\begin{equation*}
					\int_{S_{1}^{+}}|x|^{2}v_{r}^{2}G\le C \qquad \forall r\in (0,r_{0}).
				\end{equation*}
				Using $\nabla G(x,t)= \frac{x}{2t} G(x,t)$ we find 
				\begin{equation*}
					\diver (x v_{r}^{2}G)=(n+1)v_{r}^{2}G+2v_{r}x\cdot \nabla v_{r}G+ v_{r}^{2}\frac{|x|^{2}}{2t}G. 
				\end{equation*}
				Integrating over $\R^{n+1}_{+}$, for a fixed $t<0$ we obtain 
				\begin{equation*}
					\int_{\R^{n+1}_{+}\times \{t\}}|x|^{2}v_{r}^{2}G= 2(n+1)|t|\int_{\R^{n+1}_{+}\times \{t\}}v_{r}^{2}G+ 4|t|\int_{\R^{n+1}_{+}\times \{t\}}v_{r} x\cdot \nabla v_{r}G.
				\end{equation*}
				Using Young's inequality on the second term in the right-hand side above and absorbing it into the left-hand side, we find
				\begin{equation*}
					\int_{\R^{n+1}_{+}\times \{t\}}|x|^{2}v_{r}^{2}G\le 4(n+1)|t|\int_{\R^{n+1}_{+}\times \{t\}}v_{r}^{2}G+16 t^{2}\int_{\R^{n+1}_{+}\times \{t\}}|\nabla v_{r}|^{2}G.
				\end{equation*}
				Integrating this bound for $t\in (-1,0)$ and using \eqref{eq:normalization-HvsL2} we get the desired upper bound
				\begin{equation*}
					\int_{S_{1}^{+}}|x|^{2}v_{r}^{2}G\le 4(n+1)+8M.
				\end{equation*}
				As a consequence, for any $R>0$ we have
				\begin{equation}\label{eq:compactness-x-HvsL2}
					\int_{S_{1}^{+}\cap \{|x|>R\}}v_{r}^{2}G\le \frac{C}{R^{2}}\qquad \forall r\in (0,r_{0}).
				\end{equation}
				
				Next, we show that the mass of $v_{r}^{2}$ cannot concentrate at small times. Indeed, by the second inequality in \eqref{eq:unif-bounds-HvsL2}, for any $\eta\in (0,1)$ we have 
				\begin{equation}\label{eq:compactness-t-HvsL2}
					\int_{S_{\eta}^{+}}v_{r}^{2}G= \eta^{2}\frac{H(\eta r,v)}{H(r,v)}\le C\eta^{6}\qquad \forall r\in (0,r_{0}).
				\end{equation}
				
				Now, since $H(1,v_{r})=1$, by \eqref{eq:compactness-x-HvsL2} and \eqref{eq:compactness-t-HvsL2}, we may choose $R$ sufficiently large and $\eta$ sufficiently small so that 
				\begin{equation*}
					\frac{1}{2}\le \int_{B_{R}\times (-1,-\eta^{2})}v_{r}^{2}G \le C\int_{Q_{R}}v_{r}^{2}\qquad \forall r\in (0,r_{0}),
				\end{equation*}
				where in the second step we used the fact that $G$ is uniformly bounded above in $B_{R}\times (-1,-\eta^{2})$. Finally, rescaling and using the doubling property from \eqref{lemmafinal2} proves $\lVert v_{r}\rVert_{L^{2}(Q_{1})}\ge c>0$ for all $r\in (0,r_{0})$, up to further decreasing $r_{0}$.
				
				\smallskip
				\noindent \textit{Step 2: The upper bound.} First of all, applying the $L^{\infty}$-bound from \cref{lemma:uniform-bounds-second-blow-up} to $v(2r\cdot,4r^{2}\cdot)$ and using the decay estimate for the right-hand side $f$ from \eqref{controlf} we deduce 
				\begin{equation}\label{eq:L2-Linfty-HvsL2}
					\lVert v\rVert_{L^{\infty}(Q_{r})}\le C \lVert v(2r\cdot, 4r^{2}\cdot)\rVert_{L^{2}(Q_{1})}+Cr^{\ell}\qquad \forall r\in (0,r_{0}).
				\end{equation}
				In particular, given a small number $\theta\in (0,1)$ to be chosen later, we may split the $L^{2}$-norm at scale $r$ into two pieces:
				\begin{align*}
					\lVert v(r\cdot,r^{2}\cdot)\rVert_{L^{2}(Q_{1})}^{2}&=\int_{B_{1}\times (-1,-\theta)}v(r\cdot,r^{2}\cdot)^{2}+\int_{B_{1}\times (-\theta,0)}v(r\cdot,r^{2}\cdot)^{2}\\
					&\le C_\theta\int_{B_{1}\times (-1,-\theta)}v(r\cdot,r^{2}\cdot)^{2}G+C\theta\lVert v\rVert_{L^{\infty}(Q_{r})}^{2}\\
					&\le C_\theta H(r,v)+C\theta  \lVert v(2r\cdot, 4r^{2}\cdot)\rVert_{L^{2}(Q_{1})}^{2}+C\theta r^{2\ell},
				\end{align*}
				where in the second step we used the fact that the heat kernel $G$ is uniformly bounded below in $B_{1}\times (-1,-\theta)$, while in the third we applied \eqref{eq:L2-Linfty-HvsL2}. Now, dividing both sides of the last inequality by $H(r,v)$, using $H(r,v)\ge cr^{2\beta}$, $\beta<\ell$, and the doubling property from \eqref{lemmafinal2} we obtain the recurrence
				\begin{equation*}
					\lVert v_{r}\rVert_{L^{2}(Q_{1})}^{2}\le C_\theta+C\theta \lVert v_{2r}\rVert_{L^{2}(Q_{1})}^{2}\qquad \forall r\in (0,r_{0}).
				\end{equation*}
				Finally, choosing $\theta$ so small that $C\theta \le 1/2$ and solving the recurrence we find
				\begin{equation*}
					\lVert v_{r}\rVert_{L^{2}(Q_{1})}^{2}\le C\left(1+\sup_{\rho\in (r_{0}/2, r_{0})}\lVert v_{\rho}\rVert_{L^{2}(Q_{1})}^{2}\right)\le C\qquad \forall r\in (0,r_{0}/2),
				\end{equation*}
				which gives the desired upper bound.  
			\end{proof}
			We now have all the ingredients to prove the existence of second blow-ups at quadratic points. We also show that any such second blow-up must be a caloric quadratic polynomial.
			\begin{proposition}\label{prop:second-blow-up}
				Let $u$ be a solution of \eqref{eq:parabolic-signorini-problem}, with obstacle $\varphi$ satisfying \eqref{eq:ostacolo}. Suppose that $(0,0)\in \Gamma_{2}$, and let $p\in \mathcal{P}_{2}^+$ be the blow-up of $u$ at $(0,0)$. We denote by $v:=w-p$, where $w:=\zeta\widetilde u^{0,0}$ and $\widetilde u^{0,0} $ are defined in \eqref{eq:utilde}, \eqref{eq:subtraction-obstacle+cutoff}. 
				Consider the rescalings
				$$v_{r}(x,t):= \frac{v(rx,r^{2}t)}{H(r,v)^{1/2}}\qquad \forall r\in (0,1).$$
				Then, for every sequence $r_{j}\downarrow 0$, there is a function $q:S_{\infty}\to \R$ such that, up to a subsequence, 
				$$v_{r_{j}}\to q,\quad \nabla v_{r_{j}}\rightharpoonup \nabla q,\quad\partial_t v_{r_{j}}\rightharpoonup \partial_t q\qquad\text{in }L^2_{\text{loc}}(S_{\infty}).$$ 
				Moreover, $q\not\equiv0$ is a parabolically $2$-homogeneous caloric polynomial, even with respect to $x_{n+1}$, and satisfies $\partial_tq\ge0$.
			\end{proposition}
			
			\begin{proof}
				We divide the proof into two steps. 
				
				\smallskip
				\noindent \textit{Step 1: $\widetilde{N}(0^{+},v)=2$.} By point $iii)$ in \cref{lemma:asymptotic-second-bu}, we know that 
				\begin{equation*}
					\widetilde{N}(0^{+},v)\ge 2.
				\end{equation*}
				Therefore, we only need to show that $\widetilde{N}(0^{+},v)\le 2$. To prove this, we will use a barrier argument based on the monotonicity assumption $\partial_{t}u\ge 0$ to show that $r\mapsto H(r,v)^{1/2}$ cannot decay faster than $r^{2+\eps}$ for any $\eps>0$. The conclusion will then come by applying point $i)$ in \cref{lemma:asymptotic-second-bu}.
				
				We denote by $\widetilde u:=\widetilde u^{0,0}$.
				Since $p$ is the first blow-up of $u$ at $(0,0)$, by \cref{prop:p2-indep-time} we have 
				$$\widetilde u(x,t)=p(x)+o(|x|^2+|t|)\qquad \text{as $(x,t)\to (0,0)$.}$$ 
				Moreover, since $p\not\equiv  0$ is a nonnegative quadratic polynomial on $\{x_{n+1}=0\}$, after a rotation we may assume that $p(x',0)\ge c_1x_1^2$ for all $x'\in \R^{n}$, where $c_{1}>0$. As a consequence, for any given small $\delta>0$, we can find a sufficiently small $r_\delta\in (0,1)$ such that, for every $(x,t)\in \Lambda\cap Q_{r_{\delta}}$, 
				\begin{equation*}
					x_{1}^{2}\le \frac{p(x',0)}{c_{1}}=\frac{|p(x',0)-\widetilde{u}(x',0,t)|}{c_{1}}\le \delta(|x'|^{2}+|t|).
				\end{equation*}
				Equivalently, defining $\mathcal{C}_{\delta}:= \{(x,t)\in S_{\infty}: x_{1}^{2}+x_{n+1}^{2}\le \delta(|x'|^{2}+|t|)\}$, we just obtained that 
				\begin{equation}\label{eq:contact-set-in-small-harmonic-capacity}
					\Lambda\cap Q_{r_{\delta}}\subseteq \mathcal{C}_{\delta}.
				\end{equation}
				
				We now construct a suitable parabolically homogeneous positive caloric barrier in the set $Q_{r_{\delta}/2}\setminus \mathcal{C}_{2\delta}$. First, we let $\widetilde{\mathcal{C}}_{\delta}:= \{x\in \R^{n+1}: x_{1}^{2}+x_{n+1}^{2}\le \delta(|x'|^{2}+1)\}$ and we consider $\phi \in W^{1,2}_{0}(\R^{n+1}\setminus \widetilde{C}_{2\delta}, G(\cdot,-1))$ the first eigenfunction of the Ornstein-Uhlenbeck operator $-\mathcal{L}= -\Delta+\frac{x}{2}\cdot \nabla$ with homogeneous Dirichlet boundary conditions in the open set $\R^{n+1}\setminus \widetilde{C}_{2\delta}$, solving 
				$$\begin{cases}
					\mathcal{L}\phi+\frac{\kappa_\delta}2\phi=0&\qquad\text{in } \R^{n+1}\setminus \widetilde{\mathcal{C}}_{2\delta},\\
					\phi= 0&\qquad\text{in } \widetilde{\mathcal{C}}_{2\delta},\\
					\phi>0&\qquad\text{in }\R^{n+1}\setminus \widetilde{\mathcal{C}}_{2\delta},
				\end{cases}$$
				for some $\kappa_{\delta}>0$. We note that, if $\delta$ is small, then $\kappa_\delta$ is close to $\kappa_0$, which is twice the first eigenvalue in $\R^{n+1}\setminus \widetilde{\mathcal{C}}_{0}$ with zero Dirichlet boundary datum on $\widetilde{\mathcal{C}}_0:=\{x_1=x_{n+1}=0\}$. Since $\widetilde{\mathcal{C}}_0$ has zero harmonic capacity, then $\kappa_0=0$ and thus $\kappa_\delta\downarrow 0$ as $\delta\downarrow 0$. We then consider a parabolically $\kappa_\delta$-homogeneous function $\Phi:S_{\infty}\to \R$ defined as
				\begin{equation*}
					\Phi(x,t):=|t|^{\frac{\kappa_{\delta}}{2}}\phi\left(\frac{x}{\sqrt{-t}}\right).
				\end{equation*}
				One can readily check that $\Phi$ solves the following problem:
				$$\begin{cases}
					\mathbf{H}\Phi=0&\qquad\text{in } S_{\infty}\setminus \mathcal{C}_{2\delta},\\
					\Phi= 0&\qquad\text{in }{ \mathcal{C}}_{2\delta},\\
					\Phi>0&\qquad\text{in }S_{\infty}\setminus \mathcal{C}_{2\delta}.
				\end{cases}$$
				
				Next, we use the function $\Phi$ constructed above as a barrier to bound $\partial_{t}u$ from below. Since $\partial_{t}u$ is caloric away from the contact set $\Lambda$, \eqref{eq:contact-set-in-small-harmonic-capacity} and the strong maximum principle give 
				\begin{equation*}
					\mathbf{H}(\partial_{t}u)=0\quad \text{in $Q_{r_{\delta}/2}\setminus \mathcal{C}_{2\delta}$}\qquad \text{and}\qquad \partial_{t}u\ge c>0\quad \text{on $\left(\partial_{\text{par}}Q_{r_{\delta}/2}\right)\setminus \mathcal{C}_{2\delta}$}.
				\end{equation*}
				On the other hand, since the zero extension of $\Phi$ is subcaloric, by the maximum principle we get that $\Phi$ is bounded from above by a positive constant in $Q_{1/2}$ (see for instance \cite[Equation 4.17]{supercooled}), and vanishes on $\mathcal{C}_{2\delta}$. Therefore, there exists $c>0$ such that $\partial_{t}u\ge c\Phi$ on $\partial_{\text{par}}(Q_{r_{\delta}/2}\setminus \mathcal{C}_{2\delta})$. The maximum principle and the fact that $\partial_{t}u\ge 0=\Phi$ on $\mathcal{C}_{2\delta}$ then give
				\begin{equation}\label{eq:bound-below-detu-barrier-second-blowup}
					\partial_{t}u\ge c\Phi\qquad \text{in $Q_{r_{\delta}/2}$}.
				\end{equation}
				
				Now recall that the blow-up $p$ is independent of time. Therefore $\partial_{t}v=\partial_{t}w=\partial_{t}u$ in $Q_{3/4}$. In particular, integrating \eqref{eq:bound-below-detu-barrier-second-blowup} in $B_{r}\times (r^{2}t_{1},r^{2}t_{2})$ for some $r\in (0,r_{\delta}/2)$, $t_{1}\in (-1,-3/4)$ and $t_{2}\in (-1/2,-1/4)$ gives 
				\begin{align*}
					\fint_{B_{r}}\left(v(x,r^{2}t_{2})-v(x,r^{2}t_{1})\right)dx&=\int_{r^{2}t_{1}}^{r^{2}t_{2}}\fint_{B_{r}}\partial_{t}v(x,t)\,dx\,dt\\
					&\ge  c\int_{r^{2}t_{1}}^{r^{2}t_{2}}\fint_{B_{r}}\Phi(x,t)\,dx\,dt \ge cr^{2+\kappa_{\delta}}.
				\end{align*}
				After changing variables $y=x/r$ and averaging for $(t_{1},t_{2})\in (-1,-3/4)\times (-1/2,-1/4)$ we get 
				\begin{equation*}
					\left(\int_{-1}^{-1/4}\int_{B_{1}}v^{2}(ry,r^{2}s)\,dy \,ds\right)^{1/2}\ge c\int_{-1}^{-3/4}\int_{-1/2}^{-1/4} \fint_{B_{r}}\left(v(x,r^{2}t_{2})-v(x,r^{2}t_{1})\right)dx\,dt_{1}\,dt_{2}\ge cr^{2+\kappa_{\delta}}.
				\end{equation*}
				Finally, since $G$ is bounded below by a positive constant in $B_{1}\times (-1,-1/4)$ we conclude 
				\begin{equation*}
					\widetilde{H}(r,v)\ge H(r,v)\ge cr^{4+2\kappa_{\delta}} 
				\end{equation*}
				for every sufficiently small $r>0$. Thanks to point $i)$ in \cref{lemma:asymptotic-second-bu} we deduce $\widetilde{N}(0^{+},v)\le 2+\kappa_{\delta}$. Since $\kappa_{\delta}\downarrow 0$ as $\delta\downarrow 0$, we conclude $\widetilde{N}(0^{+},v)\le 2$, as desired.
				
				\smallskip
				\noindent \textit{Step 2: Existence and properties of the second blow-up.} Let $r_{j}\downarrow 0$ be a given sequence. First of all, we note that the rescalings $v_{r_{j}}$ solve 
				\begin{equation*}
					\mathbf{H}v_{r_{j}}= f_{r_{j}}-\mu_{r_{j}}\quad \text{in $Q_{1/r_{j}}$},\qquad f_{r}(x,t):=\frac{r^2f(rx,r^{2}t)}{H(r,v)^{1/2}},\quad \mu_{r}:= \frac{-2r\partial_{n+1}w(r\cdot,r^{2}\cdot)\mathcal{H}^{n+1}\res Q_{1/r}'}{H(r,v)^{1/2}}\ge 0.
				\end{equation*}
				By \textit{Step 1}, and point $i)$ in \cref{lemma:asymptotic-second-bu} we have $H(r_{j},v)^{1/2}\ge c_{\eps}r_{j}^{2+\eps}$, for any $\eps>0$. 
				Therefore, for any $R>0$ and $j$ sufficiently large, using also the uniform bound on $f$ from \eqref{controlf}, and taking $\eps$ small, we obtain
				\begin{equation}\label{eq:bound-fr-second-blow-up}
					\lVert f_{r_{j}}\rVert_{L^{\infty}(Q_{R})}\le C\frac{r_j^2(Rr_{j})^{\ell-2}}{c_{\eps}r_{j}^{2+\eps}}\le C_{\eps,R}\,r_{j}^{\ell-2-\eps}\to 0\qquad \text{as $j\to \infty$}.
				\end{equation}
				On the other hand, the second inequality in \cref{lemma:controlvconH} along with the doubling property in \cref{rmk:doubling} gives 
				\begin{equation}\label{eq:bound-L2-second-blow-up}
					\lVert v_{r_{j}}\rVert_{L^{2}(Q_{R})}= R^{(n+3)/2}\frac{\lVert v(Rr_{j}\cdot, R^{2}r_{j}^{2}\cdot)\rVert_{L^{2}(Q_{1})}}{H(r_{j},v)^{1/2}}\le C_{R}\frac{H(Rr_{j},v)^{1/2}}{H(r_{j},v)^{1/2}}\le C_{R}.
				\end{equation}
				Then, applying \cref{lemma:uniform-bounds-second-blow-up} to $v_{r_{j}}(R\cdot, R^{2}\cdot)$ and using \eqref{eq:bound-fr-second-blow-up}, \eqref{eq:bound-L2-second-blow-up}, we find a uniform $W^{1,2}$ bound for the sequence $v_{r_{j}}$ in $Q_{R/2}$. Since $R>0$ is arbitrary, we deduce that the sequence is bounded in $W^{1,2}_{\text{loc}}(S_{\infty})$. Therefore, there is a function $q$ such that, up to a subsequence, $v_{r_j}\to q$ weakly in $W^{1,2}_{\text{loc}}(S_{\infty})$ and strongly in $L^2_{\text{loc}}(S_{\infty})$. In particular, $q$ is even with respect to $x_{n+1}$ and $\partial_{t}q\ge0$.
				Moreover, by combining the local weak convergence in $W^{1,2}(S_{\infty})$ with the compactness arguments used in \textit{Step 1} of the proof of \cref{lemma:controlvconH} we deduce that 
				\begin{equation}\label{eq:Gaussian-convergence-proof-second-blow-up}
					\lim_{j\to \infty} H(r,v_{r_{j}})= H(r,q)\quad \text{and} \quad I(r,q)\le \liminf_{j\to \infty}I(r,v_{r_{j}})\qquad \forall r>0.
				\end{equation}
				In particular $H(1,q)=1$, and so $q\not\equiv 0$.
				
				To prove that $q$ is caloric, note that by the Signorini condition, the measures $\mu_{r_{j}}$ are concentrated on the sets $\{\widetilde{u}(r_{j}\cdot,r_{j}^{2}\cdot)=0\}'$, and that, by \cref{prop:p2-indep-time}, every limit point of sequences extracted from these sets must belong to $\{p=0\}'\times (-\infty,0]$. Using also \eqref{eq:bound-fr-second-blow-up} and the convergence of $v_{r_{j}}$ towards $q$, we deduce that $\mathbf{H}q$ is a measure concentrated on  $\{p=0\}'\times (-\infty,0]$. Since $\{p=0\}'$ is a linear space of dimension at most $n-1$, its harmonic capacity is $0$. Thus, $q$ is caloric.  
				
				It remains to prove that $q$ is parabolically $2$-homogeneous, as from this and the parabolic Liouville theorem we would also deduce that $q$ is a polynomial. 
				We first observe that by \textit{Step 1} and point $iii)$ in \cref{lemma:asymptotic-second-bu}, $N(0^{+},v)=\widetilde{N}(0^{+},v)=2$. Then, since $N(r,v_{r_j})=N(rr_j,v)$, using \eqref{eq:Gaussian-convergence-proof-second-blow-up} we get
				\begin{equation*}
					N(r,q)\le \liminf_{j\to \infty} N(r,v_{r_{j}})= \liminf_{j\to \infty}N(rr_{j},v)=N(0^{+},v)=2.
				\end{equation*}
				For the smooth caloric function $q$, direct differentiation and Gaussian integration by parts, using its evenness in $x_{n+1}$, give the identities
				\begin{equation*}
					H'(r,q)= \frac{4}{r}I(r,q)=\frac{2}{r^{3}}\int_{S_{r}^{+}}q(Zq)G,\qquad I'(r,q)= \frac{1}{r^{3}}\int_{S_{r}^{+}}(Zq)^{2}G.
				\end{equation*}
				Therefore, the Cauchy-Schwarz inequality gives
				\begin{equation}\label{eq:poon-caloric-second-blow-up}
					\begin{aligned}
						\frac{H(r,q)^{2}}{2}N'(r,q)&= I'(r,q)H(r,q)-I(r,q)H'(r,q)\\
						&= \frac{1}{r^{5}}\int_{S_{r}^{+}}(Zq)^{2}G\int_{S_{r}^{+}}q^{2}G-\frac{1}{r^{5}}\left(\int_{S_{r}^{+}}q(Zq)G\right)^{2}\ge 0,
					\end{aligned}
				\end{equation}
				that is, $r\mapsto N(r,q)$ is non-decreasing. We claim that 
				\begin{equation}\label{eq:constant-frequency-poon}
					N(r,q)=2\qquad\forall r>0.
				\end{equation}
				Suppose by contradiction that $N(r_{0},q)=\mu<2$ for some $r_{0}>0$. Then we would have $N(r,q)\le \mu$ for every $r\in (0,r_{0})$. Therefore,
				\begin{equation*}
					\frac{d}{dr}\log(H(r,q))= \frac{4I(r,q)}{rH(r,q)}=\frac{2}{r}N(r,q)\le \frac{2\mu}{r}\qquad \forall r\in (0,r_{0}).
				\end{equation*}
				Integrating and exponentiating we get
				\begin{equation*}
					H(r,q)\ge \frac{H(r_{0},q)}{r_{0}^{2\mu}}r^{2\mu}\qquad \forall r\in (0,r_{0}).
				\end{equation*}
				However, point $ii)$ in \cref{lemma:asymptotic-second-bu} and \cref{rmk:doubling} give, for sufficiently small $r$,
				\begin{equation*}
					H(r,q)=\lim_{j\to \infty}H(r,v_{r_{j}})=\lim_{j\to \infty}\frac{H(rr_{j},v)}{H(r_{j},v)}\le Cr^{4}.
				\end{equation*}
				Since $\mu<2$, we get the desired contradiction. This shows \eqref{eq:constant-frequency-poon}. Given this, imposing the equality in formula \eqref{eq:poon-caloric-second-blow-up} we obtain that $Zq$ is a constant multiple of $q$, which implies that $q$ is parabolically homogeneous. The quadratic growth of $H(r,q)^{1/2}$ forces the homogeneity to be $2$. This concludes the proof.
			\end{proof}

			\subsection{Orthogonality condition} To conclude this section we record the following useful orthogonality condition between first and second blow-ups at quadratic points. 
			\begin{lemma}\label{lemma:orthogonality-condition}
				Let $u$ be a solution of \eqref{eq:parabolic-signorini-problem}, with obstacle $\varphi$ satisfying \eqref{eq:ostacolo}. Suppose that $(0,0)\in \Gamma_{2}$, and let $p\in \mathcal{P}_{2}^+$ be the blow-up of $u$ at $(0,0)$. Let $q$ be obtained as a second blow-up at $(0,0)$ as in \cref{prop:second-blow-up}. Then,
				\begin{equation}\label{eq:orthogonality-conditions}
					\int_{\partial B_1\times \{-1\}}pq=0\qquad\text{and}\qquad\int_{\partial B_1\times \{-1\}}\widetilde pq\le0\quad\forall \widetilde p\in\mathcal{P}_2^+.
				\end{equation}
			\end{lemma}
			\begin{proof}
				One can argue exactly as in \cite[proof of Proposition 3.1, Step 4]{cc24} (see also \cite[Proposition 6.7]{frs24}), using the Monneau-type monotonicity formula from \cref{lemma:asymptotic-second-bu}, along with \cref{prop:second-blow-up} to deduce that
				\begin{equation*}
					\int_{S_{1}^{+}}pqG=0\qquad\text{and}\qquad \int_{S_{1}^{+}}\widetilde{p}qG \le 0 \quad\forall  \widetilde p \in \mathcal{P}_{2}^{+}.
				\end{equation*}
				Then, since $p,\widetilde{p}$ are harmonic quadratic polynomials independent of time, while $q(x,t)=At+\hat{q}(x)$, for some quadratic polynomial $\hat{q}$, we have 
				$$0\ge \int_{S_1^+}\widetilde pqG=\int_{S_1^+}\widetilde p\hat q G=c_n\int_{\partial B_1}\widetilde p\hat q =c_n\int_{\partial B_1\times\{-1\}} \widetilde p q ,$$ for some dimensional constant $c_n>0$. The same computation, replacing $\widetilde p$ with $p$, concludes the proof.
			\end{proof}
			\section{Cleaning at quadratic points}\label{sec:cleaning_quadratic}
			The goal of this section is to prove that the set of quadratic free boundary points $\Gamma_{2}$ for a solution of \eqref{eq:parabolic-signorini-problem} can be cleaned up with rate $2$ (see \cref{def:cleaned}).
			\begin{proposition}\label{prop:cleaning-gamma2}
				Let $u$ be a solution of \eqref{eq:parabolic-signorini-problem}, with obstacle $\varphi$ satisfying \eqref{eq:ostacolo}. Then $$\Gamma_{2}\in\text{Clean}(2).$$ 
			\end{proposition}
			In order to prove \cref{prop:cleaning-gamma2}, we will first need some technical semiconvexity estimates and a further characterization of second blow-ups, which we present in the following subsections.
			
			\subsection{Semiconvexity estimates}
			
			Let us start by recording the following classical semiconvexity estimate in the directions tangential to the thin space. Recall the $W^{2,1}_{2}(Q_{1}^{+})$-regularity of the solution in \eqref{eq:reg-sol}.
			\begin{lemma}\label{lemma:semiconvexity-estimate1}
				Let $w$ be a solution of \eqref{eq:w-solves}, with right-hand side $f$ satisfying \eqref{controlf}, and $\partial_t f\equiv0$ in $B_{7/8}\times(-1,0]$. Then, for every $e\in\{x_{n+1}=0\}\cap\partial B_1$, we have 
				$$\|\nabla_{x'}w\|_{L^\infty(Q_{3/4})}+\|(\partial_{ee}w)_-\|_{L^\infty(Q_{3/4})}\le C\Big(\|w\|_{L^\infty(Q_1)}+\|f\|_{C^2(B_{7/8})}\Big),$$ for some constant $C$ depending only on $n$.
			\end{lemma}
            \begin{proof}
                The result follows from \cite[Proposition 3.1]{Hu} and \cite[Proposition 3.2]{Hu}.
            \end{proof}
		Next, we leverage \cref{lemma:semiconvexity-estimate1} to show another useful semiconvexity estimate, relating second derivatives in tangential directions to the time derivative. This estimate will be used in the proof of \cref{prop:cleaning-gamma2} to show that any time-independent second blow-up is convex in the tangential directions belonging to the spine of the first blow-up.
		\begin{lemma}\label{lemma:semiconvexity-estimate2}
			Let $u$ be a solution of \eqref{eq:parabolic-signorini-problem}, with obstacle $\varphi$ satisfying \eqref{eq:ostacolo}, and suppose that $(0,0)\in\Gamma$. Then, there exists a constant $C>0$ such that for every $e\in\{x_{n+1}=0\}\cap\partial B_1$, we have \be\label{eq:constantC}|\partial_e \widetilde u|+(\partial_{ee}\widetilde u)_-\le C\partial_t\widetilde u\qquad\text{in }Q_{1/2},\ee where $\widetilde u:=\widetilde u^{0,0}$ is defined in \eqref{eq:utilde}. 
		\end{lemma}
		\begin{proof}
			Let $e\in\{x_{n+1}=0\}\cap\partial B_1$ be given. We want to prove that for almost every $(x_0,t_0)\in Q_{1/2}$, 
			\be\label{eq:desired-estimate}\partial_e \widetilde u(x_0,t_0)+(\partial_{ee}\widetilde u(x_0,t_0))_{-}\le C\partial_t \widetilde u(x_0,t_0).\ee
			Indeed, this implies the desired estimate, by applying the previous bound also to $-e$.
			We first reduce to the case in which $x_{0}$ is sufficiently close to the thin space. Let $\rho>0$ be a small number to be chosen later. We observe that if $|(x_0)_{n+1}|> \rho$, then, by standard parabolic estimates for $\widetilde u$ away from $\{x_{n+1}=0\}$, and by \cref{lemma:monotonicity-ass}, we have 
			$$|\partial_{e}\widetilde u(x_0,t_0)|+|\partial_{ee}\widetilde u(x_0,t_0)|\le C_\rho\qquad\text{and}\qquad\partial_t \widetilde u(x_0,t_0)=\partial_t u(x_0,t_0)\ge c_\rho,$$ where $C_\rho>0$ and $c_\rho>0$ depend on $\rho$. Thus, in this case we obtain \eqref{eq:desired-estimate} with $C:=C_{\rho}/c_{\rho}$. 
			
			Henceforth, we will assume that $|(x_0)_{n+1}|\le \rho$, and also $t_{0}<0$. 
			Since $u\in H^{3/2,3/4}(Q_{1}^{+}\cup Q_{1}')$, the tangential derivative $\partial_e\widetilde u$ is continuous up to the thin space. The gradient estimate in \cref{lemma:semiconvexity-estimate1} gives \be\label{eq:gradietn-estimate}|\partial_e\widetilde u|\le CM\qquad\text{in }Q_{3/4},\ee where $M:=\|\widetilde u\|_{L^\infty(Q_1)}+\|f\|_{C^{2}(B_{7/8})}$. For the other quantities in \eqref{eq:desired-estimate}, we introduce the incremental quotients.
			For any sufficiently small $h>0$ we define
			\begin{equation*}
				\delta_{t,h}\widetilde{u}(x,t):= \frac{\widetilde{u}(x,t+h)-\widetilde{u}(x,t)}{h},\qquad \delta^{2}_{e,h}\widetilde{u}(x,t):= \frac{\widetilde{u}(x+he, t)+\widetilde{u}(x-he, t)-2\widetilde{u}(x, t)}{h^{2}}.
			\end{equation*}
			We will show that 
			\begin{equation}\label{eq:semiconvexity-claim-difference-quotients}
				\partial_e \widetilde u(x_0,t_0)+(\delta^{2}_{e,h}\widetilde{u}(x_{0},t_{0}))_{-}\le C\delta_{t,h}\widetilde{u}(x_{0},t_{0}),
			\end{equation}
			uniformly for $h$ sufficiently small. 
			\eqref{eq:desired-estimate} will then hold almost everywhere in $Q_{1/2}$, by the strong $L^{2}$-convergence of difference quotients due to the regularity $\widetilde{u}\in W^{2,1}_{2}(Q_{1}^{+})$. 
			
			Let $\Omega:=B_{1/4}(x_0)\times (t_0-1/4,t_0]$ and $\Lambda=\{\widetilde{u} =0\}'$ be the contact set. First of all, we note that $\delta_{t,h}\widetilde{u}=\delta_{t,h}u$ because $\varphi$ is independent of time. Therefore, since $u$ is non-decreasing in time and caloric away from the contact set, 
			\begin{equation}\label{eq:heat-time-diff-quotient-semiconvexity}
				\mathbf{H}(\delta_{t,h}\widetilde{u})=0\qquad \text{in $\Omega\setminus \Lambda$.}
			\end{equation}
			On the other hand, since $\mathbf{H}u\le 0$ in $Q_{1}$, $\mathbf{H}(\delta^{2}_{e,h}\widetilde{u})\le \delta^2_{e,h}f$ in $\Omega\setminus \Lambda$. 
			In particular,
			\begin{equation}\label{eq:heat-twoderivatives-diff-quotient-semiconvexity}
				\mathbf{H}((\delta^{2}_{e,h}\widetilde{u})_{-})\ge -\widetilde{C}M\qquad \text{in $\Omega\setminus \Lambda$.} 
			\end{equation}
			Similarly, since $\mathbf{H}(\partial_e\widetilde u)=\partial_e f$,
			\begin{equation}\label{eq:heat-twoderivatives-diff-quotient-semiconvexity-gradient}
				\mathbf{H}(\partial_e\widetilde{u})\ge -\widetilde{C}M\qquad \text{in $\Omega\setminus \Lambda$.} 
			\end{equation}
			Given $c_1,c_2>0$ to be chosen later,
			we consider the function 
			$$\psi_{x_0,t_0}:=\delta_{t,h}\widetilde{u}+c_1\Phi-c_2(\partial_{e}\widetilde u)-c_2(\delta^{2}_{e,h}\widetilde u)_-+2c_2 \widetilde{C}M(t-t_0),$$
			where
			\begin{equation*}
				\Phi:=\frac{|x'-x'_0|^2}{4n}-\frac{t-t_0}2-\frac12 (x_{n+1}^2-(x_0)_{n+1}^2).
			\end{equation*}
			Note that $\mathbf{H}\Phi=0$. Then, \eqref{eq:heat-time-diff-quotient-semiconvexity}, \eqref{eq:heat-twoderivatives-diff-quotient-semiconvexity}, and \eqref{eq:heat-twoderivatives-diff-quotient-semiconvexity-gradient} give \bea\mathbf{H}\psi_{x_0,t_0}\le0\qquad\text{in }\Omega\setminus\Lambda.\eea
			
			We will choose $\rho$, $c_{1}$ and $c_{2}$ in such a way that 
			\begin{equation}\label{eq:psi-positive-par-boundary-semiconvexity}
				\psi_{x_{0},t_{0}}\ge 0\qquad \text{on $\partial_{\text{par}}(\Omega\setminus \Lambda)$}.
			\end{equation}
			Given that, the maximum principle implies $\psi_{x_{0},t_{0}}\ge 0$ in $\Omega\setminus \Lambda$. Evaluating $\psi_{x_{0},t_{0}}$ at $(x_{0},t_{0})$ we obtain the desired \eqref{eq:semiconvexity-claim-difference-quotients}. First, on the parabolic boundary $\partial_{\text{par}}\Omega$, we have 
			$$|x'-x_0'|^2+(x_{n+1}-(x_0)_{n+1})^2=\frac1{16}\quad \text{or}\quad t=t_0-\frac14.$$
			Then, using $\delta_{t,h}\widetilde u\ge0$, \eqref{eq:gradietn-estimate} and \cref{lemma:semiconvexity-estimate1}, we get
			\begin{equation}\label{eq:psi-positive-par-boundary-semiconvexity-1}
				\psi_{x_0,t_0}\ge \frac{c_1}{64n}-{c_1}C_1\rho^2-c_2C_2M\qquad\text{on }\partial_{\text{par}}\Omega
				\cap \{|x_{n+1}|\le \rho\}.
			\end{equation}
			On the other hand, by \cref{lemma:monotonicity-ass}, we can find $c_\rho$ such that 
			$\delta_{t,h} \widetilde{u}\ge c_\rho$ in $\Omega\setminus\{|x_{n+1}|\le \rho\}$ for every $h$ sufficiently small.
			Then 
			\begin{equation}\label{eq:psi-positive-par-boundary-semiconvexity-2}
				\psi_{x_0,t_0}\ge c_\rho-c_1C_3- c_2C_2M\qquad\text{in }\Omega\setminus\{|x_{n+1}|\le \rho\}.
			\end{equation} 
			Finally, since $\widetilde{u}\ge 0$ on $Q_{1}'$, we have $\partial_e\widetilde u=0$ and $\delta^{2}_{e,h}\widetilde u\ge0$ on $\Omega\cap \Lambda$. Then,
			\begin{equation}\label{eq:psi-positive-par-boundary-semiconvexity-3}
				\psi_{x_0,t_0}\ge c_1\Phi+2c_2\widetilde CM(t-t_0)\ge\left(\frac{c_1}2-2c_2\widetilde CM\right)(t_0-t)\qquad\text{on }\Omega\cap\Lambda.
			\end{equation}
			To conclude, we first choose $\rho$ such that $C_1\rho^2\le \frac{1}{128n}$, then $c_1$ such that $c_1C_3\le c_{\rho}/2$, and finally $c_2$ such that $c_2C_2M\le \min\{c_{\rho}/2, c_{1}/(128 n)\}$, and $c_2\widetilde CM\le c_{1}/4$. In this way, the right-hand sides of \eqref{eq:psi-positive-par-boundary-semiconvexity-1}, \eqref{eq:psi-positive-par-boundary-semiconvexity-2}, and \eqref{eq:psi-positive-par-boundary-semiconvexity-3} are all nonnegative, and \eqref{eq:psi-positive-par-boundary-semiconvexity} follows.    
		\end{proof}
		
		\begin{remark}\label{remark:faciledopo} 
        The same proof applies to an even solution $w$ of \eqref{eq:w-solves}, $\partial_tw\ge0$, with right-hand side $f$ satisfying \eqref{controlf}, and $\partial_t f\equiv0$ in $B_{7/8}\times(-1,0]$, provided
        $\mu:=\inf_{Q_{3/4}\cap\{|x_{n+1}|\ge\rho\}}\partial_tw>0$.
        In this case, setting
        $M:=\|w\|_{L^\infty(Q_1)}+\|f\|_{C^2(B_{7/8})}$,
        the constant $C$ in \eqref{eq:constantC} satisfies $C\le C'(1+M/\mu)$, where $C'$ depends only on $n$.
\end{remark}

		We will also use the following lower bound for the second derivative in time. 
		\begin{lemma}\label{lemma:semiconv-in-time}
			Let $u$ be a solution of \eqref{eq:parabolic-signorini-problem}, with obstacle $\varphi$ satisfying \eqref{eq:ostacolo}. Then
			\begin{equation}\label{eq:semiconvexity-in-time}
				\partial_{tt}u\ge -C\qquad\text{in }B_{1/2}\times(-3/4,-1/4),
			\end{equation} for some constant $C>0$ depending only on $n$ and $\|u\|_{L^\infty(Q_1)}$.
		\end{lemma}
		\begin{proof}
			For any function $w:Q_1\to\R$ and any $h>0$ sufficiently small, we denote by
			$\delta_{t,h} w(x,t):=h^{-1}(w(x,t+h)-w(x,t))$ the time difference quotient of $w$, defined for $t+h\le 0$.
			Let $\Omega:=B_{3/4}\times (-7/8,-1/8)$, and set
			$$v_h(x,t):=\delta_{t,h} \partial_tu=\frac{\partial_tu(x,t+h)-\partial_tu(x,t)}{h}.$$ 
			We claim that $v_{h}^-$ is subcaloric in $\Omega$, namely 
			\be\label{eq:z-subcaloric}\mathbf{H}v_{h}^-\ge0 \qquad\text{in } \Omega.
			\ee
			Let us see how to conclude once \eqref{eq:z-subcaloric} is known. First, since $v_{h}^{-}\ge 0$ and $\mathbf{H}(v_{h}^{-})\ge 0$ in $\Omega$, we may use \cref{lemma:wang-start} and a covering argument to get 
			\begin{equation}\label{eq:Leps-Linfty-proof-second-time-derivative}
				\lVert v_{h}^{-}\rVert_{L^{\infty}(B_{1/2}\times (-3/4,-1/4))}\le C \lVert v_{h}^{-}\rVert_{L^{1/2}(\Omega)}\le C\sup_{\theta>0}\,\theta \left|\{|v_{h}|>\theta\}\cap \Omega\right|.
			\end{equation}
			Then, note that we have $v_{h}=\partial_{t}\delta_{t,h}u$ and $\delta_{t,h}u\ge 0$. Moreover, $\mathbf{H}(\delta_{t,h}u)\ge 0$ in $B_{1}\times (-1,-1/16)$. Indeed, in $\{\delta_{t,h}u> 0\}$, we have $\mathbf{H}(\delta_{t,h}u)=-h^{-1}\mathbf{H}u\ge 0$. On the other hand, on the thin space, if $\delta_{t,h}u = 0$, $\partial_{n+1}\delta_{t,h}u\ge 0$. Take $\psi\in C_c^\infty(B_1\times(-1,0))$ such that
			$0\leq\psi\leq1$ and $\psi\equiv1$ on $\widetilde\Omega:=B_{7/8}\times (-15/16, -1/16)$.
			For all sufficiently small $h>0$, discrete integration by parts gives
			\begin{align*}
				\|\delta_{t,h}u\|_{L^1(\widetilde\Omega)}&\le \int_{Q_1}\delta_{t,h}u(x,t)\,\psi(x,t)\,dx\,dt\\
				&=\int_{Q_1}u(x,t)\frac{\psi(x,t-h)-\psi(x,t)}{h}\,dx\,dt\le C\|u\|_{L^\infty(Q_1)}.
			\end{align*}
			Then, we may apply \cref{lemma:calderon-zygmund} to $\delta_{t,h}u$ and combine the result with \eqref{eq:Leps-Linfty-proof-second-time-derivative} to get
			\begin{equation*}
				\lVert v_{h}^{-}\rVert_{L^{\infty}(B_{1/2}\times (-3/4,-1/4))}\le C \sup_{\theta>0}\,\theta \left|\{|v_{h}|>\theta\}\cap \Omega\right|\le C\lVert \delta_{t,h}u\rVert_{L^{1}(\widetilde{\Omega})}\le C\lVert u\rVert_{L^{\infty}(Q_{1})}
			\end{equation*}
			for every $h>0$ sufficiently small. Passing to the limit as $h\downarrow 0$ in the sense of distributions we get \eqref{eq:semiconvexity-in-time}.
			
			It only remains to show \eqref{eq:z-subcaloric}. We set, for $h,\tau>0$ small enough, 
			$$z_{h,\tau}(x,t):= \delta_{t,\tau}\delta_{t,h}u=\frac{\delta_{t,\tau} u(x,t+h)-\delta_{t,\tau} u(x,t)}h=\frac{u(x,t+h+\tau)-u(x,t+h)-u(x,t+\tau)+u(x,t)}{h\tau}.$$ 
			It suffices to show that 
			\be\label{eq:zhtau-proof}\mathbf{H}z_{h,\tau}^-\ge0 \qquad\text{in }\Omega.\ee
			Indeed, passing to the limit as $\tau\downarrow 0$ in \eqref{eq:zhtau-proof} and using that $z_{h,\tau}\to v_h$ in $L^1_{\text{loc}}$, we infer \eqref{eq:z-subcaloric}. 
			
			To prove \eqref{eq:zhtau-proof}, we first observe that $\mathbf{H}(z_{h,\tau})=0$ in $\Omega^+$ because $u$ is caloric away from the thin space. Then, since the maximum of subcaloric functions is subcaloric, we get $\mathbf{H}(z_{h,\tau}^-)\ge0$ in $\Omega^+$. 
			
			Thus, it only remains to show that $\partial_{n+1}z_{h,\tau}^{-}\ge 0$ on $\Omega'$.
			First let $(x,t)\in\{z_{h,\tau}<0\}\cap\Omega'$. Then we have 
			\begin{equation}\label{eq:second-quotient-time-explicit}
				\overline u(x,t+h+\tau)-\overline u(x,t+h)<\overline u(x,t+\tau)-\overline u(x,t),
			\end{equation}
			where $\overline u(x,t):=u(x,t)-\varphi(x')$ is nonnegative on $\Omega'$. We claim that $\overline{u}(x,t+h),\overline{u}(x,t+\tau)$ and $\overline{u}(x,t+h+\tau)$ are all strictly positive. Indeed, since $\partial_t\overline u=\partial_tu\ge0$, the quantities on the two sides of \eqref{eq:second-quotient-time-explicit} are both nonnegative. In particular, $\overline u(x,t+\tau)>0$. Using again $\partial_t\overline u\ge0$, we get $\overline u(x,t+h+\tau)>0$. Finally, if by contradiction we had $\overline u(x,t+h)=0$, then we would have $\overline u(x,t)=0$ as well, and \eqref{eq:second-quotient-time-explicit} would reduce to $\overline u(x,t+h+\tau)<\overline u(x,t+\tau)$, which is a contradiction, by the monotonicity again. 
			By the Signorini condition we then find 
			$$\partial_{n+1}z_{h,\tau}^-=-\frac{\partial_{n+1}u}{h\tau}\ge0\qquad\text{on }\{z_{h,\tau}<0\}\cap\Omega'.$$ 
			On the other hand, for $(x,t)\in\{z_{h,\tau}\ge0\}\cap\Omega'$, we have $z_{h,\tau}^-(x,t)=0$. Then, since $z_{h,\tau}^-\ge0$ in $\Omega^+$, we must have 
			$$\partial_{n+1} z_{h,\tau}^-\ge0\qquad\text{on }\{z_{h,\tau}\ge0\}\cap\Omega'.$$
			Combining the two inequalities above, we get $\partial_{n+1} z_{h,\tau}^-\ge0$ in $\Omega'$, concluding the proof.
		\end{proof}

		\subsection{Structure of the second blow-up}
		In the following lemma, we combine \cref{prop:p2-indep-time}, \cref{prop:second-blow-up} and \cref{lemma:orthogonality-condition} to derive a precise characterization of the pairs $(p,q)$ that may arise as first and second blow-ups at quadratic free boundary points for solutions of \eqref{eq:parabolic-signorini-problem}. 
		
		\begin{lemma}\label{lem:characterization-second-blow-up}
			Let $u$ be a solution of \eqref{eq:parabolic-signorini-problem}, with obstacle $\varphi$ satisfying \eqref{eq:ostacolo}. Suppose that $(0,0)\in \Gamma_{2}$, and let $p\in \mathcal{P}_{2}^+$ be the blow-up of $u$ at $(0,0)$. Let $q$ be obtained as a second blow-up at $(0,0)$ as in \cref{prop:second-blow-up}. Suppose that the spine $\{p=0\}'$ has dimension $m\in \{0,\dots,n-1\}$. Then, up to a rotation in the thin space, $p$ and $q$ have the following form:
			\begin{equation}\label{eq:shape-p-q-quadratic}
				p(x)= \sum_{i=m+1}^{n}\nu_{i}x_{i}^{2}-\nu x_{n+1}^{2},\qquad q(x,t)= At+ \sum_{i=1}^{m}a_{i}x_{i}^{2}+b\sum_{i=m+1}^{n+1}x_{i}^{2},
			\end{equation}
			where 
			\begin{equation}\label{eq:relations-coefficients-p-q-quadratic}
				\begin{gathered}
					\nu_i>0\quad \forall i\in\{m+1,\ldots,n\},\qquad \nu=\sum_{i=m+1}^{n}\nu_{i},\\
					a_{i}\le b\quad \forall i\in \{1,\dots,m\},\qquad A=2\sum_{i=1}^{m}a_{i}+2(n+1-m)b\ge 0,\qquad b>0.
				\end{gathered}
			\end{equation}
		\end{lemma}
		\begin{proof}
			We know that $p\in \mathcal{P}_{2}^{+}$, and that $q$ is a parabolically $2$-homogeneous polynomial, even in $x_{n+1}$. Since the spine $\{p=0\}'$ has dimension $m$, we may choose a system of coordinates $x_{1},\ldots,x_{n}$ in which
			\begin{equation*} 
				p(x)=\sum_{i=m+1}^{n}\nu_{i}x_{i}^{2}-\nu x_{n+1}^{2},\qquad q(x,t)=At+\sum_{i=1}^{n}a_{i}x_{i}^{2}+2\sum_{1\le i<j\le n}\alpha_{ij}x_{i}x_{j}+bx_{n+1}^{2},
			\end{equation*}
			where $\nu_{i}>0$ for all $i\in \{m+1,\dots,n\}$ and $\nu=\sum_{i=m+1}^{n}\nu_{i}$ by harmonicity of $p$. Up to a further rotation in the vector space spanned by the first $m$ variables, we may also assume that 
			\begin{equation*}
				\alpha_{ij}=0\qquad \forall 1\le i<j\le m.
			\end{equation*}
			Imposing caloricity of $q$ and $\partial_{t}q\ge 0$ we deduce the following additional constraint:
			\begin{equation}\label{eq:3condtions}
				A=2\sum_{i=1}^{n}a_{i}+2b\ge 0. 
			\end{equation}
			We will now test the orthogonality properties from \cref{lemma:orthogonality-condition} with some suitably chosen competitors, in order to derive the remaining identities on the coefficients. 
			
			By explicit computations, one can prove that there is a positive constant $\gamma=\gamma(n)>0$ such that 
			\begin{equation*}
				\int_{\partial B_{1}}x_{i}^{2}x_{j}^{2}= \gamma (1+2\delta_{ij})\qquad \forall  i,j\in \{1,\dots, n+1\},
			\end{equation*} 
			where $\delta_{ij}=1$ if $i=j$ and is zero otherwise. 
			For any $i\in \{1,\dots,n\}$ we let $p^{(i)}:=x_{i}^{2}-x_{n+1}^{2}\in \mathcal{P}_{2}^+$. 
			Testing the second condition in \eqref{eq:orthogonality-conditions} with $\widetilde{p}=p^{(i)}$ and using the symmetry of the integrals involved, we get
			\begin{equation}\label{eq:conditions1}
				0\ge \int_{\partial B_{1}\times \{-1\}}qp^{(i)}=2\gamma(a_{i}-b)\implies a_{i}\le b \qquad\forall i\in \{1,\dots,n\}.
			\end{equation}
			Observing that $p=\sum_{i=m+1}^{n}\nu_{i}p^{(i)}$, from the first condition in \eqref{eq:orthogonality-conditions} we also deduce 
			\begin{equation*}
				0=\int_{ \partial B_{1}\times \{-1\}}qp =2\gamma\sum_{i=m+1}^n\nu_i(a_{i}-b).
			\end{equation*}
			Combining this with \eqref{eq:conditions1}, since $\nu_{i}>0$ for every $i\in \{m+1,\dots,n\}$, we obtain
			\be\label{eq:condtions2}a_{i}=b\qquad\forall i\in \{m+1,\dots,n\}.\ee
			
			Now, for any choice of $1\le i<j\le n$ with $j\in \{m+1,\dots,n\}$, and any $s\in \R$, we call $p^{(i,j)}_{s}:= (x_{j}+sx_{i})^{2}-(1+s^{2})x_{n+1}^{2}\in \mathcal{P}_{2}^+$. Imposing the second condition in \eqref{eq:orthogonality-conditions} with $\widetilde{p}=p^{(i,j)}_{s}$ we get
			$$0\ge \int_{ \partial B_{1}\times \{-1\}}qp^{(i,j)}_{s}=2\gamma(a_{j}+s^{2}a_{i}+2\alpha_{ij}s-b-s^{2}b)=2\gamma s^{2}(a_{i}-b)+4\gamma \alpha_{ij}s,$$
			where we used \eqref{eq:condtions2}. Since the inequality above must hold for all $s\in \R$ we conclude 
			\begin{equation*}
				\alpha_{ij}=0\qquad \forall  1\le i<j\le n,\ j\in \{m+1,\dots,n\}.
			\end{equation*}
			
			We derived the form in \eqref{eq:shape-p-q-quadratic} along with all the relations in \eqref{eq:relations-coefficients-p-q-quadratic} apart from $b>0$. However, since $a_{i}\le b$ for $i\in \{1,\dots,m\}$, we have 
			\begin{equation*}
				b\ge \frac{A}{2(n+1)}\ge 0.
			\end{equation*}
			If $b=0$, then $a_{i}\le 0$ for all $i\in \{1,\dots,n\}$, and \eqref{eq:3condtions} gives $a_{i}=0$ for all $i\in \{1,\dots,m\}$, which is impossible because $q\not\equiv 0$. This shows $b>0$ and concludes the proof. 
		\end{proof}
		
		\subsection{Proof of quadratic cleaning for $\Gamma_{2}$} 
		We now have all the ingredients to prove the main result of this section.
		\begin{proof}[Proof of \cref{prop:cleaning-gamma2}]
			We need to prove that for every $(x_0,t_{0})\in \Gamma_2$, there are $\rho >0$ and $C>0$ such that
			\begin{equation}\label{eq:proof-clean2}
				\left\{(x,t)\in B_{\rho}(x_{0})\times(-1,0]:t>t_{0}+C|x-x_0|^{2}\right\}\cap \Gamma_2 = \emptyset .
			\end{equation}
			If $t_{0}=0$, \eqref{eq:proof-clean2} holds trivially. Otherwise, after translating and rescaling, we may assume that $(x_0,t_0)=(0,{-1/2})$. We denote by $p\in \mathcal{P}_{2}^{+}$ the blow-up of $u$ at $(0,-1/2)$, and we set $v:=w-p$, where $w:=\zeta\widetilde u$ and $\widetilde{u}:=\widetilde u^{0,-1/2}$ are defined in \eqref{eq:utilde}, \eqref{eq:subtraction-obstacle+cutoff}.
			We divide the proof into two steps.
			
			\smallskip
			\noindent\textit{Step 1.} We first prove that there exist $c_1>0$ and $r_{0}\in (0,1/2)$ such that
			\be\label{eq:boundH-clean}\fint_{Q_r\cap\{|x_{n+1}|\ge r/2\}}\partial_t\widetilde u\ge c_1\frac{H(r,v)^{1/2}}{r^{2}}\qquad \forall r\in(0,r_{0}).\ee
			Note that $v=\widetilde{u}-p$ and so $\partial_{t}\widetilde{u}=\partial_{t}v$ in $Q_{r}$ for $r$ small enough. Therefore, if \eqref{eq:boundH-clean} fails for every choice of $c_{1}$ and $r_{0}$, there must be a sequence $r_{j}\downarrow 0$ for which
			\be\label{eq:cond-ass-clean2}\int_{Q_1\cap\{|x_{n+1}|\ge 1/2\}}\partial_{t}v_{r_{j}}\to0\qquad \text{as $j\to \infty$},\ee
			where 
			\begin{equation*}
				v_{r}:= \frac{v(r\cdot,r^2\cdot)}{H(r,v)^{1/2}}.
			\end{equation*}
			By \cref{prop:second-blow-up}, up to a subsequence, $v_{r_{j}}$ converges locally weakly in $W^{1,2}(S_{\infty})$ to some $q$, a second blow-up of $u$ at $(0,-1/2)$. In addition, since $q$ is a parabolically $2$-homogeneous caloric polynomial, $\partial_{t}q \equiv A$ is a constant. \eqref{eq:cond-ass-clean2} and the local weak convergence $\partial_{t}v_{r_{j}}\rightharpoonup \partial_{t}q$ yield $A=0$. 
			
			Let $m\in\{0,\ldots,n-1\}$ be the dimension of the spine $\{p=0\}'$, and let us take coordinates such that the representation in \eqref{eq:shape-p-q-quadratic} for $p$ and $q$ holds. Then, $p$ is invariant in $x_{i}$ for all $i\in \{1,\ldots,m\}$.
			Thanks to \cref{lemma:semiconvexity-estimate2}, in a neighborhood of $(0,0)$ it holds 
			$$\partial_{ii}\widetilde{u}\ge -C\partial_t \widetilde{u}\qquad\forall i\in\{1,\ldots,m\}.$$
			Then, since $\partial_{ii}p=\partial_{t}p=0$, after rescaling we get 
			\begin{equation*}
				\partial_{ii}v_{r_{j}}\ge -C\partial_{t}v_{r_{j}}\qquad \text{in $Q_{1}$},
			\end{equation*}
			for $j$ large enough. Passing to the limit in this inequality in the sense of distributions, using the local weak $W^{1,2}$-convergence of $v_{r_{j}}$ to $q$ and $\partial_{t}q\equiv 0$, gives $0\le \partial_{ii}q\equiv 2a_{i}$ for all $i\in \{1,\dots,m\}$, where $a_{i}$ are the coefficients from \eqref{eq:shape-p-q-quadratic}. By \eqref{eq:relations-coefficients-p-q-quadratic}, we deduce $A>0$, a contradiction. This concludes \textit{Step 1}. 
			
			\smallskip
			\noindent\textit{Step 2.} Next, we use the bound \eqref{eq:boundH-clean} obtained in \textit{Step 1} to prove \eqref{eq:proof-clean2}. First, for a sufficiently small $r>0$ we show 
			\be\label{eq:estimare1}\widetilde u\ge p-CH(r,v)^{1/2}\qquad\text{in }B_r\times(-r^2,1/2).\ee
			The function $v_{2r}$ satisfies $\mathbf{H}v_{2r}\le (2r)^{2}f(2r\cdot, 4r^{2}\cdot)/H(2r,v)^{1/2}$ in $Q_{1}$. By \eqref{controlf}, $\lVert f\rVert_{L^{\infty}(Q_{2r})}\le Cr^{\ell-2}$, and by point $i)$ in \cref{lemma:asymptotic-second-bu}, $H(2r,v)^{1/2}\ge c_{\eps}r^{2+\eps}$, for any $\eps>0$. Therefore, choosing $\eps<\ell-2$, we find
			\begin{equation*}
				\mathbf{H}v_{2r}^{-}\ge -C_{\eps}r^{\ell-2-\eps}\ge -C\qquad \text{in $Q_{1}$}. 
			\end{equation*}
			Moreover, by \cref{lemma:controlvconH} we have $\lVert v_{2r}\rVert_{L^{2}(Q_{1})}\le C$. Then, applying the $L^{2}-L^{\infty}$ estimate from \cref{lemma:wang-start} to $v_{2r}^{-}$ we get $\lVert v_{2r}^{-}\rVert_{L^{\infty}(Q_{1/2})}\le C$. Undoing the scaling we deduce
			\begin{equation*}
				\widetilde{u}\ge p-CH(2r,v)^{1/2}\ge p-CH(r,v)^{1/2}\qquad \text{in $Q_{r}$},
			\end{equation*}
			where we also exploited the doubling property from \eqref{lemmafinal2}. By the monotonicity in time of $\widetilde{u}$ we finally conclude \eqref{eq:estimare1}.
			
			Now, since $\partial_{t}\widetilde u\ge 0$ and $\mathbf{H}(\partial_t\widetilde u)=0$ in $(B_{3r}\cap\{|x_{n+1}|\ge r/8\})\times(-9r^2,1/2)$, by the parabolic Harnack inequality, from \eqref{eq:boundH-clean} at scale $2r$ we deduce the pointwise lower bound
			$$\partial_t\widetilde u(\cdot,r^{2})\ge 2c\frac{H(r,v)^{1/2}}{r^{2}}\qquad\text{in }B_r\cap\{|x_{n+1}|\ge r/4\}.$$ 
			Combining this estimate with $\partial_{tt}\widetilde u\ge-C$ in $B_{1/2}\times (-1/4,1/4)$ from \cref{lemma:semiconv-in-time}, since $r^{-2}H(r,v)^{1/2}\ge c_\eps r^{\eps}$ for some $\eps\in (0,1)$, we get, for some $c_{*}>0$ and $r$ small enough
			$$\partial_t\widetilde u\ge 2c\frac{H(r,v)^{1/2}}{r^{2}}-Cr\ge c\frac{H(r,v)^{1/2}}{r^{2}}\qquad\text{in }\left(B_r\cap\{|x_{n+1}|\ge r/4\}\right)\times(r^{2},c_{*} r).$$
			Integrating the last inequality starting from \eqref{eq:estimare1} we obtain
			$$\widetilde u(\cdot,\tau)\ge p-CH(r,v)^{1/2}+c\frac{H(r,v)^{1/2}}{r^{2}}(\tau-r^{2})\qquad\text{in }B_r\cap\{|x_{n+1}|\ge r/4\},\quad \forall \tau\in[r^{2},c_{*} r].$$ 
			Let $K>0$ be a constant to be chosen later, and take $\overline{C}:=1+Cc^{-1}K$ and $\tau=\overline{C}r^{2}$. Then for $r$ sufficiently small, by the previous inequality we get
			\begin{equation*}
				\widetilde u(\cdot,\overline{C}r^{2})\ge p+CH(r,v)^{1/2}(K-1)\qquad \text{in } B_r\cap\{|x_{n+1}|\ge r/4\},
			\end{equation*}
			which combined with \eqref{eq:estimare1} and time monotonicity of $\widetilde{u}$ yields
			\be\label{eq:estimare2}\widetilde u\ge p+CH(r,v)^{1/2}(-1+K\chi_{\{|x_{n+1}|\ge r/4\}})\qquad\text{in }B_r\times(\overline C r^{2},1/2).\ee
			We now consider a barrier $h$ solving 
			$$\begin{cases}
				\mathbf{H}h=\sigma&\qquad\text{in } B_1\times(0,\infty),\\
				h=\chi_{\{|x_{n+1}|\ge 1/4\}}&\qquad\text{on }\partial B_1\times(0,\infty),\\
				h=0&\qquad\text{on }B_1\times\{0\},
			\end{cases}$$
			where $\sigma>0$ is chosen sufficiently small depending only on $n$ so that 
			\begin{equation}\label{eq:lower-bound-barrier-cleaning-quadratic}
				\widetilde c:=\inf_{B_{1/2}\times[1,\infty)}h>0.
			\end{equation}
			We then consider 
			$$\psi(x,t):=p(x)+CH(r,v)^{1/2}\left(-1+Kh\left(\frac xr,\frac{t-\overline Cr^2}{r^2}\right)\right).$$
			Let $\Omega_r:=B_r\times (\overline Cr^2,1/2)$. By \eqref{eq:estimare2} we get
			$$\widetilde u\ge\psi\qquad\text{on }\partial_{\text{par}}\Omega_r.$$
			On the other hand, $\mathbf{H}\psi$ is of order $r^{-2}H(r,v)^{1/2}\ge cr^{\eps}$, while $\mathbf{H}\widetilde u$ is less than $r^{\ell-2}$ in $\Omega_r$. Choosing $\eps<\ell-2$ and $r$ small enough we get that $\widetilde u-\psi$ is supercaloric in $\Omega_{r}$, and by the maximum principle, $\widetilde u\ge\psi$ in $\Omega_r$. 
			Evaluating the inequality at time $t=(\overline C+1)r^2$, taking \eqref{eq:lower-bound-barrier-cleaning-quadratic} into account, and using again the time monotonicity of $\widetilde{u}$ we get
			$$\widetilde u\ge p+CH(r,v)^{1/2}(-1+K\widetilde c)\qquad\text{in }B_{r/2}\times ((\overline C+1)r^2,1/2).$$ 
			Fix $K>\widetilde c^{-1}$. Then, for $r$ sufficiently small depending on this choice, we obtain $\widetilde u>0$ in $B'_{r/2}\times ((\overline C+1)r^2,1/2)$. This implies the desired cleaning condition from \eqref{eq:proof-clean2}, concluding the proof. 
		\end{proof}
        \begin{remark}
            A detailed inspection of the proofs reveals that the convergence rate to the first blow-up in \cref{prop:p2-indep-time}, as well as the constants from the second blow-up analysis in \cref{subsection-second-blow-up} and the constants of the cleaning from \eqref{eq:proof-clean2} are uniform for nearby quadratic points. This fact will be used without further reference in the following section.  
        \end{remark}
		
		\section{Stratification of quadratic points}\label{sect:stratification_quadratic}
		In this section we prove a stratification result for quadratic points. As key ingredients, we use \cref{prop:p2-indep-time} and the cleaning result in \cref{prop:cleaning-gamma2}. We start by introducing some notation. The spine of a non-zero $p\in \mathcal{P}_{2}^{+}$ is defined as the null space of $p(\cdot,0)$, which also coincides with the vector space of its invariant directions:
		\begin{equation}\label{eq:spine-of-p-quadratic}
			S(p):= \left\{\xi\in \R^{n}: p(\xi,0)=0\right\}=\left\{\xi\in \R^{n}: \partial_{\xi}p(\cdot,0) \equiv 0\right\}.
		\end{equation}
		The dimension of the spine is denoted by 
		\begin{equation*}
			d(p):= \dim S(p)\in \{0,\dots,n-1\}.
		\end{equation*}
		Note that $S(p)$ cannot be the full $\R^{n}$, otherwise, since $p$ is harmonic and $\partial_{n+1}p=0$ on $\R^{n}\times \{0\}$ we would have $p\equiv 0$.
		Let $u$ be a solution of \eqref{eq:parabolic-signorini-problem}, with obstacle $\varphi$ satisfying \eqref{eq:ostacolo}. By \cref{prop:p2-indep-time}, every point $(x_{0},t_{0})\in \Gamma_{2}$ admits a unique non-zero blow-up $p^{x_0,t_0}\in\mathcal{P}_{2}^{+}$. We may then split $\Gamma_{2}$ according to the value of the dimension of the spine of its blow-up:
		\begin{equation}\label{eq:slitting-gamma2} \Gamma_{2}=\bigcup_{j=0}^{n-1}\Gamma_{2}^{j},\qquad \Gamma_{2}^{j}:= \left\{(x_{0},t_{0})\in \Gamma_{2}: d(p^{x_0,t_0})=j\right\}\quad \forall j\in \{0,\dots,n-1\}.
		\end{equation}
		
		\begin{proposition}\label{corollary:stratification-even}
			Let $u$ be a solution of \eqref{eq:parabolic-signorini-problem}, with obstacle $\varphi$ satisfying \eqref{eq:ostacolo}. Consider the splitting of \,$\Gamma_{2}$ in \eqref{eq:slitting-gamma2}. Then, for every $j\in\{0,\ldots,n-1\}$,
			$\Gamma_{2}^j$ is locally contained in a $j$-dimensional manifold of class $C^{1}$ of the form $(x_{j+1},\ldots,x_{n},t)=g(x_1,\ldots,x_j)$, after a suitable rotation in the thin space. In particular, $\pi_x(\Gamma_2)$  is $(n-1)$-rectifiable and 
			\bea\label{eq:dimensional-bound-Gamma2}\dim_{\mathcal{H}}(\pi_x(\Gamma_2))\le n-1.\eea
		\end{proposition}
		
		\begin{proof} 
			The proof follows by combining the results in \cite[Section 14]{dgpt17} with the cleaning result of $\Gamma_2$ in \cref{prop:cleaning-gamma2}. Precisely, let $(x_0,t_0)\in\Gamma_2$ and $\rho>0$ sufficiently small to be chosen later. For every $(x_1,t_1)\in\Gamma_2\cap \mathscr{C}'_\rho(x_0,t_0)$, we denote by $p^{x_1,t_1}\in\mathcal{P}_2^+$ the blow-up of $u$ at $(x_1,t_1)$. 
			For every $(\beta, j)\in\N^n \times \N$ such that $|\beta|+j\le 2$, we let $f_{\beta,j}:\Gamma_2\to\R$ be 
			$$ f_{\beta,j}(x_1,t_1):=\begin{cases}
				0&\qquad\text{if }|\beta|\in\{0,1\}\text{ or }j\ge1,\\
				D_{x'}^\beta p^{x_1,t_1}(0,0)&\qquad\text{if }|\beta|=2\text{ and } j=0.
			\end{cases}$$
			For $(x_1,t_1),(x_2,t_2)\in\Gamma_2\cap \mathscr{C}'_{\rho}(x_{0},t_{0})$, we set $r:=2\dist_{\text{par}}((x_1,t_1),(x_2,t_2))$. 
			We define $$g_{\beta,j}(x_1,x_2,t_1,t_2):=f_{\beta,j}(x_1,t_1)-\sum_{|\sigma|+i\le 2-|\beta|-j}\frac{f_{\beta+\sigma,j+i}(x_2,t_2)}{\sigma!i!}(x_1'-x_2')^\sigma(t_1-t_2)^i.$$
            By the cleaning result in \cref{prop:cleaning-gamma2}, we have $r\le C|x_1-x_2|$. 
			Then, arguing as in \cite[Lemma 14.5]{dgpt17} (see also \cref{lemma:key-time-like} below), there exists a modulus of continuity $\omega$, such that 
			$$|g_{\beta,0}(x_1,x_2,t_1,t_2)|\le r^{2-|\beta|}\omega(r)\le C\Big(|x_1-x_2|+|t_1-t_2|\Big)^{2-|\beta|}\omega\Big(C(|x_1-x_2|+|t_1-t_2|)\Big) \qquad \text{for } |\beta|\le2,$$ and $g_{\beta,j}(x_1,x_2,t_1,t_2)=0$ for $j\ge1$.
            The Whitney extension theorem gives $F\in C^2(\mathscr{C}'_{\rho}(x_0,t_0))$ such that, for $|\beta|+j\le 2$, we have $\partial_{x'}^\beta \partial_{t}^j F=f_{\beta,j}$ in $\Gamma_2\cap \mathscr{C}'_{\rho}(x_{0},t_{0})$. Note that here $C^2$ denotes the space of functions with two continuous space-time derivatives with respect to the usual Euclidean distance.
			
			On the other hand, for every $(\beta,j)\in\mathbb N^n\times\mathbb N$ with $|\beta|+j\le1$, we define $\widetilde f_{\beta,j}:\Gamma_2\to\mathbb R$ by $$
                \widetilde f_{\beta,j}(x_1,t_1):=
                \begin{cases}1 &\qquad \text{if }(\beta,j)=(0,1),\\
                0 & \qquad \text{otherwise}.
                \end{cases}$$
            Then, by the quadratic cleaning for $\Gamma_2$ proved in \cref{prop:cleaning-gamma2}, the function
            $$\widetilde g_{\beta,j}(x_1,x_2,t_1,t_2):=\widetilde f_{\beta,j}(x_1,t_1)-\sum_{|\sigma|+i\le1-|\beta|-j}\frac{\widetilde f_{\beta+\sigma,j+i}(x_2,t_2)}{\sigma!i!}(x_1'-x_2')^\sigma(t_1-t_2)^i$$
            satisfies the Whitney compatibility conditions, since
            $$|\widetilde g_{0,0}(x_1,x_2,t_1,t_2)|=|t_1-t_2|\le C|x_1-x_2|^2,$$
and $\widetilde g_{\beta,j}(x_1,x_2,t_1,t_2)=0$ for $|\beta|+j=1$.
The classical Whitney extension theorem therefore yields a function $\Theta\in C^1(\mathscr{C}'_\rho(x_0,t_0))$ satisfying $$\Theta(x_1,t_1)=|\nabla_{x'}\Theta(x_1,t_1)|=0\quad\text{and}\quad\partial_t\Theta(x_1,t_1)=1\qquad \forall (x_1,t_1)\in\Gamma_2\cap \mathscr{C}'_{\rho}(x_{0},t_{0}).$$
			Arguing as in \cite[proof of Theorem 12.12, Step 2]{dgpt17} we get
			$$\ker\left(D^2_{x'}F(x_0',t_0)\right)=S(p^{x_0,t_0}),$$ where $S(p^{x_0,t_0})$ is the spine defined in \eqref{eq:spine-of-p-quadratic}. If $(x_0,t_0)\in\Gamma_2^j$ for some $j\in\{0,\ldots,n-1\}$, since $\dim S(p^{x_0,t_0})=j$, we have $${\rm rank}\left(D^2_{x'}F(x_0',t_0)\right)=n-j.$$ As a consequence, there exist distinct indices $i_1,\ldots, i_{n-j}\in\{1,\ldots,n\}$ such that, if we set $F_i:=\partial_{i}F\in C^1(\mathscr{C}'_\rho(x_0,t_0))$,
			the space-time gradients $\nabla_{x',t} F_{{i_1}}(x_0',t_0),\ldots,\nabla_{x',t} F_{{i_{n-j}}}(x_0',t_0),\nabla_{x',t} \Theta(x'_0,0,t_0)$ are linearly independent.
			Therefore, $\Gamma_2^j\cap \mathscr{C}'_\rho(x_0,t_0)$ is locally contained in the zero set of $$\Psi(x',t):=\left(F_{{i_1}}(x',t),\ldots,F_{{i_{n-j}}}(x',t),\Theta(x',0,t)\right)\in C^1(\mathscr{C}'_\rho(x_0,t_0)).$$ By the implicit function theorem, up to a rotation in the thin space this zero set can be written locally as a $j$-dimensional $C^{1}$-manifold of the form $(x_{j+1},\dots,x_{n},t)=g(x_{1},\dots,x_{j})$. This concludes the proof. 
		\end{proof}

		\section{Space-time regularity of the free boundary}\label{sec:spacetime_reg}
		In this section, we study the space-time structure of the free boundary, proving \cref{theorem:structure-spacetime-intro}.
		We show the following more precise statement. 
		\begin{proposition}\label{theorem:structure-spacetime}
			Let $u$ be a solution of \eqref{eq:parabolic-signorini-problem}, with obstacle $\varphi$ satisfying \eqref{eq:ostacolo}. Let $(x_0,t_0)\in\Gamma$ with $t_0<0$. Then, there exists $\rho>0$ and $\tau: B'_\rho(x_0)\to(-1,0)$, with $\tau\in C^1(B'_\rho(x_0))$, such that 
            \begin{align*}
                \Lambda\cap \mathscr{C}'_\rho(x_0,t_0)&=\left\{(x',0,t)\in \mathscr{C}'_\rho(x_0,t_0): \ t\le \tau(x',0)\right\},\\
                \Gamma\cap \mathscr{C}'_\rho(x_0,t_0)&=\left\{(x',0,t)\in \mathscr{C}'_\rho(x_0,t_0): \ t=\tau(x',0)\right\}.
            \end{align*}
			In particular, $\Gamma\cap \{t<0\}$ is locally an $n$-dimensional manifold of class $C^1$.
			Moreover, the singular set $\Sigma$ can be characterized as $$\Sigma\cap \mathscr{C}'_\rho(x_0,t_0)=\left\{(x',0,\tau(x',0))\in \mathscr{C}'_\rho(x_0,t_0): \ \nabla_{x'} \tau (x',0)=0\right\}.$$
		\end{proposition}

		\begin{lemma}\label{lemma:modulus-of-cont}
			Let $u$ be a solution of \eqref{eq:parabolic-signorini-problem}, with obstacle $\varphi$ satisfying \eqref{eq:ostacolo}. Let $(0,-1/2)\in\Sigma$ be a singular point. Then, there exists a modulus of continuity $\omega(r)\downarrow 0$ as $r\downarrow 0$ such that, for $r>0$ sufficiently small, we have
			\be\label{eq:eqeqeq2}|\nabla_{x'}\widetilde u|\le \omega(r)\partial_t\widetilde u\qquad\text{in } \mathscr C_{r},\ee where $\widetilde u:=\widetilde u^{0,-1/2}$ is defined in \eqref{eq:utilde}.
		\end{lemma}
		\begin{proof}
			By \cref{prop:onlyle3}, we have that $\Sigma=\Gamma_2\cup \Gamma_3$.
            Let $\kappa\in\{2,3\}$ be the frequency at $(0,-1/2)$. The backward growth estimate, combined with the heat comparison argument in the proof of \cite[Lemma 9.4]{dgpt17} and the higher-order vanishing of the source term, yields $$\|\widetilde u\|_{L^\infty(\mathscr C_r)}\le Cr^\kappa,$$
            for every sufficiently small $r$.
            Now we distinguish two cases according to the frequency at $(0,-1/2)$.

			\smallskip
			\noindent\textit{Step 1: $(0,-1/2)\in\Gamma_2$.} In this case, we first claim that for every $\sigma\in(0,1)$ there exists $c_\sigma>0$ such that \be\label{eq:claim3-for-quadratic}\partial_t\widetilde u\ge c_\sigma r^{1/2}\qquad\text{in }\mathscr C_{r}\cap \{|x_{n+1}|\ge \sigma r\}.\ee
			To prove the claim, we argue as in \textit{Step 1} of the proof of \cref{prop:cleaning-gamma2}. Precisely, 
			the same argument in \cref{prop:cleaning-gamma2} shows that for every $\sigma\in(0,1)$ there exists $c_\sigma>0$ such that \bea\fint_{(B_{2r}\times(-4r^2,-2r^2))\cap\{|x_{n+1}|\ge \sigma r\}}\partial_t\widetilde u\ge c_\sigma\frac{H(r,v)^{1/2}}{r^{2}}\qquad \forall r\in(0,r_{0}),\eea 
			where $p\in \mathcal{P}_{2}^{+}$ is the blow-up at $(0,-1/2)$ and $v:=\zeta\widetilde u-p$.
			We also observe that, by point $i)$ in \cref{lemma:asymptotic-second-bu} with $\eps=1/2$, we have $r^{-2}H(r,v)^{1/2}\ge c r^{1/2}$.
			Then, arguing as in \textit{Step 2} of the proof of \cref{prop:cleaning-gamma2}, Harnack inequality implies the claim \eqref{eq:claim3-for-quadratic}.
			
			Now we would like to apply \cref{lemma:semiconvexity-estimate2} to the rescaled function $\widetilde u_r(x,t):=r^{-2}\widetilde u(rx,r^2t)$.
			We also use the explicit dependence of the constant $C$ given by \cref{remark:faciledopo}. More precisely, by rescaling \eqref{eq:claim3-for-quadratic}, we get 
			$$\partial_t\widetilde u_r\ge c_\sigma r^{1/2}\qquad\text{in }\mathscr C_{1}\cap \{|x_{n+1}|\ge \sigma \}.$$
			We choose $\sigma:=\rho$, where $\rho$ is the constant in the proof of \cref{lemma:semiconvexity-estimate2}. In particular, $\sigma$ depends only on $n$. By \cref{remark:faciledopo}, $C$ can be chosen so that $C\le C'(1+M_r c^{-1}_\sigma r^{-1/2})$, where $M_r:=\|\widetilde u_r\|_{L^\infty(\mathscr{C}_1)}+\|f_r\|_{C^{2}(B_{7/8})}$ and $f_r(x,t):=f(rx,r^2t)$. Note that $M_r$ is bounded by a constant which does not depend on $r$. 
			Then, by \cref{lemma:semiconvexity-estimate2}, we have
			\bea|\nabla_{x'}\widetilde u_r|\le C(1+r^{-1/2})\partial_t\widetilde u_r\qquad\text{in } \mathscr C_{1/2}.\eea
			Rescaling back, we get 
			\bea|\nabla_{x'}\widetilde u|\le C(r+r^{1/2})\partial_t\widetilde u\le Cr^{1/2}\partial_t\widetilde u\qquad\text{in } \mathscr C_{r/2}.\eea
			Then, we conclude \eqref{eq:eqeqeq2}, for $\omega(r)\simeq r^{1/2}$, concluding the proof of \textit{Step 1}.

			\smallskip
			\noindent\textit{Step 2: $(0,-1/2)\in\Gamma_3$.} In this case we again apply \cref{lemma:semiconvexity-estimate2} to the function $\widetilde u_r(x,t):=r^{-3}\widetilde u(rx,r^2t)$, but the cubic rescaling simplifies the proof. Precisely, let $f_r(x,t):=r^{-1}f(rx,r^2t)$, then, arguing as in \textit{Step 1}, $M_r:=\|\widetilde u_r\|_{L^\infty(\mathscr{C}_1)}+\|f_r\|_{C^{2}(B_{7/8})}$ is bounded by a constant which does not depend on $r$. Moreover, $\widetilde u_r$ satisfies the estimate in \cref{lemma:monotonicity-ass} for some constant $c$ which does not depend on $r$. 
			Indeed, by the cubic rescaling, $$\partial_{t}\widetilde u_r=r^{-1}\partial_t \widetilde u(rx,r^2t)\ge c|x_{n+1}|\qquad\text{in } \mathscr C_{3/4}.$$
			Then, by a covering argument, \cref{lemma:semiconvexity-estimate2} gives a constant $C>0$ such that $$|\nabla_{x'}\widetilde u_r|\le C\partial_t\widetilde u_r\qquad\text{in }\mathscr C_{1/2},$$ for some constant $C>0$ independent of $r$. Rescaling back, we get $$|\nabla_{x'}\widetilde u|\le Cr\partial_t\widetilde u\qquad\text{in }\mathscr C_{r/2}.$$ The last estimate gives \eqref{eq:eqeqeq2}, for $\omega(r)\approx r$, concluding the proof.
		\end{proof}
		\begin{proof}[Proof of \cref{theorem:structure-spacetime}]
			We assume without loss of generality that $(x_0,t_0)=(0,-1/2)$. Let $\widetilde u:=\widetilde u^{0,-1/2}$ be defined in \eqref{eq:utilde} and recall that $\Lambda(\widetilde u)=\{\widetilde{u} =0\}'$.
			We divide the proof into several steps.
			
			\smallskip
			\noindent\textit{Step 1.} First, we claim that there exist $R>0$ and $L>0$ such that for every $(x_1,t_1)\in\Lambda(\widetilde u)\cap \mathscr{C}'_{R}$ and $(x_2,t_2)\in \overline{\{\widetilde u>0\}'}\cap \mathscr{C}'_{R}$, with $t_1> t_2$, we have \be\label{eq:claim1-strati} 0\le t_1-t_2\le L|x_1'-x_2'|.
			\ee
			In order to prove \eqref{eq:claim1-strati}, we first recall that, by \cref{lemma:semiconvexity-estimate2}, \be\label{eq:est1}|\nabla_{x'} \widetilde u|\le L\partial_t\widetilde u\qquad\text{in }\mathscr{C}_{1/4},\ee for some constant $L>0$.
			For $\eta>0$ small, we set $$\gamma_\eta(s):=\left(x_1'-s(x_1'-x_2'),\eta,t_1-sL|x_1'-x_2'|\right),$$ and we observe that $\gamma_\eta(s)\in \mathscr{C}_{1/4}$ for every $s\in[0,1]$, provided $R$ is chosen sufficiently small. Therefore, by \eqref{eq:est1}, we get $$\frac{d}{ds}\widetilde u(\gamma_\eta(s))=\nabla_{x'}\widetilde u\cdot (x_2'-x_1')-L\partial_t\widetilde u|x'_1-x'_2|\le0.$$
			Passing to the limit as $\eta\downarrow0$, we obtain that $\widetilde u\circ\gamma_0$ is non-increasing.
			It is enough to prove the claim when $\widetilde u(x_2,t_2)>0$. The general case follows by approximating $(x_2,t_2)$ with points in the positivity set of $\widetilde u$. Suppose by contradiction that $t_1-t_2>L|x_2'-x_1'|$, then, by the monotonicity of $\widetilde u$ and $\widetilde u\circ \gamma$, we obtain $$0<\widetilde u(x_2,t_2)\le \widetilde u(x_2,t_1-L|x_2'-x_1'|)=\widetilde u(\gamma_0(1))\le \widetilde u(\gamma_0(0))=\widetilde u(x_1,t_1)=0,$$ a contradiction. This concludes the proof of the claim \eqref{eq:claim1-strati}.
			
			\smallskip
			\noindent\textit{Step 2.} Next we prove that there exists $\rho>0$ and there exist two times $T_+>0$ and $T_-<0$ such that \be\label{eq:claim2-strati}\widetilde u(x,T_+)>0\quad\text{and}\quad \widetilde u(x,T_-)=0\qquad\forall x\in B'_\rho.\ee 
			Moreover, we will choose $\rho$ such that $L\rho\le R^2/4$.
			
			To prove \eqref{eq:claim2-strati}, we use \eqref{eq:claim1-strati}.
			First, since $(0,0)$ is a free boundary point for $\widetilde u$, then there exists $(y,s)\in \mathscr{C}_R'$ close to $(0,0)$ such that $\widetilde u(y,s)>0$. We can also take $s>0$, by monotonicity of $\widetilde u$. We set $T_+:=s + L(\rho+|y|)>0$ and note that, if $(y,s)$ are sufficiently close to $(0,0)$, then $T_+<R^2$. We prove that $\widetilde u(x,T_+)>0$ for every $x\in B_\rho'$, provided $\rho>0$ is chosen sufficiently small. By contradiction, suppose that $\widetilde u(\overline x,T_+)=0$ for some $\overline x\in B_\rho'$. Since $T_+>s$, we can apply \eqref{eq:claim1-strati} to deduce that $$L(\rho+|y|)= T_+-s\le L|\overline x-y|< L(\rho+|y|),$$ a contradiction. This implies that $\widetilde u(x,T_+)>0$ for every $x\in B_\rho'$.
			
			Similarly, we set $T_-:=-L\rho<0$ and note that, if $\rho$ is chosen small enough, then $T_->-R^2$. We prove that $\widetilde u(x,T_-)=0$ for every $x\in B_\rho'$. By contradiction, suppose that $\widetilde u(\overline x,T_-)>0$ for some $\overline x\in B_\rho'$. Since $\widetilde u(0,0)=0$ and $T_-<0$, by \eqref{eq:claim1-strati}, $$L\rho= 0-T_-\le L|\overline x|< L\rho,$$ a contradiction. This implies that $\widetilde u(x,T_-)=0$ for every $x\in B_\rho'$, concluding the proof of \eqref{eq:claim2-strati}.
			
			\smallskip
			\noindent\textit{Step 3.} We consider the function $\tau:B'_\rho\to\R$ defined as $$\tau(x',0):=\sup\{t\in[T_-,T_+]:\ \widetilde u(x',0,t)=0\}.$$ By continuity and time monotonicity of $\widetilde u$, the supremum is attained, and $$\widetilde u(x',0,t)=0 \quad\text{for }t\le\tau(x',0)\qquad\text{and}\qquad \widetilde u(x',0,t)>0 \quad\text{for }t>\tau(x',0).$$ 
			  By applying \eqref{eq:claim1-strati} first to the points $(x_1',0,\tau(x_1',0))$, $(x_2',0,\tau(x_2',0))$, and then reversing the order and reapplying \eqref{eq:claim1-strati}, we obtain that the function $\tau$ is $L$-Lipschitz \bea\label{eq:lip-reg-tau}|\tau(x_1',0)-\tau(x_2',0)|\le L|x_1'-x_2'|.\eea
			By decreasing $\rho$ so that $\mathscr{C}_\rho=B_\rho\times(-\rho^2,\rho^2)\subseteq B_\rho\times(T_-,T_+)$, by \eqref{eq:claim2-strati} we have $$\Lambda(\widetilde u)\cap \mathscr{C}'_\rho=\left\{(x',0,t)\in \mathscr{C}'_\rho: \ t\le \tau(x',0)\right\},\qquad\Gamma(\widetilde u)\cap \mathscr{C}'_\rho=\left\{(x',0,t)\in \mathscr{C}'_\rho: \ t=\tau(x',0)\right\}.$$ 
			It remains to prove that $\tau\in C^1(B_\rho')$ and that its gradient vanishes precisely at singular
			points. Then, we pick a point $x_0\in B_\rho'$ and set $t_0=\tau(x_0',0)$. We need to prove that $\nabla_{x'} \tau$ is continuous at $x_0$.
			
			\smallskip
			\noindent\textit{Step 4.}
			First suppose that $(x_0,t_0)$ is a regular point. Therefore, by \cite[Corollary 3.3]{petrosyan2019boundedness}, up to a rotation, the free boundary of $\widetilde u$ has the form $x_n=g(x'',t)$, $x''=(x_1,\ldots,x_{n-1})$, for some $g$ of class $C^{1,\alpha}$ in the space-time variables. 
			We write $(x_0,t_0)=( x_0'',g( x_0'', t_0),0,t_0)$. For $\overline t$ sufficiently close to $t_0$, we take the free boundary point $( x_0'',g( x_0'',\overline t),0,\overline t)$. By \eqref{eq:claim1-strati}, we get $$|\overline t-t_0|\le L|g( x_0'', \overline t)-g( x_0'', t_0)|.$$ As a consequence, since $g$ is differentiable in the $t$ variable, then $|\partial_t g(x_0'',t_0)|\ge L^{-1}>0$. 
			In particular, differentiating $g(x'',\tau(x',0))=x_n$ in the $e_n$ direction, we have $$\partial_{x_n}\tau(x_0',0)=\frac{1}{\partial_t g(x_0'',t_0)}\not=0.$$
			This implies that $\nabla_{x'}\tau(x_0',0)\not=0$ at regular points.
			Moreover, using again that $\partial_t g(x_0'',t_0)\not=0$, in a sufficiently small neighborhood of $(x_0,t_0)$, by the implicit function theorem, we have that $x_n=g(x'',t)$ if and only if $t=\widetilde \tau(x'',x_n)$, for some $\widetilde \tau$ of class $C^{1,\alpha}$. Since the free boundary is given by $t=\tau(x',0)$ and also by
			$t=\widetilde \tau(x'',x_n)$, then $\tau(x',0)=\widetilde \tau(x')$ and thus $\tau(x',0)$ is $C^{1,\alpha}$ in a neighborhood of $(x_0,t_0)$. This concludes the continuity of $\nabla_{x'}\tau$ around regular points.
			
			\smallskip
			\noindent\textit{Step 5.} Now we suppose that $(x_0,t_0)$ is a singular point.
			We proceed by approximation. After slightly decreasing $\rho$, we observe that $$\overline\eps:=\min_{x\in \overline {B_\rho'}} \widetilde u(x,T_+)>0.$$ Moreover, in $\big( B_\rho'\times(T_-,T_+)\big)\setminus\Lambda(\widetilde u)$, $\widetilde u$ is regular and $\partial_t\widetilde u>0$. Then, for $\eps\in(0,\overline\eps)$, the level sets $\{\widetilde u=\eps\}$ in $B_\rho'\times(T_-,T_+)$ can be written as $t=\tau_\eps(x',0)$, where $\tau_\eps$ is a $C^1$ function. Differentiating $\widetilde u(x',0,\tau_\eps(x',0))=\eps$ with respect to $x'$, we get $$\nabla \tau_\eps(x',0)=-\frac{\nabla_{x'}\widetilde u(x',0,\tau_\eps(x',0))}{\partial_t\widetilde u(x',0,\tau_\eps(x',0))}.$$
			Notice that, by \cref{lemma:semiconvexity-estimate2}, the functions $\tau_\eps$ are uniformly $L$-Lipschitz.
			Then, for $(x',0)\in B_r(x_0)$, we have $|\tau_\eps(x',0)-t_0|\le (L+1)r$, for $\eps$ sufficiently small depending on $r$. Therefore, by \cref{lemma:modulus-of-cont} applied at scale comparable with $r^{1/2}$, we obtain for every $\eps$ sufficiently small $$|\nabla\tau_\eps(x',0)|\le \omega(Cr^{1/2})\qquad\forall (x',0)\in B_{r}'(x_0),$$ where $\omega(Cr^{1/2})\downarrow0$ as $r\downarrow0$. Therefore $\tau_\eps \to \tau$ along a subsequence as $\eps\downarrow 0$, and so $$|\nabla\tau(x',0)|\le \omega(Cr^{1/2})\qquad\forall (x',0)\in B_{r}'(x_0).$$ Passing to the limit as $r\downarrow0$, we get $\nabla \tau(x_0',0)=0$. Moreover, the previous inequality also implies that if $x'_j\to x'_0$, then $|\nabla \tau(x_j)|\le \omega(Cr^{1/2})$ for $j$ sufficiently large. By the arbitrariness of $r$, we get $\nabla \tau(x_j)\to 0$. This is the desired continuity of $\nabla\tau$ at $x_0$. Finally, in \textit{Step 4} we proved that at regular points the gradient of $\tau$ is not zero, while it vanishes at singular points. This concludes the proof.
		\end{proof}
		
		\begin{proof}[Proof of \cref{theorem:structure-spacetime-intro}]
			It is a direct consequence of the more detailed statement in \cref{theorem:structure-spacetime}.
		\end{proof}

		\section{Uniqueness of blow-ups at cubic points}\label{sec:epiperimetric}
		In this section we prove the uniqueness of blow-up limits with a polynomial rate of convergence for solutions of \eqref{eq:parabolic-signorini-problem} at cubic free boundary points. This will be used crucially later in \cref{sec-cleaning-cubic} to show that $\Gamma_{3}$ is cleaned up with a superquadratic rate. We let 
		\bea\mathcal{P}_{3}^+:=\Big\{p:S_\infty\to\R :\ &\mathbf{H}p\le 0,\ \mathbf{H}p=0\text{ in } S_\infty^+,\  p\equiv 0 \text{ on }S_\infty',\\ 
			&Zp=3p,\  p(x',x_{n+1},t)=p(x',-x_{n+1},t)\Big\}
		\eea
		be the set of admissible cubic blow-ups. The main result of this section is the following.
		\begin{theorem}\label{thm:rate-for-u}
			Let $u$ be a solution of \eqref{eq:parabolic-signorini-problem}, with obstacle $\varphi$ satisfying \eqref{eq:ostacolo}, and assume that $(0,0)\in\Gamma_3$. Let $\widetilde u:=\widetilde u^{0,0}$ be defined as in \eqref{eq:utilde}. Then, there are $r_{0}>0$ and a unique non-zero $p\in\mathcal{P}_{3}^+$ such that
			$$\|\widetilde {u}-p\|_{L^\infty(Q_r)}\le Cr^{3+\alpha}\qquad\forall  r\in(0,r_{0}),$$ where $C>0$ depends only on $n$, $k$, $\gamma$, $\|\varphi\|_{C^{k,\gamma}(B_1')}$, $\|u\|_{L^\infty(Q_1)}$, while $\alpha \in (0,1)$ depends only on $n$.
		\end{theorem}
		This result is based on an epiperimetric inequality approach. 
		More precisely, after a conformal change of coordinates, in \cref{prop:epiperimetric-inequality} we prove a discrete decay estimate of the Weiss' energy, which can be viewed as a parabolic epiperimetric inequality. Iterating this decay, we obtain the uniqueness result in \cref{thm:rate-for-u}.
		
		Similar strategies are used to prove uniqueness of blow-ups at even frequencies, with a logarithmic rate of convergence \cite{shi18}. Regarding the odd frequencies, we refer to \cite{cv24,col2025}, where the uniqueness result was proved respectively in the elliptic Signorini problem and for the cubic second blow-up of the parabolic obstacle problem, in the top stratum. 
		\subsection{Odd frequencies in the parabolic Signorini problem}
		More generally, in this section we will study the blow-ups at points with any odd integer frequency $2m+1$, for $m\in \N\setminus\{0\}$. Moreover, the analysis will be carried out for general solutions of the parabolic Signorini problem, even without the monotonicity assumption $\partial_{t}u> 0$, that is, in the framework of \cite{dgpt17}. After subtracting the obstacle and multiplying by a cutoff according to the procedure described in \cref{sec-subtraction-obstacle}, one is led to study solutions $w$ of \eqref{eq:w-solves} with right-hand side $f$ satisfying \eqref{controlf}. 
		We define $\Gamma(w)$ as the set of free boundary points, as in \eqref{eq:def-freeboundary}.
		Similarly, as in \eqref{eq:def-gammalambda}, for $\lambda<k+\gamma$, we denote by $\Gamma_\lambda(w)$  the free boundary points of frequency $\lambda$.
		For any $m\in\N$, we let
		\bea\mathcal{P}_{2m+1}^+:=\Big\{p:S_\infty\to\R :\ &\mathbf{H}p\le 0,\ \mathbf{H}p=0\text{ in } S_\infty^+,\  p\equiv 0 \text{ on }S_\infty',\\ 
			&Zp=(2m+1)p,\  p(x',x_{n+1},t)=p(x',-x_{n+1},t)\Big\}
		\eea
		be the set of admissible $(2m+1)$-homogeneous blow-ups. \cref{thm:rate-for-u} is an immediate consequence of the following more general statement.

		\begin{proposition}\label{prop:rate-for-w}
			Let $\lambda:= 2m+1$ for some $m\in \N\setminus \{0\}$. Let $w$ be a solution of \eqref{eq:w-solves}, with right-hand side $f$ satisfying \eqref{controlf} for some $\ell>\lambda$, and assume that $(0,0) \in\Gamma_{\lambda}(w)$. Then, there are $r_{0}>0$ and a unique non-zero $p\in\mathcal P_{\lambda}^+$ such that 
			$$\|w-p\|_{L^\infty(Q_r)}\le Cr^{\lambda+\alpha}\qquad\forall  r\in(0,r_{0}),$$ where $C>0$ and $\alpha\in (0,1)$ depend only on $n$, $m$, $C_{f}$, $\ell-\lambda$, $\|w\|_{L^\infty(Q_1)}$.
		\end{proposition}

		As a byproduct of our analysis, we obtain the following result of independent interest, a frequency gap above odd frequencies for the parabolic Signorini problem.
		\begin{proposition}
			\label{prop:odd-gap}
			For any $m\in \N\setminus \{0\}$, there is a constant $c_{n,m}>0$ depending only on $n$ and $m$ such that the following holds. If $q$ is a parabolically $\lambda$-homogeneous solution of the parabolic Signorini problem with zero obstacle \eqref{eq:parabolic-signorini-problem-global} for some $\lambda\in(2m+1,2m+1+c_{n,m})$, then $q\equiv0$.
		\end{proposition}
		\subsection{Preliminary constructions} Before entering into the proofs of the main results of this section, we introduce some notation and we set up some constructions. We fix once and for all $m\in \N\setminus \{0\}$ and we call $\lambda:= 2m+1$.
		\subsubsection{Gaussian spaces} We consider the Gaussian measure in $\R^{n+1}$
		\begin{equation*}
			d\mu(y):= G(y,-1)dy= \frac{e^{-\frac{|y|^{2}}{4}}}{(4\pi)^{\frac{n+1}{2}}}dy.
		\end{equation*}
		Then, for any open set $\Omega\subseteq \R^{n+1}$ we denote by $L^{2}_{\mu}(\Omega)$ the space of functions $g:\Omega\to \R$ for which
		\begin{equation*}
			\lVert g\rVert_{L^{2}_{\mu}(\Omega)}:=\left(\int_{\Omega}g(y)^2d\mu(y)\right)^{1/2}<\infty.
		\end{equation*}
		Similarly, the Gaussian Sobolev space $W^{1,2}_{\mu}(\Omega)$ is the set of functions for which
		\begin{equation*}
			\lVert g\rVert_{W^{1,2}_{\mu}(\Omega)}:=\lVert g\rVert_{L^{2}_{\mu}(\Omega)}+\lVert \nabla g\rVert_{L^{2}_{\mu}(\Omega)}<\infty.
		\end{equation*}
		When $\Omega=\R^{n+1}_{+}$ we typically write $L^{2}_{\mu}$ and $W^{1,2}_{\mu}$ for $L^{2}_{\mu}(\R^{n+1}_{+})$ and $W^{1,2}_{\mu}(\R^{n+1}_{+})$.
		\subsubsection{Conformal coordinates}
		The conformal change of coordinates $\mathcal{T}:\R^{n+1}\times(-1,0)\to\R^{n+1}\times(0,\infty)$ is given by 
		
		$$\mathcal{T}(x,t):=\left(\frac{x}{\sqrt{-t}},-\log(-t)\right)=(y,\tau),\qquad \mathcal{T}^{-1}(y,\tau)=(e^{-\tau/2}y,-e^{-\tau})=(x,t).$$ 
		For any function $w:S_1\to\R$, we let $\widetilde{w}:\R^{n+1}\times (0,\infty)\to \R$ be
		\be\label{eq:widetildev}\widetilde w(y,\tau):=\frac{w(x,t)}{(\sqrt{-t})^\lambda}\Big|_{(x,t)=\mathcal{T}^{-1}(y,\tau)}=e^{\lambda\tau/2}w(e^{-\tau/2}y,-e^{-\tau}).\ee 
		If $w$ is a solution of \eqref{eq:w-solves} with right-hand side $f$ satisfying \eqref{controlf} for some $\ell>\lambda$, then an explicit computation shows that $\widetilde w$ solves 
		\be\label{eq:parabolic-signorini-problem-for-widetildew}\begin{cases}
			\mathcal{L}_\lambda \widetilde w-\partial_\tau \widetilde w=e^{-\tau}\widetilde f &\qquad \text{in }\R^{n+1}_{+}\times (0,\infty),\\
			\min\{-\partial_{n+1}\widetilde w,\widetilde w\}=0&\qquad \text{on }(\R^{n}\times \{0\})\times (0,\infty),
		\end{cases}\ee 
		where $\mathcal{L}_{\lambda}:= \Delta-\frac{y}{2}\cdot\nabla +\frac\lambda2$ is the shifted Ornstein-Uhlenbeck operator, and $\widetilde f$ is obtained from $f$ via the same change of variables as in \eqref{eq:widetildev}:
		\begin{equation*}
			\widetilde{f}(y,\tau):= e^{\lambda\tau/2}f(e^{-\tau/2}y,-e^{-\tau}).
		\end{equation*}
		By \eqref{controlf}, it follows that
		\be\label{controlftilde}
		|\widetilde f(y,\tau)|\le CC_{f}e^{-\frac{\tau}2(\ell-\lambda-2)}(|y|^{\ell-2}+1)\qquad\forall (y,\tau)\in\R^{n+1}\times (0,\infty),\ee
		and in particular 
		\be\label{controlftilde-L2}
		\|\widetilde f(\cdot, \tau)\|_{L^2_\mu}\le CC_{f}e^{-\frac{\tau}2(\ell-\lambda-2)}\qquad\forall \tau\in(0,\infty).\ee 
		
		\subsubsection{The Weiss' energy}
		For any sufficiently regular function $w:S_{1}^{+}\to \R$ and any $t\in (-1,0)$, we introduce the Weiss' energy for the $\lambda$-frequency
		\be\label{eq:weiss-original-coordinates}\mathcal{M}_\lambda(w,t):=\frac{1}{(\sqrt{-t})^{2\lambda-2}}\int_{\R^{n+1}_+}|\nabla w(x,t)|^2G(x,t)\,dx-\frac{\lambda}{2(\sqrt{-t})^{2\lambda}}\int_{\R^{n+1}_+}w(x,t)^2G(x,t)\,dx.\ee 
		Then, in conformal coordinates, for any $\tau\in (0,\infty)$ we write 
		$$W_\lambda(\widetilde w,\tau):=\mathcal{M}_{\lambda}(w, -e^{-\tau})=\int_{\R^{n+1}_+}|\nabla \widetilde w(y,\tau)|^2\,d\mu-\frac{\lambda}2\int_{\R^{n+1}_+}\widetilde w(y,\tau)^2\,d\mu.$$
		We recall the following monotonicity formula for the Weiss' energy. See \cite[Lemma 2.2 and Section 5]{shi18} for a proof.
		\begin{lemma}\label{lemma:shi}
			Let $w$ be a solution of \eqref{eq:w-solves}, with right-hand side $f$ satisfying \eqref{controlf} for some $\ell>\lambda$. Let $\widetilde w$ be defined as in \eqref{eq:widetildev}. Then, for every $0<\tau_1<\tau_2<\infty$ we have
			\be\label{eq:Weiss-difference-derivative}
			W_\lambda(\widetilde w,\tau_2)-W_\lambda(\widetilde w,\tau_1)\le -\int_{\tau_1}^{\tau_2}\|\partial_\tau \widetilde w(\cdot,\tau)\|^2_{L^2_\mu}\,d\tau+C_0e^{-\tau_1(\ell-\lambda)}-C_0e^{-\tau_2(\ell-\lambda)},
			\ee where $C_0:=\frac{CC_{f}^{2}}{\ell-\lambda}>0$. 
			In particular,
			\be\label{eq:Weiss-monotonicity-formula}
			\tau\mapsto W_\lambda(\widetilde w,\tau)+C_0e^{-\tau(\ell-\lambda)}
			\ee
             is monotone non-increasing.
		\end{lemma}
		\subsubsection{Classification of blow-ups with odd frequency.}
		Let us record the following characterization of blow-ups at odd frequency points.
		\begin{lemma}\label{lem:classification-odd-blow-ups}
			$p:S_{\infty}\to \R$ is a parabolically $2m+1$-homogeneous solution of the global problem \eqref{eq:parabolic-signorini-problem-global} if and only if $p\in \mathcal{P}_{2m+1}^{+}$. Moreover, in this case, $p$ can be written uniquely as
			\be\label{eq:expr-p-odd}p(x,t)=-|x_{n+1}|(q_0(x',t)+x_{n+1}^2q_1(x',x_{n+1},t)),\ee
			where $q_{0}:\R^{n}\times (-\infty,0]\to \R$ is a nonnegative parabolically $(2m)$-homogeneous polynomial and $q_{1}:S_{\infty}\to \R$ is a parabolically $(2m-2)$-homogeneous polynomial, even in $x_{n+1}$. 
		\end{lemma}
		\begin{proof}
			Any $p\in \mathcal{P}_{2m+1}^{+}$ is by definition a solution of \eqref{eq:parabolic-signorini-problem-global}. The converse implication is a direct consequence of \cite[Lemma 9.2]{frs24}. Now, given $p\in \mathcal{P}_{2m+1}^{+}$, since $\mathbf{H}p=0$ in $S_{\infty}^{+}$ and $p\equiv 0$ on $S_{\infty}'$, the odd extension $\widehat{p}(x,t):=\text{sign}(x_{n+1}) p(x,t)$ is caloric in the full $S_{\infty}$. Then, by the parabolic Liouville theorem $\widehat{p}$ is a polynomial. Imposing the Signorini conditions the unique representation in \eqref{eq:expr-p-odd} follows. 
		\end{proof}
		\noindent In view of \cref{lem:classification-odd-blow-ups}, for any $p\in\mathcal{P}_{2m+1}^+$ we define
		\bea\label{eq:operatorT}T[p](x',t):=-\partial_{n+1}p(x',0,t)=q_{0}(x',t),\eea
		where $q_{0}\ge 0$ is the parabolically $2m$-homogeneous polynomial appearing in \eqref{eq:expr-p-odd}.
		
		Since $p\in\mathcal{P}_{2m+1}^+$ is parabolically homogeneous of degree $\lambda=2m+1$, its transformation in conformal coordinates as in \eqref{eq:widetildev} is independent of the variable $\tau$, and can be identified with an element of 
		\begin{equation*}
			\begin{aligned}\mathcal{E}_{2m+1}^+:=\Big\{\widetilde p\in W^{1,2}_{\mu}(\mathbb R^{n+1}) :\ &\mathcal{L}_{2m+1}\widetilde{p}\le 0,\   \mathcal{L}_{2m+1}\widetilde{p}=0\text{ in } \R^{n+1}_{+},\\
				&\widetilde{p}\equiv 0 \text{ on } \R^{n}\times \{0\},\  \widetilde{p}(y',y_{n+1})=\widetilde{p}(y',-y_{n+1})\Big\},
			\end{aligned}
		\end{equation*}
		where $\mathcal{L}_{2m+1}= \Delta-\frac{y}{2}\cdot \nabla+ \frac{2m+1}{2}$ is the shifted Ornstein-Uhlenbeck operator. 
		Elements of $\mathcal{E}_{2m+1}^{+}$ will be called blow-ups in conformal coordinates.
		For any $\widetilde{p}\in \mathcal{E}_{2m+1}^+$, we then define the normal derivative operator as
		\begin{equation*}
			\widetilde{T}[\widetilde{p}](y'):= -\partial_{n+1}\widetilde{p}(y',0).
		\end{equation*}
		Removing the condition $\mathcal{L}_{2m+1}\widetilde{p}\le 0$ from the definition of $\mathcal{E}_{2m+1}^{+}$ we obtain the null eigenspace of $\mathcal{L}_{2m+1}$ in $\R^{n+1}_{+}$ with Dirichlet boundary conditions on $\{y_{n+1}=0\}$:
		\be\label{eq:eigenspacemathcalE}  \mathcal{E}_{2m+1}:=\left\{\widetilde p\in W^{1,2}_{\mu}(\mathbb R^{n+1}): \ \mathcal{L}_{2m+1} \widetilde p=0\text{ on }\R^{n+1}_+,\  \widetilde p\equiv 0 \text{ in }\R^{n}\times\{0\},\ \widetilde p(y',y_{n+1})= \widetilde p(y',-y_{n+1})\right\}.\ee
		The space $\mathcal E_{2m+1}$ is finite-dimensional, and $\mathcal E^+_{2m+1}$ is a closed convex cone in $\mathcal E_{2m+1}$.
		
		\subsection{An $L^2$-$L^\infty$ estimate}
		In the following lemma we prove an $L^2$-$L^\infty$ estimate in conformal coordinates for the discrepancy between a solution of \eqref{eq:w-solves} and an admissible blow-up with odd frequency.
		\begin{lemma}\label{lemma:wang}
			Let $w$ be a solution of \eqref{eq:w-solves}, with right-hand side $f$ satisfying \eqref{controlf} for some $\ell>\lambda:=2m+1$. Let $\widetilde w$ be the corresponding function in conformal coordinates as defined in \eqref{eq:widetildev}, and let $\widetilde p\in\mathcal{E}_{\lambda}^+$ be a blow-up in conformal coordinates. For any $\tau_0\in(0,\infty)$ and any $\sigma\in (0,1/2)$ we have
			$$\|\widetilde w(\cdot,\tau)-\widetilde p\|_{L^\infty(B_{\frac{1}{\sigma}})}\le C_\sigma\left(\int_{\tau_0}^{\tau_0+1}\|\widetilde w(\cdot,\tau')-\widetilde p\|^2_{L^2_\mu}\,d\tau'\right)^\frac12+C_\sigma e^{-\frac{\tau_0}2(\ell-\lambda)}\qquad\forall \tau\in(\tau_0+\sigma,\tau_0+1),$$ where $C_\sigma>0$ depends only on $n$, $m$, $C_{f}$, $\ell-\lambda$ and $\sigma$.
		\end{lemma}
		\begin{proof} 
			Since $\widetilde{w}$ solves \eqref{eq:parabolic-signorini-problem-for-widetildew}, then 
			\begin{equation*}
				(\mathcal{L}_{\lambda}-\partial_{\tau})\widetilde{w}= e^{-\tau}\widetilde{f}-\widetilde{\nu}\quad \text{in $\R^{n+1}\times (0,\infty)$},\qquad \widetilde{\nu}:= -2\partial_{n+1}\widetilde{w}\mathcal{H}^{n+1}\res \left(\{y_{n+1}=0\}\times (0,\infty)\right),
			\end{equation*}
			where $\widetilde{\nu}$ is a nonnegative measure concentrated on the contact set $\Lambda(\widetilde{w}):=\{\widetilde{w}=0\}\cap \left(\{y_{n+1}=0\}\times (0,\infty)\right)$ by the Signorini condition.
			Since $\mathcal{L}_{\lambda}\widetilde{p}\le 0$ and the bound on $e^{-\tau}\widetilde f$ obtained from \eqref{controlftilde}, we then get
			$$(\mathcal{L}_\lambda-\partial_\tau)(\widetilde w-\widetilde p)\ge -C_\sigma e^{-\frac{\tau_0}2(\ell-\lambda)}\qquad\text{in }\left(B_{\frac{2}{\sigma}}\times (\tau_{0},\infty)\right)\setminus \Lambda(\widetilde{w}).$$
			On the other hand, since $\widetilde{p}= 0$ on $\{y_{n+1}=0\}$, then
			$$\widetilde w-\widetilde p=0 \qquad\text{on }\Lambda(\widetilde{w}).$$ 
			It follows that
			\be\label{eq:positivepart-dim}(\mathcal{L}_\lambda-\partial_\tau)(\widetilde w-\widetilde p)_+\ge-C_\sigma e^{-\frac{\tau_0}2(\ell-\lambda)}\qquad\text{in }B_{\frac{2}{\sigma}}\times (\tau_{0},\infty)\ee
			in the viscosity sense. Similarly, since $\mathcal{L}_{\lambda}\widetilde{p}=0$ away from $\{y_{n+1}=0\}$,
			$$(\mathcal{L}_\lambda-\partial_\tau)(\widetilde p-\widetilde w)\ge -C_\sigma e^{-\frac{\tau_0}2(\ell-\lambda)}\qquad\text{in }\left(B_{\frac{2}{\sigma}}\setminus\{y_{n+1}=0\}\right)\times (\tau_{0},\infty).$$
			On the other hand, since $\widetilde{p}= 0$ on $\{y_{n+1}=0\}$, then
			$$\widetilde p-\widetilde w\le-\widetilde w\le0 \qquad\text{on }\{y_{n+1}=0\}\times (0,\infty).$$ 
			It follows that 
			\be\label{eq:negativepart-dim}(\mathcal{L}_\lambda-\partial_\tau)(\widetilde w-\widetilde p)_-=(\mathcal{L}_\lambda-\partial_\tau)(\widetilde p-\widetilde w)_+\ge-C_\sigma e^{-\frac{\tau_0}2(\ell-\lambda)}\qquad\text{in }B_{\frac{2}{\sigma}}\times (\tau_{0},\infty)\ee
			in the viscosity sense. Combining \eqref{eq:positivepart-dim} and \eqref{eq:negativepart-dim} we get 
			$$(\mathcal{L}_\lambda-\partial_\tau)|\widetilde w-\widetilde p|\ge-C_\sigma e^{-\frac{\tau_0}2(\ell-\lambda)}\qquad\text{in }B_{\frac{2}{\sigma}}\times (\tau_{0},\infty).$$
			By applying \cref{lemma:wang-start} together with a standard covering argument, we obtain \bea\sup_{B_{\frac{1}{\sigma}}\times(\tau_{0}+\sigma,\tau_{0}+1)}|\widetilde w-\widetilde p|&\le C_\sigma\left(\int_{\tau_{0}}^{\tau_{0}+1}\int_{B_{\frac{2}{\sigma}}}(\widetilde w(y,\tau)-\widetilde p(y))^2\,dy\,d\tau\right)^\frac12+C_\sigma e^{-\frac{\tau_0}2(\ell-\lambda)}\\
			&\le C_\sigma\left(\int_{\tau_{0}}^{\tau_{0}+1}\|\widetilde w(\cdot,\tau)-\widetilde p\|^2_{L^2_\mu(B_{\frac{2}{\sigma}})}\,d\tau\right)^\frac12+C_\sigma e^{-\frac{\tau_0}2(\ell-\lambda)},
			\eea 
			where we used that the Gaussian density is bounded below by some constant $c_\sigma$ in $B_{\frac{2}{\sigma}}$. This concludes the proof.
		\end{proof}
		\subsection{Localization of the contact set}
		In the following lemma we prove that if $w$ is sufficiently close to a blow-up profile $p\in \mathcal{P}^{+}_{\lambda}$, then the contact set of $w$ contains regions where $-\partial_{n+1}p$ is uniformly positive. 
		The result is again formulated in conformal coordinates.
		
		\begin{lemma}\label{lemma:localization} For every $\sigma\in(0,1/2)$ there are $\delta>0$ and $\overline \tau_0>0$ 
			depending on $n$, $m$, $C_{f}$, $\ell-\lambda$ and $\sigma$ such that the following holds. Let $w$ be a solution of \eqref{eq:w-solves}, with right-hand side $f$ satisfying \eqref{controlf} for some $\ell>\lambda:=2m+1$. Let $\widetilde w$ be the corresponding function in conformal coordinates as defined in \eqref{eq:widetildev}, and let $\widetilde p\in\mathcal{E}_{\lambda}^+$ be a blow-up in conformal coordinates satisfying $\|\widetilde p\|_{L^2_\mu}=1$. Suppose that, for some $\tau_0\in(\overline\tau_0,\infty)$ we have
			\be\label{eq:closeness-assumption-in-lemma} |\widetilde w-\widetilde p|\le \delta\qquad\text{in }B_{\frac1\sigma}\times (\tau_0+\sigma,\tau_0+1).
			\ee Then
			$$\widetilde w(y',0,\tau)=0\qquad\forall  y'\in \left\{\widetilde T[\widetilde p\,]\ge\sigma\right\},\ |y'|<\frac{1}{2\sigma}, \quad \forall \tau\in (\tau_0+2\sigma,\tau_0+1).$$
		\end{lemma}
		\begin{proof}
			Let $p\in \mathcal{P}_{\lambda}^{+}$, defined as $p(x,t):= (-t)^{\lambda/2}\widetilde{p}(x/\sqrt{-t})$ be the corresponding blow-up in physical coordinates. We also call $w_{r}(x,t):= r^{-\lambda}w(rx, r^{2}t)$ the parabolic rescaling of $w$, for $r\in (0,1)$. We divide the proof into two steps.
			
			\smallskip
			\noindent\textit{Step 1.} We first prove that there are $\overline\delta>0$, $\overline r>0$ and $\kappa\in(0,1)$ such that if $$|w_r(x,t)-p(x,t)|\le \overline\delta\qquad\forall (x,t)\in B_{\frac1\sigma}\times (-e^{\kappa\sigma},-1),$$ for some $r\in(0,\overline r)$, then 
			\be\label{eq:conclusion-partial}w_r(x',0,-1)=0\qquad\forall x'\in \{T[p](\cdot,-1)\ge\sigma\}, \ |x'|<\frac{1}{2\sigma}.\ee
			
			Let $z'\in \{T[p](\cdot,-1)\ge\sigma\}$ with $|z'|<\frac{1}{2\sigma}$, and consider the set 
			$$U:=\left\{(x,t)\in \R^{n+1}\times \R: \,|x'-z'|< \rho,\ |x_{n+1}|<\rho,\ t\in(-e^{\kappa\sigma},-1]\right\},$$
			for some $\rho>0$ and $\kappa>0$ sufficiently small. Since $\mathcal{P}_{\lambda}^{+}$ is contained in a finite-dimensional polynomial space (see \cref{lem:classification-odd-blow-ups}) and $\widetilde{p}$ is normalized, there is $C>0$ such that
			$$|T[p](x',t)-T[p](z',-1)|\le C(\rho+\kappa\sigma)\qquad \forall  (x,t)\in U.$$ 
			Then, by the representation \eqref{eq:expr-p-odd}, for any $(x,t)\in U$ we have
			\bea p(x,t)&\le -|x_{n+1}|T[p](x',t)+C|x_{n+1}|^3\\
			&\le -|x_{n+1}|T[p](z',-1)+C|x_{n+1}|(\rho+\kappa\sigma)+C|x_{n+1}|^3\\
			&\le |x_{n+1}|(-\sigma +C\rho+C\kappa\sigma)+C|x_{n+1}|^3.\eea 
			We first choose $\kappa$ sufficiently small depending on $C$ and then $\rho$ sufficiently small depending on $\sigma$, such that $-\sigma +C\rho+C\kappa\sigma\le -\sigma/2$.  Then, after further reducing $\rho$  and choosing $\overline\delta$ small enough depending on $\rho$, for every $(x,t)\in \partial_{\text{par}}U$ we get
			\bea w_r(x,t)&\le p(x,t)+\overline\delta\le -\frac\sigma 2|x_{n+1}|+C|x_{n+1}|^3+\overline\delta\\
			&\le\begin{cases}
				-\frac\sigma 2 \rho +C\rho^3+\overline\delta\le -(n+2)\rho^2\qquad &\text{if $|x_{n+1}|=\rho$},\\
				-\frac\sigma 2 |x_{n+1}|+C\rho^3+\overline\delta\le -(n+2)x_{n+1}^2+\rho^2\qquad &\text{if $|x_{n+1}|<\rho$ and $|x'-z'|=\rho$},\\
				C\rho^3+\overline\delta\le -(n+2)\rho^2+2(e^{\kappa\sigma}-1)\qquad &\text{if $|x_{n+1}|<\rho$ and $t=-e^{\kappa\sigma}$}.\\
			\end{cases} \eea
			It follows that for any $c\ge 0$, 
			$$w_r(x,t)\le -(n+2)x_{n+1}^2+|x'-z'|^2-2(1+t)+c=:\phi_c(x,t)\qquad\forall (x,t)\in\partial_{\text{par}}U.$$
			
			Now, for $c$ large enough, we have $w_r(x,t)< \phi_c(x,t)$ in $\overline U$. Then, we can decrease $c$ until a contact point occurs in $\overline U$. Since $w_{r}(z',0,-1)\ge 0$, we see that such a contact point must occur for some value $c_\ast\ge0$. We claim that $c_{\ast}=0$. Given the claim one gets $w_r\le\phi_0$ in $\overline U$, which in turn implies \eqref{eq:conclusion-partial} at the point $(z',0,-1)$ since $\phi_0(z',0,-1)=0$. This would conclude \textit{Step 1} by the arbitrariness of $z'$.
			
			If $c_{\ast}>0$, then, by what we proved above, the contact point $(x_0,t_0)$ lies in $\overline U\setminus \partial_{\text{par}}U$. Note that $\phi_{c_\ast}>0$ in the points of $U$ belonging to the contact set of $w_{r}$, so $(x_{0},t_{0})$ lives outside this contact set. Then, by the bound on the right-hand side $f$ in \eqref{controlf}, up to choosing $\overline r$ small enough depending on $\sigma$ we have 
			$$\mathbf{H}w_r(x_{0},t_{0})=r^{2-\lambda}f(rx_{0}, r^{2}t_{0})\ge -C_\sigma r^{\ell-\lambda}>- 2=\mathbf{H}\phi_{c_\ast}(x_{0},t_{0})\qquad\forall r\in (0,\overline r).$$ This contradicts the fact that $\phi_{c_{\ast}}$ touches $w_{r}$ from above at $(x_{0},t_{0})$. Therefore, the case $c_\ast>0$ is impossible.

			\smallskip
			\noindent\textit{Step 2.} To conclude, we transfer the result proved in \textit{Step 1} to conformal coordinates. Note that we have 
			\be\label{eq:wr-and-widetildew}w_r(x,t)={(\sqrt{-t})^\lambda} \widetilde w\left(\frac{x}{\sqrt{-t}},-\log(-t r^2)\right),\qquad p(x,t)= \sqrt{-t}^{\lambda}\widetilde{p}\left(\frac{x}{\sqrt{-t}}\right).\ee
			Moreover, for any 
			\begin{equation}\label{eq:choice-r-localization-contact-set}
				r\in(e^{-\frac{\tau_0+1}2},e^{-\frac{\tau_0+(1+\kappa)\sigma}2}),    
			\end{equation}
			we get 
			$$\left(\frac{x}{\sqrt{-t}}, -\log(-t r^2)\right) \in B_{\frac1\sigma}\times (\tau_0+\sigma, \tau_0+1)\qquad \forall  (x,t)\in B_{\frac{1}{\sigma}}\times (-e^{\kappa\sigma}, -1).$$ 
			Therefore, combining hypothesis \eqref{eq:closeness-assumption-in-lemma} with \eqref{eq:wr-and-widetildew}, we obtain
			$$|w_r(x,t)-p(x,t)|\le (\sqrt{-t})^\lambda\delta\le e^{\frac{\kappa\sigma\lambda}{2}}\delta\qquad\forall (x,t)\in B_{\frac1\sigma}\times (-e^{\kappa\sigma},-1).$$
			Now we choose $\delta$ sufficiently small so that $e^{\frac{\kappa\sigma\lambda}{2}}\delta\le \overline\delta$, and $\overline \tau_0$ sufficiently large so that $e^{-\frac{\overline{\tau}_{0}}{2}}\le \overline r$. Then, we may apply \textit{Step 1} and get \eqref{eq:conclusion-partial} for any $r$ as in \eqref{eq:choice-r-localization-contact-set}. By \eqref{eq:wr-and-widetildew}, this implies
			$$\widetilde w(y',0,\tau)=0\qquad\forall y'\in \{\widetilde T[\widetilde p\,]\ge\sigma\}, \ |y'|<\frac{1}{2\sigma}, \quad \forall \tau \in (\tau_0+(1+\kappa)\sigma,\tau_0+1).$$
			Since $\kappa\in (0,1)$, we have $(\tau_0+2\sigma,\tau_0+1)\subseteq(\tau_0+(1+\kappa)\sigma,\tau_0+1)$ and the conclusion follows.
		\end{proof}
		
		\subsection{Epiperimetric inequality for odd frequencies}
		In this subsection, we prove the main technical result of this part, an epiperimetric inequality for odd frequencies, for the modified Weiss' energy 
		\be\label{eq:widetildeWeiss}\widehat W_\lambda(\widetilde w,\tau):=W_\lambda(\widetilde w,\tau)+2C_0e^{-\tau(\ell-\lambda)},\ee where $C_0$ is the constant in \cref{lemma:shi} and $\ell$ is the constant in \eqref{controlf}.
		Our main results of this section, \cref{prop:rate-for-w} and \cref{prop:odd-gap}, will follow from the next proposition.
		\begin{proposition}\label{prop:epiperimetric-inequality}
			There are constants $\eps_0>0$ and $\delta_0>0$ depending only on $n$, $m$, $C_{f}$ and $\ell-\lambda$ such that the following holds.
			Let $w$ be a solution of \eqref{eq:w-solves}, with right-hand side $f$ satisfying \eqref{controlf} for some $\ell>\lambda:=2m+1$. Let $\widetilde w$ be the corresponding function in conformal coordinates as defined in \eqref{eq:widetildev}, and let $\widetilde p\in\mathcal{E}_{\lambda}^+$ be a blow-up in conformal coordinates such that $\|\widetilde p\|_{L^2_\mu}=1$.
			Suppose that, for some $\tau\in(0,\infty)$, we have
			\be\label{eq:closeness-assumption} \|\widetilde w(\cdot,\tau')-\widetilde p\|_{L^2_\mu}\le \delta_0 \quad \text{and}\quad 0\le \widehat W_{\lambda}(\widetilde w,\tau')\le 1\qquad\forall \tau'\in(\tau,\tau+1).
			\ee 
			Then
			$$\widehat W_{\lambda}(\widetilde w,\tau+1)\le (1-\eps_0)\widehat W_{\lambda}(\widetilde w,\tau),$$ where $\widehat W_{\lambda}$ is the modified Weiss' energy, defined in \eqref{eq:widetildeWeiss}.
		\end{proposition}
		
		\begin{proof}
			We divide the proof into three steps.
			
			\smallskip
			\noindent\textit{Step 1: Contradiction assumptions.}
			Suppose, for a contradiction, that there are solutions $w_j$ of \eqref{eq:w-solves}, with right-hand sides $f_j$ satisfying \eqref{controlf} with common constants $C_f$ and $\ell$, times $\tau_j\in(0,\infty)$, blow-ups in conformal coordinates $\widetilde p_j\in\mathcal{E}_{\lambda}^+$, $\|\widetilde p_j\|_{L^2_\mu}=1$ and constants $\eps_j\in(0,1/2)$, $\delta_j>0$, with $\eps_j\downarrow 0$, $\delta_j\downarrow 0$, such that 
			\be\label{eq:closeness-assumption-contradiction-ass}\|\widetilde w_j(\cdot,\tau')-\widetilde p_j\|_{L^2_\mu}\le\delta_j\quad \text{and}\quad 0\le\widehat W_\lambda(\widetilde w_j,\tau')\le 1\qquad\forall \tau'\in(\tau_j,\tau_j+1),\ee
			and
			$$\widehat W_{\lambda}(\widetilde w_j,\tau_j+1)>(1-\eps_j)\widehat W_{\lambda}(\widetilde w_j,\tau_j).$$ 
			Here $\widetilde w_j$ is obtained from $w_j$ through the transformation \eqref{eq:widetildev}.
			If $I_j:=[\tau_j,\tau_j+1]$, we have 
			$$\widehat W_{\lambda}(\widetilde w_j,\tau_j+1)-\widehat W_{\lambda}(\widetilde w_j,\tau_j)>\widehat W_{\lambda}(\widetilde w_j,\tau_j+1)-\frac{\widehat W_{\lambda}(\widetilde w_j,\tau_j+1)}{1-\eps_j}
			=\frac{-\eps_j}{1-\eps_j}\widehat W_{\lambda}(\widetilde w_j,\tau_j+1).$$
			Then, using the monotonicity formula \eqref{eq:Weiss-difference-derivative} and the definition of $\widehat{W}_{\lambda}$, we get
			\bea
			\int_{I_j}\|\partial_\tau \widetilde w_j\|_{L^2_\mu}^2+C_0e^{-\tau_j(\ell-\lambda)}-C_0e^{-(\tau_{j}+1)(\ell-\lambda)}&\le \widehat W_{\lambda}(\widetilde w_j,\tau_j)-\widehat W_{\lambda}(\widetilde w_j,\tau_j+1)
			\\&\le\frac{\eps_j}{(1-\eps_j)}\widehat W_{\lambda}(\widetilde w_j,\tau_j+1)\\&\le2\eps_j\int_{I_j}\widehat W_{\lambda}(\widetilde w_j,\tau)\,d\tau,\eea where in the last inequality we used again the monotonicity of $\widehat W_{\lambda}$ and that $\eps_j\le1/2$. Moreover,
            $$C_0e^{-\tau_j(\ell-\lambda)}-C_0e^{-(\tau_{j}+1)(\ell-\lambda)}\ge 2c_1C_0e^{-\tau_j(\ell-\lambda)},$$ for some $c_1>0$ depending only on $\ell-\lambda$. Then, using the definition of $\widehat W_{\lambda}$ in the right-hand side above, we obtain 
			\be\label{eq:proofepi-contrad-ass}\int_{I_j}\|\partial_\tau \widetilde w_j\|_{L^2_\mu}^2+c_1{C_0}e^{-\tau_j(\ell-\lambda)}\le2\eps_j\int_{I_j} W_{\lambda}(\widetilde w_j,\tau)\,d\tau,
			\ee for $j$ sufficiently large.
			
			We consider the projection of the time average $\int_{I_{j}}\widetilde w_j(\cdot,\tau)\,d\tau$ onto the eigenspace $\mathcal{E}_{\lambda}$ defined in \eqref{eq:eigenspacemathcalE}. Precisely, we let $\widetilde q_j\in\mathcal{E}_{\lambda}$ be such that 
			$$\int_{I_j}\|\widetilde w_j(\cdot,\tau)-\widetilde q_j\|_{L^2_\mu}^2\,d\tau=\min\left\{\int_{I_j}\|\widetilde w_j(\cdot,\tau)-\widetilde q\|_{L^2_\mu}^2\,d\tau:\ \widetilde q\in\mathcal{E}_\lambda\right\}.$$
			Let $\widetilde v_j(\cdot,\tau):= \widetilde w_j(\cdot,\tau)-\widetilde q_j(\cdot)$.
			Since $\mathcal{E}_{\lambda}$ is a vector space, we have the orthogonality condition
			\be\label{eq:proofepi-orthogonality-condition}\int_{I_j}\int_{\R^{n+1}_+} \widetilde v_j(y,\tau)\widetilde q(y)\,d\mu(y)\,d\tau =0\qquad\forall \widetilde q\in\mathcal{E}_{\lambda}.\ee
			
			\smallskip
			\noindent\textit{Step 2: Estimate for the $\tau$ derivative and the spatial gradient.} In this step we show the following estimates:
			\be\label{eq:proofepi-step3}\int_{I_j}\|\partial_\tau \widetilde w_j\|_{L^2_\mu}^2+c_{1}C_{0}e^{-\tau_j(\ell-\lambda)}\le C\eps_j\int_{I_j}\|\widetilde v_j\|_{L^2_\mu}^2\qquad\text{and}\qquad \int_{I_j}\|\nabla \widetilde v_j\|_{L^2_\mu}^2\le C\int_{I_j}\|\widetilde v_j\|_{L^2_\mu}^2.\ee
			In order to prove the first estimate in \eqref{eq:proofepi-step3}, we show an upper bound for $W_{\lambda}(\widetilde w_j,\tau)$ in terms of $\|\widetilde v_j\|_{L^2_\mu}^2$, and then use \eqref{eq:proofepi-contrad-ass}. Since $\widetilde q_j\in\mathcal{E}_\lambda$, an integration by parts gives $W_{\lambda}(\widetilde q_j,\cdot)\equiv0$. Then, expanding the definition of Weiss' energy, we get
			\be\label{eq:proofepi-identity}\begin{aligned} W_{\lambda}(\widetilde w_j,\tau)&=W_{\lambda}(\widetilde w_j,\tau)-W_{\lambda}(\widetilde q_j,\tau)=-W_{\lambda}(\widetilde v_j,\tau)+2\int_{\R^{n+1}_+}\left(\nabla \widetilde w_j\cdot \nabla \widetilde v_j-\frac{\lambda}2\widetilde w_j\widetilde v_j\right)\,d\mu\\&=
				-W_{\lambda}(\widetilde v_j,\tau)-2\int_{\R^{n+1}_+}\widetilde v_j\mathcal{L}_{\lambda}\widetilde w_j\,d\mu-2\int_{\R^{n+1}_0}\widetilde v_j\partial_{n+1}\widetilde w_j\,G(\cdot,-1).\end{aligned}\ee
			We want to integrate the last identity over $I_j$ and use \eqref{eq:proofepi-contrad-ass}. Thus we need to estimate the three terms in the right-hand side of \eqref{eq:proofepi-identity}. 
			We first observe that, by \eqref{eq:proofepi-contrad-ass} and the bounds for $\widetilde f_{j}$ in \eqref{controlftilde-L2}, we have
			
			\bea\label{eq:proofepi-trick}\begin{aligned}2\left|\int_{I_j}\int_{\R^{n+1}_+}\widetilde v_j(\partial_\tau\widetilde w_j+e^{-\tau}\widetilde f_{j})\,d\mu\,d\tau \right|&\le \int_{I_j}\|\widetilde v_j\|^2_{L^2_\mu}+\int_{I_j}\|\partial_\tau\widetilde w_j+e^{-\tau}\widetilde f_{j}\|^2_{L^2_\mu}\\&\le\int_{I_j}\|\widetilde v_j\|^2_{L^2_\mu}+2\int_{I_j}\|\partial_\tau\widetilde w_j\|^2_{L^2_\mu}+2C C_f^2\int_{I_j}e^{-\tau(\ell-\lambda)}\,d\tau
				\\&\le \int_{I_j}\|\widetilde v_j\|^2_{L^2_\mu}+C\eps_j\int_{I_j} W_{\lambda}(\widetilde w_j,\tau)\,d\tau.\end{aligned}\eea 
			Then, using the equation for $\widetilde w_j$ from \eqref{eq:parabolic-signorini-problem-for-widetildew}, we get the following estimate for the second term in the right-hand side of \eqref{eq:proofepi-identity}: 
			\be\label{eq:proofepi2}2\left|\int_{I_j}\int_{\R^{n+1}_+}\widetilde v_j\mathcal{L}_{\lambda}\widetilde w_j\,d\mu\,d\tau\right| \le \int_{I_j}\|\widetilde v_j\|^2_{L^2_\mu}+C\eps_j\int_{I_j} W_{\lambda}(\widetilde w_j,\tau)\,d\tau.\ee
			For the third term in the right-hand side of \eqref{eq:proofepi-identity}, since $\widetilde w_j\partial_{n+1}\widetilde w_j\equiv \widetilde q_j\equiv0$ on $\R^{n+1}_0$, then \be\label{eq:proofepi3}\int_{\R^{n+1}_0}\widetilde v_j\partial_{n+1}\widetilde w_j=\int_{\R^{n+1}_0}(\widetilde w_j-\widetilde q_j)\partial_{n+1}\widetilde w_j=0.\ee
			Combining \eqref{eq:proofepi2} and \eqref{eq:proofepi3}, then \eqref{eq:proofepi-identity} gives
			\be\label{eq:utiledopo} \int_{I_j} W_{\lambda}(\widetilde w_j,\tau)\,d\tau\le -\int_{I_j}W_{\lambda}(\widetilde v_j,\tau)\,d\tau+\int_{I_j}\|\widetilde v_j\|_{L^2_\mu}^2+C\eps_j\int_{I_j} W_{\lambda}(\widetilde w_j,\tau)\,d\tau.\ee
			Absorbing the last term above into the left-hand side for $j$ large enough, and using that \bea\label{eq:proofepi1}-\int_{I_j}W_{\lambda}(\widetilde v_j,\tau)\,d\tau\le C\int_{I_j}\|\widetilde v_j\|_{L^2_\mu}^2,\eea
			we obtain
			$$\int_{I_j} W_{\lambda}(\widetilde w_j,\tau)\,d\tau\le C\int_{I_j}\|\widetilde v_j\|_{L^2_\mu}^2.$$
			Together with \eqref{eq:proofepi-contrad-ass}, this proves the first estimate in \eqref{eq:proofepi-step3}.
			
			Similarly, absorbing the last term into the left-hand side of \eqref{eq:utiledopo} and using that $\widehat W_\lambda(\widetilde w_j,\tau)\ge0$, we have $$\int_{I_j}W_{\lambda}(\widetilde v_j,\tau)\,d\tau\le C\int_{I_j}\|\widetilde v_j\|_{L^2_\mu}^2+2C_0\int_{I_j}e^{-\tau(\ell-\lambda)}\,d\tau\le C\int_{I_j}\|\widetilde v_j\|_{L^2_\mu}^2,$$ where  in the last inequality we used the first estimate in \eqref{eq:proofepi-step3}. Finally, by definition of $W_\lambda$, we conclude the second estimate in \eqref{eq:proofepi-step3}. 
			
			\smallskip
			\noindent\textit{Step 3: Compactness and conclusion.} 
			We define the linearized sequence 
			$$\hat v_j(\cdot,\tau):=\frac{\widetilde v_j(\cdot,\tau_j+\tau)}{\eta_j}\quad \forall \tau\in(0,1),\qquad \eta_j^2:=\int_{I_j}\|\widetilde v_j\|_{L^2_\mu}^2=\int_{0}^1\|\widetilde v_j(\cdot,\tau_j+\tau)\|_{L^2_\mu}^2\,d\tau.$$ 
			By the minimizing property of $\widetilde q_j$, using $\widetilde p_j$ as a competitor, we have $\eta_j\le\delta_j$. Moreover, $\eta_j>0$ for all sufficiently large $j$: otherwise $\widetilde w_j(\cdot, \tau)=\widetilde q_j$ is stationary and $W_\lambda(\widetilde w_j,\tau)=0$ for $\tau \in [\tau_{j},\tau_{j}+1]$, contradicting the assumed failure of the decay inequality for $\widehat{W}_{\lambda}(\widetilde w_j, \cdot)$.
			We also set the corresponding linearized right-hand side
			$$\hat f_j(\cdot,\tau):=\frac{e^{-(\tau_j+\tau)}\widetilde f_{j}(\cdot,\tau_j+\tau)}{\eta_j}.$$
			Then $\hat v_j$ solves $$ \mathcal{L}_\lambda\hat v_j-\partial_{\tau}\hat v_j=\hat f_j\qquad\text{in }\R^{n+1}_+\times(0,1).$$
			Note that \eqref{eq:proofepi-step3} gives uniform boundedness of $\hat v_j$ in $W^{1,2}(\R^{n+1}_{+}\times(0,1), d\mu \,d\tau)$, since $\partial_{\tau}\widetilde v_j=\partial_{\tau}\widetilde w_j$. Therefore, up to extracting a subsequence, we can find a function $\hat{v}\in W^{1,2}(\R^{n+1}_{+}\times(0,1), d\mu \,d\tau)$ such that 
			$$\hat v_j\rightharpoonup \hat v\qquad \text{in } W^{1,2}(\R^{n+1}_{+}\times(0,1), d\mu \,d\tau),$$ 
			with strong convergence in $L^2(\R^{n+1}_{+}\times(0,1), d\mu \,d\tau)$.
			Moreover, by the first estimate in \eqref{eq:proofepi-step3}, we have $\partial_{\tau}\hat{v}_{j}\to 0$, and hence we actually get that $\partial_{\tau}\hat v\equiv0$ in $\R^{n+1}\times (0,1)$. We thus identify $\hat{v}$ with a function in $W^{1,2}_\mu$ such that $\|\hat v\|_{L^2_\mu}=1$, by the strong $L^{2}(d\mu\,d\tau)$-convergence.
			
			The orthogonality condition \eqref{eq:proofepi-orthogonality-condition} passes to the limit, and we get
			$$\int_{\R^{n+1}_+} \hat v\widetilde q\,d\mu =0\qquad\forall \widetilde q\in\mathcal{E}_{\lambda},$$ since both $\hat v$ and $\widetilde q$ are independent of $\tau$.
			In order to conclude the proof, we need to show that $\hat v\in\mathcal{E}_\lambda$, which implies $\hat{v}\equiv 0$, contradicting $\|\hat v\|_{L^2_\mu}=1$.
			In order to show $\hat v\in\mathcal{E}_\lambda$, observe first that, by the $L^{2}_{\mu}$-bound on $\widetilde f_{j}$ from \eqref{controlftilde-L2} and the first inequality in \eqref{eq:proofepi-step3} we have 
			$$\|\hat f_j(\cdot,\tau)\|_{L^2_\mu}\le \frac{Ce^{-\frac{\tau_j}2(\ell-\lambda)}}{\eta_j}\le C\eps_j^{1/2}\qquad\forall \tau\in(0,1).$$
			Then $\hat v$ is a solution of \bea\label{eq:proofepi-stationarysol}\mathcal{L}_{\lambda}\hat v=0\qquad\text{in }\R^{n+1}_+.\eea
			Thus, it only remains to prove that $\hat v\equiv 0$ on $\R^{n+1}_0$. Fix $\sigma\in (0,1/2)$, and let $\delta>0$ and $\overline{\tau}_{0}>0$ be the associated constants given by \cref{lemma:localization}. Then, by \cref{lemma:wang}, we have 
			$$\|\widetilde w_j(\cdot,\tau)-\widetilde p_j\|_{L^\infty(B_{\frac1\sigma})}\le C_\sigma\left(\int_{I_j}\|\widetilde w_j(\cdot,\tau')-\widetilde p_j\|^2_{L^2_\mu}\,d\tau'\right)^\frac12+C_\sigma e^{-\frac{\tau_j}2(\ell-\lambda)}\qquad\forall \tau\in(\tau_j+\sigma,\tau_j+1).$$ 
			In particular, \eqref{eq:closeness-assumption-contradiction-ass} and the first inequality in \eqref{eq:proofepi-step3} give 
			$$|\widetilde w_j-\widetilde p_j|\le C_\sigma (\delta_j+C\eps_j^{1/2}\eta_j) \qquad\text{in }B_{\frac1\sigma}\times (\tau_j+\sigma,\tau_j+1).$$ 
			Moreover, since $\mathcal{E}^+_\lambda$ is a closed cone in a finite-dimensional space, there is $\widetilde p\in \mathcal{E}^+_\lambda$ with $\|\widetilde{p}\|_{L^2_\mu}=1$ such that, up to a subsequence, $\widetilde p_j\to\widetilde p$ locally uniformly. 
			Thus, for sufficiently large $j$ we have 
			\bea |\widetilde w_j-\widetilde p|\le \delta\qquad\text{in }B_{\frac1\sigma}\times (\tau_j+\sigma,\tau_j+1).\eea
			Thanks to the first inequality in \eqref{eq:proofepi-step3}, $\tau_{j}\to \infty$ and so $\tau_{j}>\overline{\tau}_{0}$ for sufficiently large $j$. We can then apply \cref{lemma:localization} to get 
			$$\widetilde w_j(y',0,\tau)=0\qquad\forall  y'\in \left\{\widetilde T[\widetilde p\,]\ge\sigma\right\},\ |y'|<\frac{1}{2\sigma},\quad \forall  \tau\in (\tau_j+2\sigma,\tau_j+1).$$ Since $\widetilde q_j\equiv0$ in $\R^{n+1}_0$, then, for every $j$ large enough $$\hat v_j(y',0,\tau)=0\qquad\forall  y'\in \left\{\widetilde T[\widetilde p\,]\ge\sigma\right\},\ |y'|<\frac{1}{2\sigma}, \quad \forall  \tau\in (2\sigma,1).$$ 
			Passing to the limit for the thin traces we get 
			$$\hat v(y',0)=0\quad\forall  y'\in \left\{\widetilde T[\widetilde p\,]\ge\sigma\right\},\, |y'|<\frac{1}{2\sigma}.$$ By the arbitrariness of $\sigma$, we obtain that $\hat{v}(y',0)=0$ for every $y'\in \{\widetilde{T}[\widetilde{p}]>0\}$. Since $\widetilde{p}$ is not identically zero, the set $\{\widetilde T[\widetilde p\,]=0\}\times \{0\}$ has zero capacity in $\R^{n+1}$. This implies that $\hat v\equiv0$ on $\R^{n+1}_0$, as desired. The proof is thus concluded.
		\end{proof}
		
		\begin{remark}\label{rmk:epi-zero-rhs}
			When $f\equiv0$, the same conclusion of \cref{prop:epiperimetric-inequality} holds for the unmodified Weiss' energy $W_\lambda$, with constants depending only on $n$ and $m$. To establish this, it is sufficient to take $C_{0}=0$ in the proof, after deleting the forcing terms. Note that in this case $\tau_j\to\infty$ is no longer needed at the end of the proof, because one may take $\overline\tau_0=0$ in \cref{lemma:localization}.
		\end{remark}
		\subsection{Proof of the uniqueness of blow-ups at odd frequency points.}
		In this subsection we prove uniqueness of the blow-up at odd frequency points from \cref{prop:rate-for-w}. This is obtained as a consequence of the epiperimetric inequality in \cref{prop:epiperimetric-inequality}.
		
		\begin{proof}[Proof of \cref{prop:rate-for-w}]
			We divide the proof into three steps.
			
			\smallskip
			\noindent \textit{Step 1.}
			We first prove that there is $\delta_1>0$ such that, for every blow-up in conformal coordinates $\widetilde p\in\mathcal{E}_{\lambda}^+$ satisfying $\|\widetilde p\|_{L^2_\mu}=1$, if 
			$$\|\widetilde w(\cdot,0)-\widetilde p\|_{L^2_\mu}\le \delta_1,\quad  \widehat W_{\lambda}(\widetilde w,0)\le\delta_1\quad\text{and}\quad\widehat W_\lambda(\widetilde w,\tau)\ge0\qquad \forall \tau\in(0,\infty),$$ then \eqref{eq:closeness-assumption} holds for every $\tau'\in(0,\infty)$. In particular, the epiperimetric inequality in \cref{prop:epiperimetric-inequality} can be applied to $\widetilde w(\cdot,\tau),$ for every $\tau\in(0,\infty)$. 
			
			Let $\delta_1\in(0,1)$ to be chosen small enough later.
			We observe that, by the monotonicity formula \eqref{eq:Weiss-monotonicity-formula} for $\widehat{W}_{\lambda}$,
			$$0\le\widehat W_{\lambda}(\widetilde w,\tau')\le\delta_1\le1\qquad\forall \tau'\in(0,\infty).$$
			Then, it only remains to prove the $L^2_\mu$-closeness condition in \eqref{eq:closeness-assumption}, for every $\tau'\in(0,\infty)$. Note that, by \eqref{eq:Weiss-difference-derivative}, for every $0\le \tau_1\le\tau_2<\infty$ we have
			\be\label{eq:servedopo}\begin{aligned} \|\widetilde w(\cdot,\tau_2)-\widetilde w(\cdot,\tau_1)\|_{L^2_\mu}&\le \int_{\tau_1}^{\tau_2}\|\partial_\tau\widetilde w\|_{L^2_\mu}\le \left(\widehat W_\lambda(\widetilde w,\tau_1)-\widehat W_\lambda(\widetilde w,\tau_2)\right)^{1/2}(\tau_2-\tau_1)^{1/2}\\&\le \widehat W_\lambda(\widetilde w,\tau_1)^{1/2}(\tau_2-\tau_1)^{1/2}.\end{aligned}\ee
			Then, given $\sigma>1$ to be chosen large enough later, for every $\tau\in(0,\sigma]$, we have
			$$\|\widetilde w(\cdot,\tau)-\widetilde p\|_{L^2_\mu}\le \|\widetilde w(\cdot,\tau)-\widetilde w(\cdot,0)\|_{L^2_\mu}+\|\widetilde w(\cdot,0)-\widetilde p\|_{L^2_\mu}\le \delta_1^{1/2}\sigma^{1/2}+\delta_1\le C_{\sigma}\delta_1^{1/2}.$$ If $\delta_0$ is the constant in the epiperimetric inequality from \cref{prop:epiperimetric-inequality}, we choose $\delta_1$ depending on $\sigma$ in such a way that $C_{\sigma}\delta_1^{1/2}\le \delta_0$. In particular, we have 
			\begin{equation*}
				\|\widetilde w(\cdot,\tau)-\widetilde p\|_{L^2_\mu}\le \delta_{0}\qquad \forall \tau\in (0,\sigma].
			\end{equation*}
			We wish to extend this bound to any $\tau \in (0,\infty)$.
			
			Let $\tau_0\in(\sigma,\infty]$ be the largest number for which $\|\widetilde w(\cdot,\tau)-\widetilde p\|_{L^2_\mu}\le\delta_0$ for all $\tau \in (0,\tau_{0})$. We assume by contradiction that $\tau_0<\infty$.
			We may then apply repeatedly the epiperimetric inequality from \cref{prop:epiperimetric-inequality} to $\widetilde w(\cdot,\tau)$ for every $\tau\in(0,\tau_0-1)$ and get
			$$\widehat W_\lambda(\widetilde w,j)\le (1-\eps_{0})^j\widehat W_\lambda (\widetilde w,0)\qquad\forall j\in\N,\ j\le\tau_0.$$
			Combining the above inequality with the monotonicity formula \eqref{eq:Weiss-monotonicity-formula} for $\widehat{W}_{\lambda}$, we obtain
			
			\be\label{eq:decay-weiss-energy}\widehat W_{\lambda}(\widetilde w,\tau)\le Ce^{-\alpha\tau}\widehat W_{\lambda}(\widetilde w,0)\qquad\forall \tau\in[0,\tau_0],\ee for some $\alpha\in(0,1)$.
			For every $0\le \tau_1\le\tau_2\le\tau_0<\infty$, using \eqref{eq:servedopo} at all subsequent integer times between $\tau_1$ and $\tau_2$, we obtain 
			\be\label{eq:rate-implication}\|\widetilde w(\cdot,\tau_2)-\widetilde w(\cdot,\tau_1)\|_{L^2_\mu}\le Ce^{-\alpha\tau_1/2}\widehat W_{\lambda}(\widetilde w,0)^{1/2}.\ee 
			In particular
			$$\|\widetilde w(\cdot,\tau_0)-\widetilde w(\cdot,\sigma)\|_{L^2_\mu}\le Ce^{-\alpha\sigma/2}\widehat W_{\lambda}(\widetilde w,0)^{1/2}\le Ce^{-\alpha\sigma/2}\le\frac{\delta_0}2,$$ provided that $\sigma$ is chosen large enough.
			Therefore, using \eqref{eq:servedopo} and further decreasing  $\delta_1$ if necessary, for every $\tau\in(\tau_0,\tau_0+1)$ we find 
			\bea \|\widetilde w(\cdot,\tau)-\widetilde p\|_{L^2_\mu}&\le\|\widetilde w(\cdot,\tau)-\widetilde w(\cdot,\tau_0)\|_{L^2_\mu}+\|\widetilde w(\cdot,\tau_0)-\widetilde w(\cdot,\sigma)\|_{L^2_\mu}+\|\widetilde w(\cdot,\sigma)-\widetilde p\|_{L^2_\mu}\\&\le \delta_1^{1/2}+\frac{\delta_0}2+C_{\sigma}\delta_1^{1/2}\le \delta_0.\eea  This contradicts the maximality of $\tau_0$.

			\smallskip
			\noindent \textit{Step 2.} Next we prove that there exists a non-zero $p\in\mathcal{P}_\lambda^+$ and $t_0\in[-1,0)$ such that for every $t\in[t_0,0)$, 
			\be\label{eq:proof-step2}\int_{\R^{n+1}_+}|w(x,t)-p(x,t)|^2G(x,t)\,dx\le C(\sqrt{-t})^{2\lambda+2\alpha},\ee 
			for some $C>0$ depending only on $n$, $m$, $C_{f}$, $\ell$ and $\|w\|_{L^\infty(Q_1)}$.
			
			For any $\rho\in(0,1)$, we set
			\begin{equation*}
				w^\rho(x,t):=\frac{w(\rho x,\rho^2t)}{\big((\lambda+1)H(\rho,w)\big)^{1/2}}.
			\end{equation*}
			We want to apply \textit{Step 1} to the function $\widetilde w^\rho$ for some $\rho$ sufficiently small.
			Since $0\in\Gamma_{\lambda}(w)$, by \cite[Theorem 7.3]{dgpt17}, there exists $q\in\mathcal{P}_\lambda^+$ satisfying $H(1,q)=1/(\lambda+1)$, or equivalently $\|\widetilde q\|_{L^2_\mu}=1$, such that $q$ is a blow-up of $w$ at $0$ along some sequence of radii $\rho_{j}\downarrow 0$. In addition, there exists some $t_\ast\in(-1,-e^{-1})$ such that $$\int_{\R^{n+1}_+}|w^{\rho_{j}}(\cdot, t_{\ast})-q(\cdot, t_{\ast})|^2G(\cdot, t_{\ast})+\int_{\R^{n+1}_+}|\nabla w^{\rho_{j}}(\cdot, t_{\ast})-\nabla q(\cdot, t_{\ast})|^2G(\cdot,t_\ast)\to 0$$ as $j\to\infty$.
			Moreover, recalling that $\mathcal{M}_\lambda$ is the Weiss' energy in the original coordinates as defined in \eqref{eq:weiss-original-coordinates}, then
			$\mathcal{M}_\lambda(w^{\rho_j},t_\ast)\to \mathcal{M}_\lambda(q,t_\ast)=0$ as $j\to\infty$. The monotonicity formula from \cref{lemma:shi} for $\mathcal M_\lambda$ reads as follows:
			$$r\mapsto \mathcal M_{\lambda}(w^{\rho_{j}},-r^2)+C_0(\rho_{j}) r^{2(\ell-\lambda)}$$
            is monotone non-decreasing. Moreover, the corresponding constant $C_0(\rho_{j})$ goes to zero as $j\to \infty$.
			Therefore $\mathcal{M}_\lambda(w^{\rho_j},t_\ast)+C_0(\rho_j)(-t_{*})^{\ell-\lambda}\to 0$ as $j\to\infty$.
			As a consequence, we can choose $\rho=\rho_{N}$ for a sufficiently large $N\in \N$ such that, for $\tau_\ast:=-\log (-t_\ast)\in (0,1)$ we have 
			\be\label{eq:comb-1}\begin{aligned}\|\widetilde w^\rho(\cdot,\tau_\ast)-\widetilde q\|^2_{L^2_\mu}&=e^{\lambda\tau_\ast}\int_{\R^{n+1}_+}|w^\rho(e^{-\tau_\ast/2}x,-e^{-\tau_\ast})-q(e^{-\tau_\ast/2}x,-e^{-\tau_\ast})|^2G(x,-1)\,dx\\&\le C\int_{\R^{n+1}_+}|w^\rho(x,t_\ast)-q(x,t_\ast)|^2G(x,t_\ast)\,dx\le\delta_1^2,
			\end{aligned}\ee 
			and 
			\be\label{eq:comb-2}\widehat W_{\lambda}(\widetilde w^\rho,\tau_\ast)=\mathcal{M}_{\lambda}(w^\rho,t_\ast)+2C_0(\rho)(-t_\ast)^{\ell-\lambda}\le \delta_1.\ee 
			Finally, since $\widehat W_\lambda(\widetilde w^\rho,\infty)=\mathcal M_{\lambda}(w^\rho,0^-)=0$ and $\widehat W_\lambda$ is non-increasing, then 
			\be\label{eq:comb-3}\widehat W_{\lambda}(\widetilde w^\rho,\tau)\ge0\qquad\forall \tau\in(0,\infty).\ee
			Having \eqref{eq:comb-1}, \eqref{eq:comb-2}, and \eqref{eq:comb-3}, we can now apply \textit{Step 1} to $\widetilde{w}^{\rho}(\cdot, \tau_{*}+\cdot)$ and obtain, as in \eqref{eq:decay-weiss-energy},
			$$\widehat W_{\lambda}(\widetilde w^\rho,\tau)\le Ce^{-\alpha\tau}\qquad\forall \tau\in(\tau_\ast, \infty).$$ Thus, also \eqref{eq:rate-implication} holds for $\widetilde{w}^{\rho}$ for any $\tau_{\ast}<\tau_{1}\le \tau_{2}<\infty$, and therefore, sending $\tau_{2}\to \infty$, we obtain  
			$$\|\widetilde w^\rho(\cdot,\tau)-\widetilde p^\rho\|_{L^2_\mu}\le Ce^{-\alpha\tau/2}\widehat W_{\lambda}(\widetilde w^\rho,\tau_\ast)^{1/2}\qquad \forall \tau\in(\tau_\ast,\infty),$$ for some $\widetilde p^\rho\in\mathcal E_\lambda^+\setminus\{0\}$.
			Going back to physical coordinates and setting $p^\rho(x,t):=(-t)^{\lambda/2}
            \widetilde p^\rho\left(\frac{x}{\sqrt{-t}}\right),$ we find, for every $t\in[t_\ast,0)$,
			\bea \int_{\R^{n+1}_+}|w^\rho(x,t)-p^\rho(x,t)|^2G(x,t)\,dx&= (\sqrt{-t})^{2\lambda}\int_{\R^{n+1}_+}|\widetilde w^\rho(y,-\log(-t))-\widetilde p^\rho(y)|^2\,d\mu\\&\le
			C(\sqrt{-t})^{2\lambda}e^{-\alpha(-\log(-t))}\\&\le C(\sqrt{-t})^{2\lambda+2\alpha},\eea for some 
			$C>0$.
            The conclusion \eqref{eq:proof-step2} follows after decreasing $\alpha$ and choosing $t_0\in[t_\ast\rho^2,0)$.

			\smallskip
			\noindent \textit{Step 3.} Now we conclude the proof of the proposition.
			We consider the rescaling $w_r(x,t):=r^{-\lambda}w(rx,r^2t)$, for $r>0$.
			By \eqref{eq:proof-step2}, there exists $r_0\in(0,1)$ such that 
			\be\label{eq:decay-r}\int_{Q_{1}}|w_r(x,t)-p(x,t)|^2G(x,t)\,dx\,dt=\frac{1}{r^{2\lambda+2}}\int_{B_{r}\times(-r^2,0)}|w(x,t)-p(x,t)|^2G(x,t)\,dx\,dt\le Cr^{2\alpha},\ee  for every $r\in(0,r_0)$.
			Arguing as in \cref{lemma:wang}, we get $$\|w_r-p\|_{L^\infty(Q_{1/2})}\le C\|w_r-p\|_{L^2(Q_{1})}+Cr^{\ell-\lambda}\qquad \forall r\in(0,r_0).$$ 
			Let $\eta\in(0,1)$ to be chosen small enough later. Then, we can find a constant $C_\eta>0$ depending on $\eta$ such that, by \eqref{eq:decay-r}, for every $r\in(0,r_0)$, we have
			\bea \|w_r-p\|^2_{L^\infty(Q_{1/2})}&\le C\int_{Q_1}|w_r(x,t)-p(x,t)|^2\,dx\,dt+Cr^{2(\ell-\lambda)}\\&\le C\int_{B_1\times(-1,-\eta)}|w_r(x,t)-p(x,t)|^2\,dx\,dt+C\eta \|w_r-p\|^2_{L^\infty(Q_{1})}+Cr^{2(\ell-\lambda)}\\&\le C_{\eta}\int_{B_1\times(-1,-\eta)}|w_r(x,t)-p(x,t)|^2G(x,t)\,dx\,dt+C\eta \|w_r-p\|^2_{L^\infty(Q_{1})}+Cr^{2(\ell-\lambda)}\\&\le C_{\eta}r^{2\alpha}+C\eta \|w_r-p\|^2_{L^\infty(Q_{1})}+Cr^{2(\ell-\lambda)},
			\eea 
            where we used that $G\ge c_{\eta}$ in $B_1\times (-1,-\eta)$.
			Then, up to further decreasing $\alpha$, we have 
            \bea \|w-p\|^2_{L^\infty(Q_{r/2})}\le C_{\eta}r^{2\lambda +2\alpha}+C\eta \|w-p\|^2_{L^\infty(Q_{r})}\qquad \forall r\in(0,r_0).
			\eea
			If $\eta$ is chosen small enough, a standard iteration argument (see e.g.~the proof of \cite[Lemma 8.23]{GilbargTrudinger:Elliptic98}) gives 
            $$\|w-p\|^2_{L^\infty(Q_{r})}\le Cr^{2\lambda+2\alpha}\qquad\forall r\in(0,r_0),$$ after further decreasing $\alpha$ and $r_{0}$. This concludes the proof.
		\end{proof}
		
		\begin{proof}[Proof of \cref{thm:rate-for-u}]
			The result follows immediately from \cref{prop:rate-for-w}, for $w:=\zeta \widetilde u^{0,0}$. Note that, after multiplying $u$ by a positive normalization constant, we can assume that $\|u\|_{L^\infty(Q_1)}$ and the constant $C_f$ in \eqref{controlf} are bounded by $1$. Therefore, since $\lambda=3$ and $\ell\ge4$, we can apply \cref{prop:rate-for-w} obtaining an exponent $\alpha$ depending only on $n$.
		\end{proof}
        \begin{remark}
            A closer analysis of the proofs shows that the constants appearing in \cref{thm:rate-for-u} are uniform for cubic points in a given compact set. This remark will be used in the next sections without further reference. 
        \end{remark}
		\subsection{Proof of the frequency gap} As a further application of the epiperimetric inequality from \cref{prop:epiperimetric-inequality}, we prove the gap result at odd integer frequencies.
		\begin{proof}[Proof of \cref{prop:odd-gap}]
			Let $\lambda=2m+1$. Suppose, for a contradiction, that there is a sequence $u_{k}$ of non-trivial parabolically $(2m+1+\eps_k)$-homogeneous global solutions of \eqref{eq:parabolic-signorini-problem-global}, where $\eps_k>0$ and $\eps_k\downarrow 0$ as $k\to\infty$.
			We normalize the solutions so that $\|u_k(\cdot,-1)\|_{L^2_\mu}=1$. The blow-up analysis yields some $p\in\mathcal{P}_{\lambda}^+$, with $\lVert \widetilde{p}\rVert_{L^{2}_{\mu}}=1$ such that, up to a subsequence,  
			$$\int_{\R^{n+1}_+}|u_k(x,-1)- p(x,-1)|^2G(x,-1)\,dx\to 0\qquad\text{as }k\to\infty.$$
			We consider $$\widetilde u_k(y,\tau):=e^{\lambda\tau/2}u_k(e^{-\tau/2}y,-e^{-\tau})=e^{-\eps_k\tau/2}u_k(y,-1)\not\equiv0.$$ 
			We have
			\begin{equation*}
				\sup_{\tau'\in (1,2)}\|\widetilde u_k(\cdot,\tau')-\widetilde p\|_{L^2_\mu}\to 0,
			\end{equation*}
			and
			\begin{equation*}
				0<W_\lambda(\widetilde u_k,\tau')=\frac{\varepsilon_k}{2}e^{-\varepsilon_k\tau'}\le1\qquad \forall \tau'\in (1,2)
			\end{equation*}
			for all sufficiently large $k$. Thus \eqref{eq:closeness-assumption} holds for $\tau=1$ and $\widehat{W}_{\lambda}$ replaced by $W_{\lambda}$.
			We may then apply the epiperimetric inequality from \cref{prop:epiperimetric-inequality} and \cref{rmk:epi-zero-rhs} and obtain
			$$e^{-\eps_k(\tau+1)}\mathcal{M}_{\lambda}(u_k,-1)=W_{\lambda}(\widetilde u_k,\tau+1)\le (1-\eps_0)W_{\lambda}(\widetilde u_k,\tau)=(1-\eps_0)e^{-\eps_k\tau}\mathcal{M}_{\lambda}(u_k,-1).$$ Dividing by $\mathcal{M}_{\lambda}(u_k,-1)>0$, we see that $e^{-\eps_k}\le(1-\eps_0)$, which is a contradiction for $k$ large enough.
		\end{proof}

\section{Cleaning at cubic points}\label{sec-cleaning-cubic}
In this section we prove that the set $\Gamma_3$ of cubic free boundary points can be cleaned up with a superquadratic rate, namely $\Gamma_3\in\text{Clean}(2+\eta)$ for some $\eta>0$. The proof uses the polynomial rate of convergence to the blow-up
established in \cref{thm:rate-for-u} and a quantitative refinement of the barrier argument in \cref{lemma:localization}. 
\begin{proposition}\label{prop:cleaning-odd}
	Let $u$ be a solution of \eqref{eq:parabolic-signorini-problem}, with obstacle $\varphi$ satisfying \eqref{eq:ostacolo}. Then $$\Gamma_{3}\in\text{Clean}(2+\eta),$$ for some $\eta>0$ depending only on $n$.  
\end{proposition}
\begin{proof}
	The result follows if we prove that there is $\eta>0$ depending only on $n$ such that, for every $(x_0,t_0)\in\Gamma_3$, there exists $\rho >0$ such that
	\begin{equation*}\label{eq:claim-cleaning-cubic}
		\left\{(x,t)\in B_{\rho}(x_0)\times(-1,0]:t>t_{0}+|x-x_0|^{2+\eta}\right\}\cap\Gamma_3 = \emptyset.
	\end{equation*}
	Let $(x_1,t_1)\in \Gamma_3$ be such that $r:=|x_1-x_0|<\rho$ and $t_1>t_0+r^{2+\eta}$. We will find a contradiction for sufficiently small $\rho>0$ and $\eta>0$, the latter depending only on $n$. 
	For $\rho$ small enough we may translate and rescale the solution around $(x_{1}, t_{1})$, so that, without loss of generality, we suppose that $(x_{1},t_{1})=(0,0)$.
	We will show that the function $\widetilde u:= \widetilde{u}^{0,0}$ as defined in \eqref{eq:utilde} vanishes in a neighborhood of $(x_0,t_0)$ inside $Q_{1}'$, which is in contradiction with $(x_0,t_0)\in\Gamma_3$.

	By \cref{thm:rate-for-u}, we know that there is $p\in\mathcal{P}^+_{3}$ such that, for all sufficiently small $r>0$,
	\be\label{eq:rate-clean}\|\widetilde u_r-p\|_{L^\infty(Q_2)}\le Cr^\alpha,\ee where $\widetilde u_r(x,t):=r^{-3}\widetilde u(rx,r^2t)$. 
	By \cref{lemma:monotonicity-ass}, there is some $c_{0}>0$ depending on $u$ such that 
	$$\partial_t\widetilde u=\partial_t u\ge c_{0}|x_{n+1}|\qquad\text{in }Q_{1/2}.$$ 
	Then, by the cubic rescaling, we deduce
	\be\label{eq:rescaling-bound}\partial_t \widetilde u_r\ge c|x_{n+1}|\qquad\text{in }Q_2,\ee for all $r>0$ small enough. 
	
	We will prove that 
	\be\label{eq:toprove}\widetilde u_r(x,t)=0\qquad\forall (x,t)\in B_{\frac32}'\times(-1,-r^\eta).\ee
	Given this, rescaling back and using the time-monotonicity of $\widetilde u$, we get \be\label{eq:bettercleaning}\widetilde u(x,t)=0\qquad\forall  (x,t)\in B_{\frac{3}{2}r}'\times (-1,-r^{2+\eta}),\ee from which the conclusion follows, since $|x_0|=r$ and $t_0<-r^{2+\eta}$.
	
	In order to prove \eqref{eq:toprove}, let $(z,\overline t)\in B_{3/2}'\times (-1,-r^{\eta})$. We need to prove that $\widetilde u_r(z,\overline t)=0$.
	We let $$U:=\left\{(x,t):\,|x'-z'|< 1/2,\ |x_{n+1}|<\sigma,\ t\in(\overline t-1/5,\overline t]\right\}$$ for some $\sigma\in(0,1)$, to be chosen below. By \cref{lem:classification-odd-blow-ups}, $p$ can be written as $$p(x,t)=-|x_{n+1}|(q_0(x',t)+bx_{n+1}^2),$$ where $q_0\ge0$ is a non-zero parabolically $2$-homogeneous polynomial, and $b\in\R$. In particular, $p(x,t)\le C|x_{n+1}|^3$.
	Then, choosing $\sigma=r^{\alpha/2}$, $\eta=\alpha/4$ and using \eqref{eq:rate-clean}, \eqref{eq:rescaling-bound}, for every $(x,t)\in \partial_{\text{par}}U$ and $r$ sufficiently small, we get
	\bea \widetilde u_r(x,t)&\le \widetilde u_r(x,0)+c_0|x_{n+1}|t\le c_0|x_{n+1}|t+
	p(x,0)+Cr^\alpha\le c_0|x_{n+1}|t+C|x_{n+1}|^3+Cr^\alpha\\&\le\begin{cases}
		-c_0 r^\eta\sigma +C\sigma^3+Cr^\alpha\le -(n+2)\sigma^2\qquad &\text{if $|x_{n+1}|=\sigma$},\\
		-c_0 r^\eta|x_{n+1}|+C\sigma^3+Cr^\alpha\le -(n+2)x_{n+1}^2+1/4\qquad &\text{if $|x_{n+1}|<\sigma$ and $|x'-z'|=1/2$},\\
		C\sigma^3+Cr^\alpha\le -(n+2)\sigma^2+2/5\qquad &\text{if $|x_{n+1}|<\sigma$ and $t=\overline t-1/5$}.\\
	\end{cases} \eea
	It follows that, for every $c\ge 0$,
	$$\widetilde u_r(x,t)\le -(n+2)x_{n+1}^2+|x'-z'|^2-2(t-\overline t)+c=:\phi_c(x,t)\qquad\forall (x,t)\in\partial_{\text{par}}U.$$
	Arguing exactly as in the proof of \cref{lemma:localization}, we may use $\phi_{c}$ as an upper barrier (since $\mathbf{H}\phi_{c}=-2$ and $\phi_c>0$ in $\{x_{n+1}=0\}$ for $c>0$), and get $\widetilde u_r(x,t)\le \phi_0(x,t)$ for every $(x,t)\in U$ and every $r$ sufficiently small. In particular $\widetilde u_r(z,\overline t)=0$. By the arbitrariness of $z$ and $\overline t$, we get \eqref{eq:toprove}, concluding the proof.
	\end{proof}
		
		\section{Higher order expansions at cubic points}\label{sec-expansion-cubic}
		The goal of this section is to prove a higher order expansion of a solution $u$ in backward cylinders near cubic points. If $\ell=k+\gamma$ is the regularity of the obstacle, we will approximate $u$ at order $k$ using ansatz polynomials in the following class:
		\bea \mathcal{P}_{\le k}:=\Big\{P:S_\infty\to\R&\text{ such that $P|_{S_{\infty}^{+}}$ is a polynomial of parabolic degree $\le k$},\ \\&\mathbf{H}P=0\text{ in }S_{\infty}^{+},\ P\equiv0\text{ on }S_{\infty}',\ P(x',x_{n+1},t)=P(x',-x_{n+1},t) \Big\}.\eea 
		We also introduce the backward cylinder \be\label{Kr}K_r:=B_r\times (-r^2,-r^2/100).\ee
		The main result of this section is the following theorem.
		\begin{theorem}\label{prop:high-order-blow-up}
			Let $u$ be a solution of \eqref{eq:parabolic-signorini-problem}, with obstacle $\varphi\in C^{k,\gamma}(B_1')$ satisfying \eqref{eq:ostacolo}, and assume that $(0,0)\in\Gamma_3$. Let $\widetilde u:=\widetilde u^{0,0}$ be defined as in \eqref{eq:utilde}. Then, there are $r_0\in (0,1)$ and a unique non-zero $P_{\le k}\in\mathcal{P}_{\le k}$ such that, for every $\alpha\in(0,\gamma)$, 
			\bea\|\widetilde u-P_{\le k}\|_{L^\infty(K_r)}\le Cr^{k+\alpha}\qquad\forall r\in(0,r_0),\eea
			for some constant $C>0$.
		\end{theorem}
		Let $\eta>0$ be the constant given by \cref{prop:cleaning-odd}. For every $N>0$, we set
		$$\Omega_N^\eta:=\left\{(x,t)\in\R^{n+1}\times (-\infty,0):\ |x|^{2+\frac\eta2}< N|t|\right\}.$$
		The following lemma, which is a direct consequence of the $2+\eta$ cleaning of cubic points, states that $\widetilde u$ vanishes on the thin part of $\Omega_{N}^{\eta}$.
		\begin{lemma}\label{lemma:u=0}
			Let $u$ be a solution of \eqref{eq:parabolic-signorini-problem}, with obstacle $\varphi$ satisfying \eqref{eq:ostacolo}, and assume that $(0,0)\in\Gamma_3$.
			For every $N>0$, there exists $\rho >0$ such that
			\begin{equation*}
				\Omega^\eta_N\cap Q'_\rho\subseteq \Lambda.
			\end{equation*}
		\end{lemma}
		\begin{proof}
			By the definition of $\Omega_{N}^{\eta}$, for a sufficiently small $\rho>0$, we have 
			\bea\Omega^\eta_N\cap Q'_\rho = \left\{(x,t)\in Q_\rho':\ N|t|>|x|^{2+\frac\eta2}\right\}\subseteq \left\{(x,t)\in Q_\rho':\ |t|>|x|^{2+\eta}\right\}.\eea
			Now, any $(x,t)\in Q_{\rho}'$ such that $|t|>|x|^{2+\eta}$ belongs to $B_{\frac{3}{2}r}'\times (-1,-r^{2+\eta})$ for $r:=|x|<\rho$.
			Then, up to further decreasing $\rho$, the argument in the proof of \cref{prop:cleaning-odd} gives $\widetilde{u}^{0,0}(x,t)=0$ (see \eqref{eq:bettercleaning}).
			This concludes the proof. 
		\end{proof}
		We will also use the following Liouville-type lemma.
		\begin{lemma}\label{lemma:w}
			Let $k\in\N$, $\alpha\in(0,1)$ and $w:S_\infty\to \R$ be an even function with respect to $\{x_{n+1}=0\}$ such that the following hold:
			\begin{itemize}
				\item[i)] $\mathbf{H}w=0$ in $S_\infty^+$;

                \smallskip
				\item[ii)] $w\equiv0$ on $S_\infty'$;

                \smallskip
				\item[iii)] $w$ satisfies the growth estimate $$\|w\|_{L^2(Q_R)}\le CR^{k+\alpha+\frac{n+3}2}\qquad\forall R\ge 1;$$

                \smallskip
				\item[iv)] $w$ is orthogonal to $\mathcal{P}_{\le k}$ in $Q_1$, namely $$\int_{Q_1}wP=0\qquad\forall P\in\mathcal{P}_{\le k}.$$
			\end{itemize}
			Then $w\equiv0$.
		\end{lemma}
		\begin{proof}
			By $i)$ and $ii)$, the odd extension $\hat{w}:= \sgn(x_{n+1})w$ is caloric in the whole $S_{\infty}$. 
			Then, $iii)$ and standard parabolic estimates give $$\|\hat w\|_{L^\infty(Q_R)}\le CR^{-\frac{n+3}2}\|\hat w\|_{L^2(Q_{2R})}\le CR^{k+\alpha}\qquad\forall R\ge1.$$ The standard parabolic Liouville theorem then gives 
			$w\in\mathcal{P}_{\le k}$. Therefore, by the orthogonality condition iv) with $P=w$, we obtain $w\equiv0$.
		\end{proof}
		In the following lemma we prove an estimate of the space-time $W^{1,2}$-norm in terms of the $L^2$-norm in domains $\Omega_N^\eta$.
		\begin{lemma}\label{lemma:technical-lemma}
			For every $\theta\in(1/2,1)$, there exist $N_0>0$ and $C>0$ depending only on $n$ and $\theta$ such that for every $N\ge N_0$ and every solution $w$ of \,$\mathbf{H}w=f$ in $\Omega_N^\eta\cap Q_2$, we have
			$$\int_{\Omega_{\theta N}^\eta\cap Q_1}\Big(|\nabla w|^2+|\partial_t w|^2\Big)\le C\int_{\Omega_N^\eta\cap Q_2} w^2+C\|f\|^2_{L^\infty(Q_2)}.$$ 
		\end{lemma}
		
		\begin{proof}
			Let $\mu := \sqrt{\theta}$, $\alpha := 1/(2+\frac{\eta}{2}) \in \left(1/3,1/2\right)$, let $p>2$ satisfy
			$\frac{p}{p+1} = 2\alpha$, and let $\xi \in C^{1,1}([0,\infty))$ be a cut-off function such that 
            \begin{equation*}
            \begin{cases}
                \xi(s)=1\qquad &\text{for $s\in [0,\mu^{\alpha}]$},\\
                \xi \text{ is concave and decreasing }\qquad &\text{in $[\mu^{\alpha},\frac{\mu^{\alpha}+1}{2}]$},\\
                \xi(s)=\kappa(1-s)^{2+p}\qquad &\text{for $s\in [\frac{\mu^{\alpha}+1}{2},1]$},\\
                \xi(s)=0\qquad &\text{for $s\in [1,\infty)$},
            \end{cases}
            \end{equation*}
			where $\kappa>0$ is chosen depending on $\mu$ and $p$. Set
			$$
			\xi_N(x,t) := \xi\left(\frac{|x|}{(N|t|)^{\alpha}}\right),
			$$
			so that $\operatorname{supp}\xi_N \subseteq \overline{\Omega_N^\eta}$ and $\xi_N \equiv 1$ on
			$\Omega_{\mu N}^\eta$. Arguing as in \cite[proof of Lemma 13.7]{frs24}, we can find $N_0>0$
			and $C_0>0$, depending only on $n$, $\theta$ and $\eta$, such that for every $N\ge N_0$,
            \begin{equation*}
                \partial_t \xi_N \leq 0\qquad \text{and}\qquad \Delta \xi_N + \partial_t \xi_N \leq C_0
			\quad \text{in } \mathbb{R}^{n+1}\times(-4,0).
            \end{equation*}
			Moreover, we claim that after enlarging $N_0$ or $C_0$ if necessary,
			$$16|\nabla\xi_{\mu N}|^2 + 2\xi_{\mu N}\,\partial_t\xi_{\mu N} \leq C_0.$$
			
			To prove the claim, let $s = (\mu N|t|)^{-\alpha}|x|$, and note that if $s \in [0,\mu^\alpha]$ or $s \geq 1$, $\xi' = 0$ and the claim is trivial. Now, if $s \in (\mu^\alpha,1)$, $|x| > \mu^{\alpha}(\mu N|t|)^\alpha$, and then since $\xi' \leq 0$,
			\begin{align*}
				16|\nabla\xi_{\mu N}|^2 + 2\xi_{\mu N}\,\partial_t\xi_{\mu N} &= 16(\mu N|t|)^{-2\alpha}|\xi'(s)|^2 + 2\alpha(\mu N)^{-\alpha}|t|^{-\alpha-1}|x|\,\xi(s)\xi'(s)\\
				&\leq 16(\mu N|t|)^{-2\alpha}|\xi'(s)|^2 + 2\alpha\mu^\alpha|t|^{-1}\,\xi(s)\xi'(s)\\
				&= 2\alpha\mu^\alpha|\xi'(s)|\left(\frac{8\mu^{-3\alpha}N^{-2\alpha}}{\alpha}|t|^{-2\alpha}|\xi'(s)| - |t|^{-1}\xi(s)\right)\\
				&= C|\xi'(s)|\left(\varepsilon|t|^{-2\alpha}|\xi'(s)| - |t|^{-1}\xi(s)\right),
			\end{align*}
			where $\varepsilon > 0$ can be chosen as small as necessary by making $N_0$ large.
			
			Fix $\delta>0$ small enough that $1-\delta|t|^{1-2\alpha}>\frac{1+\mu^\alpha}{2}$ for $|t|<4$. Then, we separate three cases.
			\begin{itemize}
				\item If $s \in [\mu^\alpha, \frac{1+\mu^\alpha}{2}]$, we just note that $|\xi'(s)| \leq L$ and $\xi(s) \geq \kappa(\frac{1-\mu^\alpha}{2})^{2+p}$. Hence, for $\varepsilon$ small enough
				$$\varepsilon|t|^{-2\alpha}|\xi'(s)|-|t|^{-1}\xi(s) \leq |t|^{-1}\left(4\varepsilon L - \kappa\left(\frac{1-\mu^\alpha}{2}\right)^{2+p}\right) \leq 0.$$

                \smallskip
				\item If $s \in [\frac{1+\mu^\alpha}{2}, 1 - \delta|t|^{1-2\alpha}]$,
				\begin{align*}
					|\xi'(s)|\left(\varepsilon|t|^{-2\alpha}|\xi'(s)|-|t|^{-1}\xi(s)\right) &= \kappa^2(2+p)(1-s)^{2+2p}\left((2+p)\varepsilon|t|^{-2\alpha} - |t|^{-1}(1-s)\right)\\
					&\leq C(1-s)^{2+2p}\left((2+p)\varepsilon|t|^{-2\alpha} - \delta|t|^{-2\alpha}\right) \leq 0,
				\end{align*}
				after choosing $\varepsilon$ smaller if necessary.

                \smallskip
				\item Finally, if $s > 1 - \delta|t|^{1-2\alpha}$, $\xi(s) = \kappa(1-s)^{2+p}$, and hence
				$$|\xi'(s)| \le \kappa(2+p)\delta^{1+p}|t|^{(1-2\alpha)(1+p)} = \kappa(2+p)\delta^{1+p}|t|,$$
				which implies
				$$C|\xi'(s)|\left(\varepsilon|t|^{-2\alpha}|\xi'(s)| - |t|^{-1}\xi(s)\right) \le C\varepsilon|t|^{-2\alpha}|\xi'(s)|^2 \le C|t|^{2-2\alpha},$$
				which is bounded.
			\end{itemize}
			
			Let us prove our regularity estimates. First, set $\widetilde\eta \in C^\infty_c(B_{5/3})$ such that $\widetilde\eta \equiv 1$ in $\overline{B_{4/3}}$, $0 \leq \widetilde\eta \leq 1$ and $|\nabla\widetilde\eta|+|D^2\widetilde\eta| \leq C$. Then
			\begin{align*}
				\frac{d}{dt}\int_{B_{5/3}\times\{t\}}w^2\widetilde\eta\,\xi_N + 2\int_{B_2\times\{t\}} (w\mathbf{H}w+|\nabla w|^2)\widetilde\eta\,\xi_N &= \int_{B_2\times\{t\}}w^2(\partial_t + \Delta)(\widetilde\eta\,\xi_N)\\
				&= \int_{B_2\times\{t\}}w^2(\widetilde\eta\,(\partial_t+\Delta)\xi_N+2\nabla\widetilde\eta\,\cdot\nabla\xi_N+\xi_N\Delta\widetilde\eta).
			\end{align*}
			Now, $\widetilde\eta\,(\partial_t+\Delta)\xi_N \leq C_0\chi_{B_{5/3}}(x)\chi_{\Omega^\eta_N}$. Moreover, on the support of $\nabla\widetilde\eta\,\cdot\nabla\xi_N$, $|x| \sim 1$, $\frac{1}{(N|t|)^\alpha} \sim 1$, and hence $|\nabla\xi_N|$ is bounded. Therefore,
			$$\frac{d}{dt}\int_{B_{5/3}\times\{t\}}w^2\widetilde\eta\,\xi_N + 2\int_{B_2\times\{t\}} (w\mathbf{H}w+|\nabla w|^2)\widetilde\eta\,\xi_N \leq C\int_{B_{5/3}\times\{t\}}w^2\chi_{\Omega^\eta_N}.$$
			Furthermore, using that there exists $T_* \in (\tfrac{3}{2},2)$ such that
			$$\int_{(B_{5/3}\times\{-T_*\})\cap\Omega^\eta_N}w^2\widetilde\eta\,\xi_N \leq 2\int_{(B_{5/3}\times(-2,0))\cap\Omega^\eta_N}w^2\widetilde\eta\,\xi_N,$$
			and then integrating with respect to $t \in (-T_*,0)$, using that $\mathbf{H}w = f$,
			\begin{align*}
				2\int_{-T_*}^0\int_{B_2\times\{t\}} (wf+|\nabla w|^2)\widetilde\eta\,\xi_N &\leq \int_{B_{5/3}\times\{-T_*\}}w^2\widetilde\eta\,\xi_N - \int_{B_{5/3}\times\{0\}}w^2\widetilde\eta\,\xi_N + C\int_{B_{5/3}\times(-T_*,0)}w^2\chi_{\Omega^\eta_N}\\
				&\leq C\int_{(B_{5/3}\times(-2,0))\cap\Omega^\eta_N}w^2.
			\end{align*}
			Finally, recalling the definitions of $\widetilde\eta$ and $\xi_N$ and using Young's inequality, it follows that
			$$\int_{(B_{4/3}\times(-\frac32,0))\cap \Omega^\eta_{\mu N}}|\nabla w|^2 \leq C\int_{\Omega^\eta_N\cap Q_2}w^2 + f^2,$$
			concluding our estimate for $|\nabla w|^2$. Now to estimate $|\partial_t w|^2$, we define $\eta \in C^\infty_c(B_{4/3})$ such that $\eta \equiv 1$ in $\overline{B_1}$. Then, integrating by parts,
			\begin{align*}
				(*) &:= 2\int_{B_2\times\{t\}}|\partial_t w|^2\eta^2\xi_{\mu N}^2 + \frac{d}{dt}\int_{B_2\times\{t\}}|\nabla w|^2\eta^2\xi_{\mu N}^2 - \int_{B_2\times\{t\}}|\nabla w|^2\partial_t\left[\eta^2\xi_{\mu N}^2\right]\\
				&= 2\int_{B_2\times\{t\}}|\partial_t w|^2\eta^2\xi_{\mu N}^2 + 2\int_{B_2\times\{t\}}\nabla w\cdot\nabla \partial_t w\,\eta^2\xi_{\mu N}^2\\
				&= \int_{B_2\times\{t\}}2(\partial_t w - \Delta w)\partial_t w \,\eta^2\xi_{\mu N}^2 - 4\nabla w\cdot\nabla\eta\, \partial_t w\,\eta\,\xi_{\mu N}^2 - 4\partial_t w \nabla w\cdot\nabla\xi_{\mu N}\,\eta^2\,\xi_{\mu N}.
			\end{align*}
			Now, using $\mathbf{H}w = f$, $4ab \leq \frac{1}{4}a^2 + 16b^2$ and $2ab \leq 2a^2 + \frac{1}{2}b^2$,
			\begin{align*}
				(*) &\leq -\int_{B_2\times\{t\}}2f\,\partial_t w \,\eta^2\xi_{\mu N}^2 + \frac{1}{2}\int_{B_2\times\{t\}}|\partial_t w|^2\eta^2\xi_{\mu N}^2 + 16\int_{B_2\times\{t\}}|\nabla w|^2(|\nabla\eta|^2\xi_{\mu N}^2 + |\nabla\xi_{\mu N}|^2\eta^2)\\
				&\leq 2\int_{B_2\times\{t\}}f^2\eta^2\xi_{\mu N}^2 + \int_{B_2\times\{t\}}|\partial_t w|^2\eta^2\xi_{\mu N}^2 + 16\int_{B_2\times\{t\}}|\nabla w|^2(|\nabla\eta|^2\xi_{\mu N}^2 + |\nabla\xi_{\mu N}|^2\eta^2).
			\end{align*}
			Since $|\nabla\eta|$ is bounded, $\operatorname{supp}(\eta\,\xi_{\mu N})\subseteq B_{4/3}\cap\overline{\Omega^\eta_{\mu N}}$, and $16|\nabla\xi_{\mu N}|^2 + 2\xi_{\mu N}\,\partial_t\xi_{\mu N} \leq C_0$, it follows that
			\begin{align*}
				\int_{B_2\times\{t\}}|\partial_t w|^2\eta^2\xi_{\mu N}^2 + \frac{d}{dt}\int_{B_2\times\{t\}}|\nabla w|^2\eta^2\xi_{\mu N}^2 &\leq C\int_{(B_{4/3}\times\{t\})\cap\Omega^\eta_{\mu N}}\big(|\nabla w|^2 + f^2\big).
			\end{align*}
			
			Finally, choosing $T_\diamond \in (1,\frac{3}{2})$ such that
			$$\int_{(B_{4/3}\times\{-T_\diamond\})\cap\Omega^\eta_{\mu N}}|\nabla w|^2\eta^2\xi_{\mu N}^2 \leq 2\int_{(B_{4/3}\times(-\frac{3}{2},-1))\cap\Omega^\eta_{\mu N}}|\nabla w|^2\eta^2\xi_{\mu N}^2,$$
			and then integrating in $(-T_\diamond,0)$ as before and using the definition of $\eta$,
			$$\int_{\Omega^\eta_{\mu^2 N}\cap Q_1}|\partial_t w|^2 \leq C\int_{(B_{4/3}\times(-\frac32,0)) \cap \Omega^\eta_{\mu N}}|\nabla w|^2 + f^2,$$
			which together with the estimate for $|\nabla w|^2$ completes the proof.
		\end{proof}
		
		The proof of \cref{prop:high-order-blow-up} will follow by iterating the following improvement of flatness lemma. \begin{lemma}\label{lemma:improvement}
			Let $u$ be a solution of \eqref{eq:parabolic-signorini-problem}, with obstacle $\varphi$ satisfying \eqref{eq:ostacolo}. Let $\widetilde u:=\widetilde u^{0,0}$ be defined as in \eqref{eq:utilde}, and let $\alpha\in(0,\gamma)$. Then, there exists $N_0\in\N$ such that the following holds. 
			Suppose that for some $m\in\N_{\ge N_0}$ there are polynomials $\{P_j\}_{j=0}^{m}\subseteq \mathcal{P}_{\le k}$ such that
			$$\left(\int_{\Omega^\eta_{N_0}\cap Q_1}(\widetilde u-P_j)^2(2^{-j}x,4^{-j}t)\,dx\,dt\right)^\frac12\le (2^{-j})^{k+\alpha}\qquad\forall j\in \{0,\ldots, m\}.$$
			Suppose also that $\widetilde u\equiv0$ in $\Omega^\eta_{N_0}\cap Q_{1}'$. Then there exists $ P_{m+1}\in \mathcal{P}_{\le k}$ such that 
			$$\left(\int_{\Omega^\eta_{N_0}\cap Q_1}(\widetilde u-P_{m+1})^2(2^{-(m+1)}x,4^{-(m+1)}t)\,dx\,dt\right)^\frac12\le ( 2^{-(m+1)})^{k+\alpha}.$$
		\end{lemma}
		\begin{proof}
			Suppose by contradiction that, for every $N\in\N$, there exist integers $m_N\ge N$, solutions $u^{N}$ and polynomials $\{P_{j}^{N}\}_{j=0}^{m_{N}}\subseteq \mathcal{P}_{\le k}$ such that \be\label{eq:contrad-ass-improvement}\left(\int_{\Omega^\eta_{N}\cap Q_1}(\widetilde u^{N}-P_j^{N})^2(2^{-j}x,4^{-j}t)\,dx\,dt\right)^\frac12\le (2^{-j})^{k+\alpha}\qquad\forall j\in \{0,\ldots,m_N\},\ee
			and $\widetilde u^{N}\equiv0$ in $\Omega^\eta_{N}\cap Q_1'$,
			but the conclusion does not hold, namely:
			\be\label{eq:contrad-thesis-improvement} \left(\int_{\Omega^\eta_{N}\cap Q_1}(\widetilde u^{N}-P)^2(2^{-(m_N+1)}x,4^{-(m_N+1)}t)\,dx\,dt\right)^\frac12> ( 2^{-(m_N+1)})^{k+\alpha}\qquad\forall P\in\mathcal{P}_{\le k}.\ee
			We define $$w^N(x,t):=\frac{\widetilde u^N(2^{-m_N}x,4^{-m_N}t)-P^{N}_{m_N+1}(2^{-m_N}x,4^{-m_N}t)}{(2^{-m_N})^{k+\alpha}},$$ where $P^N_{m_N+1}\in\mathcal{P}_{\le k}$ is such that
			$$P^N_{m_N+1}\in \underset{P\in\mathcal{P}_{\le k}}{\text{argmin}}\int_{\Omega^\eta_N\cap Q_1} (\widetilde u^N-P)^2(2^{-m_N}x,4^{-m_N}t)\,dx\,dt.$$
			Note that, by \eqref{eq:contrad-thesis-improvement}, there exists a constant $c_0>0$ such that \be\label{eq:finalestimate-improv}\int_{\Omega_{2^{-\eta/2 }N}^\eta\cap Q_{1/2}}(w^N)^2> c_0.\ee
			We divide the rest of the proof into several steps.

			\smallskip
			\noindent\textit{Step 1: $P^N_{m_N+1}$ approximates $\widetilde u^N$ at every scale.} In the first step we prove that
			\be\label{eq:contrad-ass-improvement-with-j}\left(\int_{\Omega^\eta_N\cap Q_1}(\widetilde u^{N}-P_{m_N+1}^{N})^2(2^{-j}x,4^{-j}t)\,dx\,dt\right)^\frac12\le C (2^{-j})^{k+\alpha}\qquad\forall j\in\{0,\ldots,m_N\},\ee for some $C>0$ independent of $N$.
			
			As a consequence of \eqref{eq:contrad-ass-improvement}, for every $j\in\{0,\ldots,m_N-1\}$, we have \bea\left(\int_{\Omega^\eta_{N}\cap Q_1}(\widetilde u^{N}-P_j^{N})^2(2^{-(j+1)}x,4^{-(j+1)}t)\,dx\,dt\right)^\frac12&\le C\left(\int_{\Omega^\eta_{N}\cap Q_1}(\widetilde u^{N}-P_j^{N})^2(2^{-j}x,4^{-j}t)\,dx\,dt\right)^\frac12
			\\&\le C(2^{-(j+1)})^{k+\alpha}.\eea
			Therefore, combining the previous inequality with \eqref{eq:contrad-ass-improvement} applied for $j+1$, the triangle inequality gives   
			\be\label{eq:allafine}\left(\int_{\Omega^\eta_{N}\cap Q_1}(P_{j+1}^{N}-P_{j}^{N})^2(2^{-(j+1)}x,4^{-(j+1)}t)\,dx\,dt\right)^\frac12\le C(2^{-(j+1)})^{k+\alpha}\qquad\forall j\in\{0,\ldots, m_N-1\}.\ee 
			We denote by $a^{N}_{i,j}$ the coefficient vector of the $i$-homogeneous part of $P_j^N$. Since $\mathcal{P}_{\le k}$ is finite-dimensional, we have $$|a^N_{i,j+1}-a^N_{i,j}|\le C(2^{-j})^{k+\alpha-i}\qquad\forall i\in\{0,\ldots ,k\},\ \forall j\in\{0,\ldots,m_N-1\}.$$
			Summing the geometric series from $j$ to $m_N-1$, we get
			\be\label{eq:bound-a}|a^N_{i,m_N}-a^N_{i,j}|\le C(2^{-j})^{k+\alpha-i}\qquad\forall i\in\{0,\ldots, k\},\ \forall j\in\{0,\ldots,m_N-1\}.\ee
			Then, using again that the space $\mathcal{P}_{\le k}$ is finite-dimensional, we deduce that
			\bea\label{eq:0est-improvement}\left(\int_{\Omega^\eta_N\cap Q_1}(P_{m_N}^{N}-P_{j}^{N})^2(2^{-j}x,4^{-j}t)\,dx\,dt\right)^\frac12\le C (2^{-j})^{k+\alpha}\qquad\forall j\in\{0,\ldots,m_N\}.\eea 
			Combining this estimate with \eqref{eq:contrad-ass-improvement}, we obtain
			\bea\label{eq:1est-improvement} \left(\int_{\Omega^\eta_{N}\cap Q_1}(\widetilde u^{N}-P_{m_N}^{N})^2(2^{-j}x,4^{-j}t)\,dx\,dt\right)^\frac12\le C (2^{-j})^{k+\alpha}\qquad\forall j\in\{0,\ldots,m_N\}.\eea
			On the other hand, the definition of $P^N_{m_N+1}$ and \eqref{eq:contrad-ass-improvement} with $j=m_N$ give  \bea\left(\int_{\Omega^\eta_N\cap Q_1}(\widetilde u^{N}-P_{m_N+1}^{N})^2(2^{-m_N}x,4^{-m_N}t)\,dx\,dt\right)^\frac12\le (2^{-m_N})^{k+\alpha}.\eea
			Combining this inequality with \eqref{eq:contrad-ass-improvement} for $j=m_N$, we have 
			$$\left(\int_{\Omega^\eta_N\cap Q_1}(P_{m_N+1}^{N}-P_{m_N}^N)^2(2^{-m_N}x,4^{-m_N}t)\,dx\,dt\right)^\frac12\le C(2^{-m_N})^{k+\alpha}.$$
			Arguing as above, we obtain that $$|a^N_{i,m_N+1}-a^N_{i,m_N}|\le C(2^{-m_N})^{k+\alpha-i}\qquad\forall i\in\{0,\ldots, k\}.$$
			Combining this bound with \eqref{eq:bound-a}, we get
			$$|a^N_{i,m_N+1}-a^N_{i,j}|\le C(2^{-j})^{k+\alpha-i}\qquad\forall i\in\{0,\ldots, k\},\ \forall j\in\{0,\ldots,m_N\}.$$
			Therefore, 
			$$\left(\int_{\Omega^\eta_N\cap Q_1}(P_{m_N+1}^{N}-P_{j}^N)^2(2^{-j}x,4^{-j}t)\,dx\,dt\right)^\frac12\le C(2^{-j})^{k+\alpha}\qquad\forall j\in\{0,\ldots,m_N\}.$$
			Finally, combining the previous estimate with \eqref{eq:contrad-ass-improvement}, we get \eqref{eq:contrad-ass-improvement-with-j}.
			
			\smallskip
			\noindent\textit{Step 2: $w^N$ satisfies a growth estimate.} In this step we prove that for every $R\ge1$, there exists $N_R\in\N$ such that for every $N\in\N_{\ge N_R}$, we have    
            \be\label{step2:improvement}\|w^N\|_{L^2(\Omega_N^\eta\cap Q_R)}\le CR^{k+\alpha+\frac{n+3}2},\ee for some $C>0$ independent of $R$ and $N$.   
			
			Let $R\ge1$ and choose $q\in\N$ such that $2^{q-1}< R\le 2^q$. For $N$ large enough, by \eqref{eq:contrad-ass-improvement-with-j} with $j=m_N-q$, we have 
			$$\left(\int_{\Omega^\eta_{N}\cap Q_1}(\widetilde u^{N}-P_{m_N+1}^{N})^2(2^{-m_N}2^{q}x,4^{-m_N}4^{q}t)\,dx\,dt\right)^\frac12\le C(2^{-m_N}2^{q})^{k+\alpha}.$$
			With the change of variables $(y,\tau)=(2^qx, 4^q t)$,  
			one can verify that if $(y,\tau)\in \Omega_N^\eta \cap Q_{2^q}$, then $(x,t)=(2^{-q}y,4^{-q}\tau)\in \Omega_N^\eta\cap Q_1$. Therefore, since $\,dx\,dt= 2^{-q(n+3)}\,dy\,d\tau$, we get
			$$\left(\int_{\Omega^\eta_{N}\cap Q_{2^q}}(\widetilde u^{N}-P_{m_N+1}^{N})^2(2^{-m_N}x,4^{-m_N}t)\,dx\,dt\right)^\frac12\le C(2^{-m_N})^{k+\alpha}(2^{q})^{k+\alpha+\frac{n+3}2}.$$
			By $Q_R\subseteq Q_{2^q}$ and $2^q\le 2R$, the proof of \eqref{step2:improvement} is concluded.  
			
			\smallskip
			\noindent\textit{Step 3: Compactness.}
			In this step we prove that there exists a function $w:S_\infty\to\R$ such that $$w^N\rightharpoonup w\quad\text{in } L^2(K)\qquad \forall K\mathrel{\subset\!\subset}\R^{n+1}\times (-\infty,0),$$ and $w$ satisfies the hypotheses of \cref{lemma:w}.
			The existence of $w$ follows from the fact that for every compact set $K\mathrel{\subset\!\subset} \R^{n+1}\times (-\infty,0)$, $K\subseteq\Omega_N^\eta\cap Q_R$ for $N$ and $R$ large enough. Then, by \eqref{step2:improvement}, the sequence $w^N$ is bounded in $L^2(K)$, and so, up to extracting a subsequence, it converges weakly in $L^2(K)$ to some function $w$. Moreover, by \eqref{step2:improvement} and the lower semicontinuity of the $L^2$-norm, the function $w$ satisfies
			$$\|w\|_{L^2(K)}\le CR^{k+\alpha+\frac{n+3}2}$$
			for some $C$ which does not depend on $K$. 
			Choosing a compact exhaustion of $Q_R$, by the monotone convergence theorem we obtain the growth estimate 
			$$\|w\|_{L^2(Q_R)}\le CR^{k+\alpha+\frac{n+3}2}\qquad\forall R\ge1.$$
			Additionally, for every $K\mathrel{\subset\!\subset} B_1\times(-1,0)$, by the definition of $P^N_{m_N+1}$, we have $$0=\int_{\Omega^\eta_N\cap Q_1} w^N P=\int_{K} w^N P+\int_{(\Omega^\eta_N\cap Q_1)\setminus K} w^N P\qquad\forall P\in\mathcal{P}_{\le k}.$$ 
			In particular, by \eqref{step2:improvement}, $$\left|\int_{K} w^N P\right|\le C\|P\|_{L^2(Q_1\setminus K)}\qquad\forall P\in\mathcal{P}_{\le k}.$$
			Therefore, since $w^N$ converges weakly to $w$ in $L^2(K)$, passing to the limit we can replace $w^N$ with $w$ in the estimate above. Then,
			by the arbitrariness of $K$, we obtain $$\int_{Q_1} w P=0\quad\forall P\in\mathcal{P}_{\le k}.$$

			Next we prove that $w$ is caloric in $S_\infty^+$. We first recall that $$\mathbf{H}\widetilde u^N=f^N\quad \text{ in }Q_1^+\qquad\text{and}\qquad \|f^N\|_{L^\infty(Q_r)}\le Cr^{k+\gamma-2},$$ for every $r\in(0,1)$. Moreover, the polynomials in $\mathcal{P}_{\le k}$ are caloric in $S_\infty^+$. Then, given $R\ge1$, $w^N$ solves
			\be\label{eq:wN-improvement}\mathbf{H} w^N(x,t)=\widetilde f^N(x,t):=\frac{2^{-2m_N}f^N(2^{-m_N}x,4^{-m_N}t)}{({2^{-m_N}})^{k+\alpha}}\qquad\text{in }Q_{R}^+,\ee
			for $N$ large enough. Here, since $\alpha\in(0,\gamma)$,
			\be\label{eq:estfN}\|\widetilde f^N\|_{L^\infty(Q_{R})}=(2^{-m_N})^{2-k-\alpha}
			\|f^N\|_{L^\infty(Q_{R2^{-m_N}})}\le C(R)(2^{-m_N})^{\gamma-\alpha}\to0\qquad\text{as }N\to\infty.\ee
			Passing to the limit as $N\to \infty$ in \eqref{eq:wN-improvement}, we obtain  $\mathbf{H} w=0$ in $Q_R^+$. By the arbitrariness of $R\ge1$, we obtain that $w$ is caloric in $S_\infty^+$.
			
			Finally, we prove that $w\equiv0$ on $S_\infty'$.
			Given $K\mathrel{\subset\!\subset} \R^{n+1}\times(-\infty,0)$, we choose $R\ge1$ such that $K\subseteq Q_R$. Then, there exists $N$ large enough such that if $(x,t)\in K$, then $(2^{-m_N}x,4^{-m_N}t)\in\Omega_N^\eta\cap Q_1$. Since the polynomials in $\mathcal{P}_{\le k}$ vanish on $\{x_{n+1}=0\}$ and $\widetilde u^N\equiv0$ in $\Omega^\eta_N\cap Q'_1$, then $w^N\equiv0$ on $K'$ for any $N$ large enough. In $K$, we can consider the odd reflection of $w^N$ with respect to $\{x_{n+1}=0\}$, which satisfies the heat equation with a right-hand side given by the odd reflection of $\widetilde f^N$ with respect to $\{x_{n+1}=0\}$. Then the weak convergence in $L^2(K)$ can be improved to uniform convergence in $K$, by standard parabolic estimates. This gives $w\equiv0$ in $K'$, and so, by the arbitrariness of $K$, we get $w\equiv 0$ on $S_\infty'$.

			\smallskip
			\noindent\textit{Step 4: Contradiction and conclusion.} 
			By \cref{lemma:w}, the limit function $w$ obtained in \textit{Step 3} is identically zero in $S_\infty$.
			We choose $\theta:=2^{-\eta/2}$. After decreasing $\eta$, if necessary, we can suppose that $\theta\in(1/2,1)$.
			Since $\mathbf{H}w^N=\widetilde f^N$ in $Q_2^+$ and $w^N\equiv0$ in $\Omega_N^\eta\cap Q_2'$ (see \textit{Step 3}), then we can apply \cref{lemma:technical-lemma} to the odd reflection of $w^N$ in $\Omega_N^\eta\cap Q_2$, obtaining 
			$$\int_{\Omega_{\theta N}^\eta\cap Q_1}|\nabla w^N|^2+|\partial_t w^N|^2\le C\int_{\Omega_N^\eta\cap Q_2} (w^N)^2+C\|\widetilde f^N\|^2_{L^\infty(Q_2)}$$
			for $N$ large enough.
			By \eqref{step2:improvement} and \eqref{eq:estfN}, the sequence $w^N$ is uniformly bounded in 
			$W^{1,2}(\Omega^{\eta}_{\theta N}\cap Q_{1/2})$. Moreover, by Sobolev inequality \be\label{eq:sobolev-ineq-unif}\|w^N\|_{L^q(\Omega^\eta_{\theta N}\cap Q_{1/2})}\le C,\ee for some $q>2$. Note that the constant in the Sobolev inequality in \eqref{eq:sobolev-ineq-unif} is uniform since the boundary of $\Omega^\eta_N$ is given by a function with uniformly bounded Lipschitz constant.
			
			Now let $K\mathrel{\subset\!\subset} \R^{n+1}\times (-1,0)$, with $K\subseteq Q_{1/2}$ and write
			\be\label{eq:write}\int_{\Omega^\eta_{\theta N}\cap Q_{1/2}}(w^N)^2=\int_{K}(w^N)^2+\int_{(\Omega^\eta_{\theta N}\cap Q_{1/2})\setminus K}(w^N)^2.\ee
			On the one hand, by H\"older inequality along with \eqref{step2:improvement} and \eqref{eq:sobolev-ineq-unif}, we have $$\left|\int_{(\Omega^\eta_{\theta N}\cap Q_{1/2})\setminus K}(w^N)^2\right|\le C|Q_{1/2}\setminus K|^{1-\frac2q}.$$
			On the other hand, $w^N\to 0$ strongly in $L^2(K)$ up to a subsequence by the $W^{1,2}$-boundedness. Therefore, by the arbitrariness of $K$, \eqref{eq:write} gives
			$$\lim_{N\to \infty}\int_{\Omega^\eta_{\theta N}\cap Q_{1/2}}(w^N)^2=0.$$
			This last estimate contradicts \eqref{eq:finalestimate-improv}, concluding the proof.
		\end{proof}

		Now we can iterate \cref{lemma:improvement} to get the higher order expansion of $u$ in \cref{prop:high-order-blow-up}.
		\begin{proof}[Proof of \cref{prop:high-order-blow-up}]
			Using the notation in \cref{sec-subtraction-obstacle}, for $r\in (0,1/2)$, we define $$U_r(x,t):=r^{-3}(u(rx,r^2t)-\widetilde q_k(rx))\quad\text{and}\quad \varphi_r(x'):=r^{-3}(\varphi(rx')-q_k(rx'))$$ so that $\widetilde U_r(x,t)=\widetilde u_r(x,t):=r^{-3}\widetilde u(rx,r^2t)$.
            In order to apply \cref{lemma:improvement}, we notice that $\|U_r\|_{L^\infty(Q_1)}\le C$ and $\|\varphi_r\|_{C^{k,\gamma}(B_1')}\le Cr^{k+\gamma-3}$. Therefore, the threshold $N_0\in\N$ in \cref{lemma:improvement} can be chosen independently of $r$. Enlarging $N_0$ if necessary, we can assume that $K_{r}\subseteq \Omega_{N_0}^\eta\cap Q_1$. 
            By \cref{lemma:u=0}, for any sufficiently small $r_0>0$, we have
			\be\label{eq:u=0proof}\widetilde u_{r_0}\equiv0\qquad\text{on }\Omega_{N_0}^\eta\cap Q'_1.\ee 
			Moreover, since $(0,0)$ is a cubic point, then by \cref{thm:rate-for-u}, there exist $p\in\mathcal{P}_{3}^+\subseteq \mathcal{P}_{\le k}$ and $\alpha_0\in(0,1)$ such that
			$$\|\widetilde u_{r_0}-p\|_{L^\infty(Q_1)}\le Cr_0^{\alpha_0}.$$
			Then, up to choosing $r_0$ small enough depending on $N_0$, we can assume that $$\left(\int_{\Omega^\eta_{N_0}\cap Q_1}(\widetilde u_{r_0}-p)^2(2^{-j}x,4^{-j}t)\,dx\,dt\right)^\frac12\le (2^{-N_0})^{k+\alpha}\le (2^{-j})^{k+\alpha}\qquad\forall j\in \{0,\ldots, N_0\}.$$ Iterating \cref{lemma:improvement}, we get a sequence of polynomials $\{P_j\}_{j\in\N}\subseteq \mathcal{P}_{\le k}$, with $P_j=p$ for every $j\in \{0,\ldots,N_0\}$, such that 
			\be\label{eq:est1-proof-f}\left(\int_{\Omega^\eta_{N_0}\cap Q_1}(\widetilde u_{r_0}-P_j)^2(2^{-j}x,4^{-j}t)\,dx\,dt\right)^\frac12\le (2^{-j})^{k+\alpha}\qquad\forall j\in\N.\ee
			Arguing as in \eqref{eq:allafine}, we get 
			$$\left(\int_{\Omega^\eta_{N_0}\cap Q_1}(P_{j+1}-P_j)^2(2^{-(j+1)}x,4^{-(j+1)}t)\,dx\,dt\right)^\frac12\le C(2^{-(j+1)})^{k+\alpha}\qquad\forall j\in\N.$$ Using that $\mathcal{P}_{\le k}$ is a finite-dimensional space, we get $P_j\to P_\infty\in\mathcal{P}_{\le k}$, where $P_\infty$ satisfies \be\label{eq:est2-proof-f}\left(\int_{\Omega^\eta_{N_0}\cap Q_1}(P_{\infty}-P_j)^2(2^{-j}x,4^{-j}t)\,dx\,dt\right)^\frac12\le C(2^{-j})^{k+\alpha}\qquad\forall j\in\N.\ee
			For every $r\in(0,1)$, we choose $j$ such that $2^{-(j+1)}<r\le 2^{-j}$. Then, combining \eqref{eq:est1-proof-f} and \eqref{eq:est2-proof-f}, we get 
			$$\left(\int_{\Omega^\eta_{N_0}\cap Q_1}(\widetilde u_{r_0}-P_\infty)^2(rx,r^2t)\,dx\,dt\right)^\frac12\le Cr^{k+\alpha}.$$
			Finally, by \eqref{eq:u=0proof} we have $\widetilde u_{r_0}\equiv0$ in $K_{r}'$ for every $r\in(0,1)$, since $K_{r}\subseteq \Omega_{N_0}^\eta\cap Q_1$. In particular, the odd extension of $\widetilde{u}_{r_{0}}-P_{\infty}$ satisfies the heat equation with right-hand side the odd extension of $r_{0}^{-1}f(r_{0}\cdot,r_{0}^{2}\cdot)$ in $K_{r}$. Then, by standard parabolic estimates (see \cref{lemma:wang-start}), the above $L^2$-estimate and the bound for the right-hand side $f$ from \eqref{controlf}, we deduce
			\begin{equation*}
				\lVert \widetilde u_{r_0}-P_\infty\rVert_{L^{\infty}(K_{r})}\le Cr^{k+\alpha}\qquad \forall r\in (0,1).
			\end{equation*}
			Rescaling back, the conclusion follows by choosing $P_{\le k}(x,t):=r_0^3P_\infty(x/r_0,t/r_0^2)$.
		\end{proof}

		\section{Fine stratification of cubic points}\label{sec-time-like-cubic}
		In this section we study the fine structure of the set of cubic points $\Gamma_3$. We first prove generic regularity for cubic points in the dimensions allowed by the regularity of the obstacle, and then split the time-like stratum into a part with improved cleaning and a part with lower-dimensional spatial projection. Both results rely on the higher order expansion provided by \cref{prop:high-order-blow-up}.

        \smallskip
		The first result of this section is about generic regularity for $\Gamma_{3}$. This is a consequence of the higher order expansion given by \cref{prop:high-order-blow-up} combined with Sard's theorem.
		
		\begin{theorem}\label{prop:sard}
			Let $u$ be a solution of \eqref{eq:parabolic-signorini-problem} in dimension $n+1\le k$, with obstacle $\varphi\in C^{k,\gamma}(B_1')$ satisfying \eqref{eq:ostacolo}. Then \be\label{eq:generic-cubic}\mathcal{H}^1(\pi_t(\Gamma_3))=0.\ee In particular, if $\varphi\in C^\infty(B_1')$, then \eqref{eq:generic-cubic} holds in any dimension $n+1\ge 2$.
		\end{theorem}

        In order to state the next results, we start by introducing some notation.
		For any non-trivial cubic blow-up $p\in \mathcal{P}_{3}^{+}$, the normal derivative $x'\mapsto -\partial_{n+1}p(x',0,0)$ is a nonnegative quadratic polynomial in the $x'$-variable. The spine of $p$ is defined as the null vector space of $-\partial_{n+1}p(x',0,0)$, which also coincides with the space of its invariant directions:
		\begin{equation}\label{eq:spine-of-p-cubic}
			S(p):= \left\{\xi\in \R^{n}: -\partial_{n+1}p(\xi,0,0)=0\right\}=\left\{\xi\in \R^{n}: -\partial_{\xi}\partial_{n+1}p(\cdot, 0, 0) \equiv 0\right\}.
		\end{equation}
		The dimension of the spine is denoted by 
		\begin{equation*}
			d(p):= \dim S(p)\in \{0,\dots,n\}.
		\end{equation*}
		Let $u$ be a solution of \eqref{eq:parabolic-signorini-problem}, with obstacle $\varphi$ satisfying \eqref{eq:ostacolo}. By \cref{thm:rate-for-u}, every point $(x_{0},t_{0})\in \Gamma_{3}$ admits a unique non-zero blow-up $p^{x_0,t_0}\in\mathcal{P}_{3}^{+}$. We may then split $\Gamma_{3}$ according to the value of the dimension of the spine of its blow-up:
		\begin{equation}\label{eq:slitting-gamma3}          \Gamma_{3}=\bigcup_{j=0}^{n}\Gamma_{3}^{j},\qquad \Gamma_{3}^{j}:= \left\{(x_{0},t_{0})\in \Gamma_{3}: d(p^{x_0,t_0})=j\right\}\quad \forall j\in \{0,\dots,n\}.
		\end{equation}
		Notice that $\Gamma_3^n$ coincides with the set of cubic points for which the blow-up is of the form $a|x_{n+1}|(6t+x_{n+1}^2)$ for some $a>0$. 
		The points in this set are called \emph{time-like}.
		In contrast, the points belonging to $\Gamma_3^j$ for $j\in \{0,\ldots,n-1\}$ have blow-ups depending also on at least one tangential variable.

		Next, we study the structure of the strata $\Gamma_{3}^{j}$, for any $j\in\{0,\ldots,n\}$.
		
		\begin{proposition}\label{prop:stratification-for-u}
			Let $u$ be a solution of \eqref{eq:parabolic-signorini-problem}, with obstacle $\varphi$ satisfying \eqref{eq:ostacolo}. Consider the splitting of \,$\Gamma_3$ in \eqref{eq:slitting-gamma3}. Then, for every $\alpha\in(0,\gamma)$ and $j\in \{0,\dots,n\}$, 
			$\Gamma_{3}^j$ is locally contained in a $j$-dimensional manifold of class $C^{k-2,\alpha}$ of the form $(x_{j+1},\ldots,x_{n},t)=g(x_1,\ldots,x_j)$, after a suitable rotation in the thin space. 
            In particular 
			\bea\label{eq:n-1bound}\dim_{\mathcal{H}}(\pi_x(\Gamma_3\setminus \Gamma_3^n))\le n-1.\eea
            Moreover, if $\varphi\in C^\infty(B_1')$, then the covering manifolds are $C^\infty$.
		\end{proposition}

		\cref{prop:stratification-for-u} follows from standard arguments based on the Whitney extension theorem, once uniqueness
and classification of blow-ups are known. 
		Observe that \cref{prop:stratification-for-u} does not give any information about the dimension of the spatial projection of the set of time-like cubic points $\Gamma_{3}^{n}$.
		In the following theorem we obtain a finer stratification of $\Gamma_{3}^{n}$.  
		\begin{theorem}\label{prop:complete-time-like}
			Let $u$ be a solution of \eqref{eq:parabolic-signorini-problem}, with obstacle $\varphi\in C^{k,\gamma}(B_1')$ satisfying \eqref{eq:ostacolo}. Then, there exists a subset $\Gamma_3^{n,\ge k}\subseteq \Gamma_3^n$ such that, for any $\alpha\in(0,\gamma)$,
			$$\Gamma_3^{n,\ge k}\in \text{Clean}(k+\alpha-1)\qquad\text{and}\qquad\text{dim}_{\HH}(\pi_x(\Gamma_3^n\setminus \Gamma_3^{n,\ge k}))\le n-1.$$ If $\varphi\in C^\infty(B_1')$, then there is $\Gamma_3^{n,\infty}\subseteq \Gamma_3^n$ such that 
			$$\Gamma_3^{n,\infty}\in \text{Clean}(\infty)\qquad\text{and}\qquad\text{dim}_{\HH}(\pi_x(\Gamma_3^n\setminus \Gamma_3^{n,\infty}))\le n-1.$$
		\end{theorem}
	
		Roughly speaking, the set $\Gamma_3^{n,\ge k}$ is given by those cubic free boundary points around which the solution has a time-like behavior up to the $k$-th order expansion (see \eqref{eq:gamma3ngek} below).
		\subsection{The higher order expansion of the normal derivative} 
		For every $j\in\N$, we let $\mathcal{Q}_{j}$ be the set of parabolically $j$-homogeneous polynomials of $n+1$ variables: \be\label{eq:Qj}\mathcal{Q}_{j}:=\left\{Q=Q(x',t):\R^{n}\times \R\to \R \text{ parabolically $j$-homogeneous polynomial}\right\}.\ee 
		
		Let $u$ be a solution of \eqref{eq:parabolic-signorini-problem}, with obstacle $\varphi\in C^{k,\gamma}(B_{1}')$ satisfying \eqref{eq:ostacolo}.
		For every $(x_0,t_0)\in \Gamma_3$, \cref{prop:high-order-blow-up} gives a unique polynomial ansatz $P^{x_0,t_0}_{\le k}\in\mathcal{P}_{\le k}$ such that
		\be\label{eq:estimateLinfty} \|\widetilde u^{x_0,t_0}-P_{\le k}^{x_0,t_0}\|_{L^\infty(K_r)}\le Cr^{k+\alpha}\qquad\forall r\in(0,r_0),\ee
		where $\widetilde u^{x_0,t_0}$ is defined in \eqref{eq:utilde} and $K_r$ is the backward cylinder introduced in \eqref{Kr}. 
		We decompose uniquely $P^{x_0,t_0}_{\le k}=\sum_{j=3}^k P^{x_0,t_0}_j$, where $P^{x_0,t_0}_j$ are the even reflections of parabolically $j$-homogeneous caloric polynomials on $S_{\infty}^{+}$. Then, we set 
        $$Q_j^{x_0,t_0}(x',t):=\partial_{n+1}P_{j+1}^{x_0,t_0}(x',0,t)\in\mathcal{Q}_j\qquad\forall j\in \{2,\ldots,k-1\}.$$ 
		Since $\widetilde u^{x_0,t_0}\equiv0$ in $K_r'$ by \cref{lemma:u=0}, and $\partial_{n+1}\widetilde u^{x_0,t_0}=\partial_{n+1} u(x_0+\cdot,t_0+\cdot)$ on the thin space, then, for any $\alpha\in(0,\gamma)$, \eqref{eq:estimateLinfty} and standard parabolic estimates give  
		\be\label{eq:expansion-normal-derivative} \Bigg\|\partial_{n+1} u(x_0+\cdot,t_0+\cdot)-\sum_{j=2}^{k-1}Q_j^{x_0,t_0}\Bigg\|_{L^\infty(K_r')}\le Cr^{k+\alpha-1}\qquad\forall r\in(0,r_0),\ee 
		for some $C>0$ and $r_{0}\in (0,1)$.
		We may then define the map
		\be\label{eq:map-Pi}\Pi:\Gamma_3\to \mathcal{Q}_2\times\ldots\times\mathcal{Q}_{k-1}\ee  
		which assigns to any $(x_{0},t_{0})\in\Gamma_3$ the list of homogeneous modes in the expansion of the normal derivative on $\{x_{n+1}=0\}$:
		$$\Pi(x_0,t_0)=(\Pi_2(x_0,t_0),\ldots,\Pi_{k-1}(x_0,t_0))=(Q_2^{x_0,t_0},\ldots, Q_{k-1}^{x_0,t_0}).$$ 
		Observe that $\Pi_{2}(x_{0},t_{0})(\cdot,0)\equiv 0$ if and only if $(x_{0},t_{0})\in \Gamma_{3}^{n}$.
		
		\subsection{Localization of cubic points in a regular time graph}
		The main result of this subsection is the following proposition stating that $\Gamma_{3}$ is locally contained in the graph of a $C^{k-1,\alpha}$ function $t=T(x',0)$ whose gradient vanishes on $\pi_{x}(\Gamma_{3})$. 
		The new information with respect to \cref{theorem:structure-spacetime} is the improved regularity of the time graph.

		\begin{proposition}\label{prop:t=T(x)}
			Let $u$ be a solution of \eqref{eq:parabolic-signorini-problem}, with obstacle $\varphi\in C^{k,\gamma}(B_{1}')$ satisfying \eqref{eq:ostacolo}.
			For every $(x_0,t_0)\in\Gamma_3$ there are $\rho>0$ and $T:B'_\rho(x_0)\to \R$ of class $C^{k-1,\alpha}(B'_\rho(x_0))$ for any $\alpha\in (0,\gamma)$, such that
			\bea\label{eq:construction-G-corollary}\Gamma_3\cap \mathscr{C}_\rho'(x_0,t_0)\subseteq \{t=T(x',0)\}\eea
			and such that, for every $x_{1}\in \pi_{x}(\Gamma_3\cap \mathscr{C}_\rho'(x_0,t_0))$, we have     $$\nabla_{x'}T(x_1)=0\qquad\text{and}\qquad \ker \left(D_{x'}^2 T(x_1)\right)=S(p^{x_1,T(x_1)}),$$
			where $S(p)$ is the spine defined in \eqref{eq:spine-of-p-cubic}. Moreover, if $\varphi\in C^{\infty}(B_{1}')$, then $T$ is $C^{\infty}$. 
		\end{proposition}

		The proof of \cref{prop:t=T(x)} is based on the following two lemmas. 
		First, we show that cubic points are locally contained in the zero set of a space-time function $\Theta$, constructed by Whitney extension from the polynomial expansion of the normal derivative of the solution.  
		\begin{lemma}\label{lemma:key-time-like}
			Let $u$ be a solution of \eqref{eq:parabolic-signorini-problem}, with obstacle $\varphi\in C^{k,\gamma}(B_{1}')$ satisfying \eqref{eq:ostacolo}.
			For every $(x_0,t_0)\in\Gamma_3$, there are $\rho>0$ and a function $\Theta:\mathscr{C}'_\rho(x_0,t_0)\to \R$ of class $H^{k+\alpha-1,\frac{k+\alpha-1}{2}}(\mathscr{C}'_\rho(x_0,t_0))$ for any $\alpha\in (0,\gamma)$, such that 
			\bea\label{eq:construction-G}\Gamma_3\cap \mathscr{C}'_\rho(x_0,t_0)\subseteq \{\Theta=0\}\eea 
			and such that, for every $(x_1,t_1)\in \Gamma_3\cap \mathscr{C}'_\rho(x_0,t_0)$ the following hold:
			\be\label{eq:prop-of-theta}\Theta(x_1,t_1)=|\nabla_{x'}\Theta(x_1,t_1)|=0\quad \text{and}\quad \partial_{x'}^\beta\partial^j_t\Theta(x_1,t_1)=\partial^\beta_{x'}\partial_t^j Q_{|\beta|+2j}^{x_1,t_1}(0,0)\qquad \text{for $2\le |\beta|+2j\le k-1$}.\ee 
            Moreover, if $\varphi\in C^{\infty}(B_{1}')$, then $\Theta$ is $C^{\infty}$.
		\end{lemma}

		\begin{proof}
			Let $(x_0,t_0)\in\Gamma_3$.
			For any multiindex $(\beta,j)\in \N^n\times\mathbb N$ such that $|\beta|+2j\le k-1$, we define the coefficient $f_{\beta,j}:\Gamma_3\to \R$ as
			\begin{equation*}
				f_{\beta,j}(x_{1},t_{1}):=
				\begin{cases}
					0\qquad &\text{if $0\le |\beta|+2j\le 1$},\\
					\partial^\beta_{x'}\partial_t^j Q_{|\beta|+2j}^{x_1,t_1}(0,0)\qquad &\text{if $2\le |\beta|+2j\le k-1$}.
				\end{cases}
			\end{equation*}
			We claim that the family $\{f_{\beta,j}\}_{|\beta|+2j\le k-1}$ satisfies the Whitney compatibility condition in the statement of \cite[Theorem B.1]{dgpt17} with a H\"older modulus of continuity in the set $\Gamma_{3}\cap \mathscr{C}'_{\rho}(x_{0},t_{0})$ for a sufficiently small $\rho>0$. Assuming the claim, the theorem provides us with a function $\Theta\in H^{k+\alpha-1,\frac{k+\alpha-1}{2}}(\mathscr{C}'_\rho(x_0,t_0))$ such that 
            $$\partial^\beta_{x'}\partial_t^j\Theta=f_{\beta,j}\qquad\text{in }\Gamma_3\cap \mathscr{C}'_\rho(x_0,t_0),$$ concluding the proof. 
			
			In order to prove the claim, let us consider the jets
			\begin{align*}
				J_{k-1}^{x_1,t_1}(x',t)&:=Q_{2}^{x_1,t_1}(x'-x_1',t-t_1)+\ldots+Q_{k-1}^{x_1,t_1}(x'-x_1',t-t_1)\\
				&=\sum_{|\beta|+2j\le k-1}\frac{f_{\beta,j}(x_1,t_1)}{\beta!j!}(x'-x_1')^\beta(t-t_1)^j.
			\end{align*}
			The expansion of the normal derivative \eqref{eq:expansion-normal-derivative} becomes \be\label{eq:expansion-final}\left\|\partial_{n+1} u(x',0,t)-J_{k-1}^{x_1,t_1}(x',t)\right\|_{L^\infty((x_1,t_1)+K_r')}\le Cr^{k+\alpha-1} \qquad \forall  r\in (0,r_{0}).\ee 
			Let $(x_1,t_1),(x_2,t_2)\in \Gamma_3\cap \mathscr{C}'_\rho(x_0,t_0)$ and assume without loss of generality that $t_1\le t_2$. We set $$r:=2\,\text{dist}_{\text{par}}((x_1,t_1),(x_2,t_2))=2\sqrt{|x_1'-x_2'|^2+|t_1-t_2|},$$ and we consider the cylinder $$\mathcal{C}'_r:=B'_{r/2}(x_1)\times ( t_1-r^2/2,t_1-r^2/4).$$ Then one can verify that $ \mathcal{C}'_r\subseteq \Big((x_1,t_1)+K'_r\Big)\cap \Big((x_2,t_2)+K'_r\Big).$ Therefore, inside $\mathcal{C}_r'$, we can subtract the two expansions \eqref{eq:expansion-final} centered at $(x_1,t_1)$ and $(x_2,t_2)$ to get the compatibility of jets
			\be\label{witney:compat1}\|J_{k-1}^{x_1,t_1}-J_{k-1}^{x_2,t_2}\|_{L^\infty(\mathcal{C}_r')}\le Cr^{k+\alpha-1}.\ee 
			In addition, $$J_{k-1}^{x_1,t_1}(x',t)-J_{k-1}^{x_2,t_2}(x',t)=\sum_{|\beta|+2j\le k-1}\frac{g_{\beta,j}(x_1,x_2,t_1,t_2)}{\beta!j!}(x'-x_1')^\beta(t-t_1)^j,$$
			where
			$$g_{\beta,j}(x_1,x_2,t_1,t_2):=f_{\beta,j}(x_1,t_1)-\sum_{|\sigma|+2i\le k-1-|\beta|-2j}\frac{f_{\beta+\sigma,j+i}(x_2,t_2)}{\sigma!i!}(x_1'-x_2')^\sigma(t_1-t_2)^i. $$ Then, by the compatibility of the jets \eqref{witney:compat1} and using the fact that the space of parabolic polynomials of degree less than or equal to $k-1$ is finite-dimensional (arguing as in the proof of \cref{lemma:improvement}), we obtain $$|g_{\beta,j}(x_1,x_2,t_1,t_2)|\le Cr^{k+\alpha-1-|\beta|-2j}\le C\text{dist}_{\text{par}}((x_1,t_1),(x_2,t_2))^{k+\alpha-1-|\beta|-2j}.$$
			Then the family $\{f_{\alpha,j}\}_{|\alpha|+2j\le k-1}$ satisfies the Whitney compatibility in \cite[Theorem B.1]{dgpt17}, with H\"older modulus of continuity, producing a function $\Theta$ of class $H^{k+\alpha-1,\frac{k+\alpha-1}2}$. This concludes the proof for $C^{k,\gamma}$ obstacles. 
            If $\varphi\in C^\infty(B_1')$, the same argument yields compatible jets of all orders, so the $C^\infty$ Whitney extension theorem provides a smooth function $\Theta$.
		\end{proof}

		Next, we show that blow-ups at cubic points must have non-trivial time dependence.
		As a consequence, the function $\Theta$ constructed in \cref{lemma:key-time-like} satisfies $\partial_t\Theta\not=0$ at points of $\Gamma_{3}$.
		\begin{lemma}\label{lemma:p-depend-time}
			Let $u$ be a solution of \eqref{eq:parabolic-signorini-problem}, with obstacle $\varphi$ satisfying \eqref{eq:ostacolo}, and suppose that $(0,0)\in\Gamma_3$. Let $p\in\mathcal{P}_3^+$ be the blow-up of $u$ at $0$, as in \cref{thm:rate-for-u}. Then $\partial_tp\not\equiv0$. As a consequence, $\partial_t\partial_{n+1}p(0,0)\not=0$.
		\end{lemma}
		\begin{proof}
			Suppose by contradiction that $p$ is time-independent. We denote by $\widetilde v=\widetilde u-p$, where $\widetilde u:=\widetilde u^{0,0}$ is defined in \eqref{eq:utilde}. Then, by \cref{lemma:monotonicity-ass}, $$\partial_t\widetilde v=\partial_t\widetilde u=\partial_tu \ge c|x_{n+1}|\qquad\text{in }Q_{1/2}.$$
			Applying this estimate to $(x',x_{n+1})=(0,r)$, we have $$\widetilde v(0,r,0)-\widetilde v(0,r,-r^2)=\int_{-r^2}^{0}\partial_t\widetilde v(0,r,s)\,ds\ge c r^3.$$
			On the other hand, by the rate of convergence at cubic points in \cref{thm:rate-for-u}, we have $$\|\widetilde v\|_{L^\infty(Q_r)}\le Cr^{3+\alpha}.$$ Then
			$|\widetilde v(0,r,0)|\le Cr^{3+\alpha}$ and $|\widetilde v(0,r,-r^2)|\le Cr^{3+\alpha}$. Thus $$2Cr^{3+\alpha}\ge cr^3,$$ which is a contradiction for $r$ small enough. This shows that $ \partial_tp\not\equiv0$, namely $p$ is time-dependent. Since also $p$ is of the form \eqref{eq:expr-p-odd} and cubic, we obtain that $q_1\in\R$ and thus $\partial_{t}q_0\not\equiv0$. This implies that $\partial_t\partial_{n+1}p(0,0)\not=0$, concluding the proof. 
		\end{proof}
		
		We are now in the position to prove the main result of this subsection.
		\begin{proof}[Proof of \cref{prop:t=T(x)}]
			Let $(x_0,t_0)\in \Gamma_3$, and let $\rho>0$ and $\Theta\in H^{k+\alpha-1,\frac{k+\alpha-1}{2}}(\mathscr{C}'_\rho(x_0,t_0))$ be given by \cref{lemma:key-time-like}.
			Then, by \cref{lemma:key-time-like} and \cref{lemma:p-depend-time}, for every $(x_1,t_1)\in\Gamma_3\cap \mathscr{C}'_\rho(x_0,t_0),$ we have
			\be\label{eq:boh1}\nabla_{x'}\Theta(x'_1,0,t_1)=0\qquad\text{and}\qquad\partial_t \Theta(x'_1,0,t_1)=\partial_t Q_2^{x_1,t_1}(0,0)=\partial_t\partial_{n+1}p^{x_1,t_1}(0,0)\not=0.\ee
			After further decreasing  $\rho$ if necessary, we can assume that
            $|\partial_t\Theta|\ge c$ in $\mathscr{C}'_\rho(x_0,t_0)$, for some $c>0$. By the implicit function theorem, we can find a function $T:B_\rho'(x_0)\to\R$ of class $C^{1}$ such that 
			\be\label{eq:boh2}\Gamma_3\cap \mathscr{C}'_\rho(x_0,t_0)\subseteq \{\Theta(x',0,t)=0\}=\{t=T(x',0)\}.\ee
			Moreover, $$\nabla_{x'} T(x',0)=-\frac{\nabla_{x'}\Theta(x',0,T(x',0))}{\partial_t\Theta(x',0,T(x',0))}=0\qquad\forall (x',0)\in \pi_x(\Gamma_3\cap \mathscr{C}'_\rho(x_0,t_0)).$$
			Note that if $\varphi\in C^\infty(B_1')$, then \cref{lemma:key-time-like} gives the smoothness of $\Theta$, and the implicit function theorem directly yields a smooth time graph $T$.
            Thus, in order to conclude the proof for $C^{k,\gamma}$ obstacles, we only need to show that $T\in C^{k-1,\alpha}(B'_\rho(x_0))$.
			To this end, we use as a key ingredient the quadratic (actually superquadratic) cleaning of $\Gamma_3$ from \cref{prop:cleaning-odd}.
			
			For $x_1\in \pi_x(\Gamma_3\cap \mathscr{C}'_\rho(x_0,t_0))$, let us consider the jets of $\Theta$:
			$$ J^{x_1}_{k-1}(x',t):=\sum_{|\beta|+2j\le k-1}\frac{\partial_{x'}^\beta \partial^j_t\Theta(x_1,T(x_1))}{\beta!j!}(x'-x_1')^\beta(t-T(x_1))^j.$$
			The Whitney estimate for $\Theta$ gives $$|\Theta(x',0,t)- J^{x_1}_{k-1}(x',t)|\le C\text{dist}_{\text{par}}((x,t),(x_1,T(x_1)))^{k+\alpha-1}
			\qquad \forall (x,t)\in \mathscr{C}_{\rho}'(x_0,t_0).$$ 
             Given $x_1,x_2\in \pi_x(\Gamma_3\cap \mathscr{C}'_\rho(x_0,t_0))$, we set $r:=|x_1-x_2|$. By \eqref{eq:boh1} and \eqref{eq:boh2}, we have 
			\bea c|T(x_1)-T(x_2)|&\le |\Theta(x_1,T(x_1))-\Theta(x_1,T(x_2))|\\&=
			|\Theta(x_1,T(x_2))-\Theta(x_2,T(x_2))-\nabla_{x'} \Theta(x_2,T(x_2))\cdot (x_{1}'-x_{2}')|\le C|x_1-x_2|^2.\eea 
            As a consequence,
            \bea\label{eq:reggsdsd}\text{dist}_{\text{par}}((x_1,T(x_1)),(x_2,T(x_2)))\le Cr.\eea
            Therefore, on a common parabolic cylinder $\mathcal{C}_r'$ lying at distance comparable with $r$ from both $(x_1,T(x_1))$ and $(x_2,T(x_2))$, one has
			\be\label{eq:comp-jet}\|J_{k-1}^{x_1}-J_{k-1}^{x_2}\|_{L^\infty(\mathcal{C}_r')}\le Cr^{k+\alpha-1}.\ee
			Since $\partial_t J_{k-1}^{x_1}=\partial_t\Theta\not=0$ at $(x_1,T(x_1))$, the zero set of $J^{x_1}_{k-1}$ is locally the graph of a smooth function $\tau^{x_1}$ satisfying $$ J_{k-1}^{x_1}(x',\tau^{x_1}(x))=0\qquad\text{and}\qquad \tau^{x_1}(x_1)=T(x_1).$$
			By finite-dimensional norm equivalence after parabolic rescaling, \eqref{eq:comp-jet} also holds on a comparable parabolic cylinder containing both graphs over a ball $D_r$ of radius comparable to $r$. Together with the uniform lower bound on the time derivatives of $J^{x_1}$, this yields the following estimate on $D_r$:
            \begin{align*}
                |\tau^{x_1}(x)-\tau^{x_2}(x)|&\le C |J_{k-1}^{x_1}(x,\tau^{x_1}(x))-J_{k-1}^{x_1}(x,\tau^{x_2}(x))|=C|J_{k-1}^{x_1}(x,\tau^{x_2}(x))-J_{k-1}^{x_2}(x,\tau^{x_2}(x))|\\&\le Cr^{k+\alpha-1}.
            \end{align*}
			Let $P_{k-1}^{x_i}$ be the Taylor polynomial of order $k-1$ of $\tau^{x_i}$ at $x_i$, for $i=1,2$. By uniform $C^k$-estimates, we have $\|\tau^{x_i}-P_{k-1}^{x_i}\|_{L^\infty(D_r)}\le Cr^k$.
            Therefore, since $\alpha<1$, we have
			$$\|P_{k-1}^{x_1}-P_{k-1}^{x_2}\|_{L^\infty(D_r)}\le \|\tau^{x_1}-P_{k-1}^{x_1}\|_{L^\infty(D_r)}+\|\tau^{x_2}-P_{k-1}^{x_2}\|_{L^\infty(D_r)}+\|\tau^{x_1}-\tau^{x_2}\|_{L^\infty(D_r)}\le Cr^{k+\alpha-1}.
			$$
            By finite-dimensional norm equivalence, the coefficients of the polynomials $P_{k-1}^{x_1}$ satisfy the Whitney compatibility conditions of class $C^{k-1,\alpha}$ on $\pi_x(\Gamma_3\cap \mathscr{C}'_\rho(x_0,t_0))$. Hence the classical Whitney extension theorem gives a function $\widetilde T:B_\rho'(x_0)\to\R$ of class $C^{k-1,\alpha}$ such that $\widetilde T(x_1)=\tau^{x_1}(x_1)=T(x_1)$.
            In particular, \bea\Gamma_3\cap \mathscr{C}_\rho'(x_0,t_0)\subseteq \{t=\widetilde T(x',0)\}.\eea Moreover, differentiating $J_{k-1}^{x_1}(x',\tau^{x_1}(x))=0$ and using that $\nabla_{x'}\Theta(x_1,T(x_1))=0$, we get $\nabla_{x'}\tau^{x_1}(x_1)=0$. Consequently, $\nabla_{x'} \widetilde T(x_1)=\nabla_{x'}\tau^{x_1}(x_1)=0$. Differentiating once more, we obtain $$D_{x'}^2 \widetilde T(x_1)=D_{x'}^2\tau^{x_1}(x_1)=-\frac{D^2_{x'}Q_2^{x_1,\widetilde T(x_1)}(0,0)}{\partial_tQ_2^{x_1,\widetilde T(x_1)}(0,0)}.$$ Consequently, $$\ker \left(D_{x'}^2 \widetilde T(x_1)\right)=\ker \left(D_{x'}^2 Q_2^{x_1, \widetilde T(x_1)}(0,0)\right)=S(p^{x_1,\widetilde T(x_1)}).$$ Renaming $\widetilde T$ as $T$, we conclude the proof.
		\end{proof}

		\subsection{Generic regularity and stratification of cubic points} In this subsection we use \cref{prop:t=T(x)} to prove \cref{prop:sard} and \cref{prop:stratification-for-u}.
		Generic regularity for cubic points is a consequence of Sard's theorem.
		\begin{proof}[Proof of \cref{prop:sard}]
			Let $(x_0,t_0)\in \Gamma_3$ and let $\rho>0$ and $T\in C^{k-1,\alpha}(B'_\rho(x_0))$ be given by \cref{prop:t=T(x)}.
			Then
			$$\Gamma_3\cap \mathscr{C}'_\rho(x_0,t_0)\subseteq \left\{(x',0,t)\in \mathscr{C}'_\rho(x_0,t_0): \ t=T(x',0),\ \nabla_{x'}T(x',0)=0\right\}.$$ Therefore, $$\pi_t(\Gamma_3\cap \mathscr{C}'_\rho(x_0,t_0))\subseteq T\left(\{\nabla_{x'}T=0\}\cap B_{\rho}'(x_0)\right).$$ 
			Since $T$ is $C^{k-1}$-regular, by Sard's theorem (see for instance \cite[Theorem 3.4.3]{federer}), the $1$-dimensional Lebesgue measure of the set in the right-hand side is $0$, provided that $n\le k-1$. Therefore, $$\mathcal{H}^1(\pi_t(\Gamma_3\cap \mathscr{C}'_\rho(x_0,t_0)))=0.$$ Covering $\Gamma_{3}$ with countably many cylinders as above concludes the proof. 
		\end{proof}
		
		Next we prove the stratification result for cubic points in \cref{prop:stratification-for-u}.
		\begin{proof}[Proof of \cref{prop:stratification-for-u}]
			For $j=n$, the result follows from \cref{prop:t=T(x)}. Then, we assume that $j<n$. Let $(x_0,t_0)\in\Gamma_3^{j}$, for some $j\in\{0,\ldots,n\}$, and let $T$ be the function given by \cref{prop:t=T(x)}. For every $i\in \{1,\ldots,n\}$, we set $F_i(x'):=\partial_{x_i} T(x',0)$. 
			Then, by \cref{prop:t=T(x)}, there is some $\rho>0$ for which  
            \be\label{eq:inclusion-gamma-new}\Gamma_{3}^{j}\cap \mathscr{C}'_\rho(x_0,t_0)\subseteq \left\{t=T(x',0), \ F_i(x')=0 \ \ \forall i\in\{1,\ldots,n\} \right\}.\ee
			Consider the map $F(x')=(F_1(x'),\ldots, F_n(x'))$. Then, again by \cref{prop:t=T(x)} the kernel of the linear map $D_{x'}F(x_0')$ coincides with the spine of $p$: $$\ker\left(D_{x'}F(x_0')\right)=S(p^{x_0,T(x_0)}),$$ where $S(p)$ is the spine defined in \eqref{eq:spine-of-p-cubic}. Since $(x_0,t_0)\in\Gamma_3^j$, the spine $S(p^{x_0,T(x_0)})$ has dimension $j$. Therefore, $${\rm rank}\left(D_{x'}F(x_0')\right)=n-j.$$ Then, there exist distinct indices $i_1,\ldots, i_{n-j}\in\{1,\ldots,n\}$ such that $\nabla_{x'} F_{{i_1}}(x_0'),\ldots,\nabla_{x'} F_{{i_{n-j}}}(x_0')$ are linearly independent. We define the function 
			\bea\Psi(x'):=\left(F_{{i_1}}(x'),\ldots, F_{{i_{n-j}}}(x')\right).\eea
			By \eqref{eq:inclusion-gamma-new}, around $(x_0,t_0)$, the set $\Gamma_3^j$ is contained in the zero set $\{t=T(x',0),\ \Psi=0\}$. Moreover, by the implicit function theorem, the zero set $\{\Psi=0\}$ is locally a $j$-dimensional manifold of class $C^{k-2,\alpha}$ of the form $(x_{j+1},\ldots,x_n)=\widetilde g(x_1,\ldots,x_j)$. By defining $$g(x_1,\ldots,x_j):=\Big(\widetilde g(x_1,\ldots,x_j),T(x_1,\ldots,x_j,\widetilde g(x_1,\ldots,x_j),0)\Big),$$ we conclude the construction of the manifolds.
			Finally, by a covering argument, for every $j\in\{0,\ldots,n-1\}$,
			$$\text{dim}_{\HH}(\pi_x(\Gamma_3^{j}))\le \text{dim}^{\text{eucl}}_{\HH}(\Gamma_3^{j})\le j\le n-1,$$ where $\text{dim}^{\text{eucl}}_{\HH}$ is the dimension with respect to the usual Euclidean distance. 
            Finally, if $\varphi$ is smooth, \cref{prop:t=T(x)} gives that $T$ is $C^\infty$, and thus the zero set $\{\Psi=0\}$ is locally smooth, concluding the proof.
		\end{proof}

			\subsection{Splitting of the time-like cubic points} 
			In order to prove \cref{prop:complete-time-like}, we introduce the following subset of parabolically $j$-homogeneous polynomials: 
			$$\mathcal{Q}_j^{0}:=\{Q\in \mathcal{Q}_{j}: \ Q(x',0)=0\quad \forall x'\in\R^n\}\subseteq \mathcal{Q}_j,$$ where $\mathcal{Q}_j$ is defined in \eqref{eq:Qj}.

			The set of time-like cubic points $\Gamma_3^n$ is further decomposed as follows. 
			We consider the map $\Pi$ as defined in \eqref{eq:map-Pi}. For any $m\in\{3,\ldots,k-1\}$ we set
			\bea \Gamma_3^{n,m}:=\left\{(x_0,t_0)\in \Gamma_3^n:\ \Pi_j(x_0,t_0)\in\mathcal{Q}_{j}^{0} \quad \forall j\in \{2,\ldots,m-1\}, \quad \Pi_{m}(x_0,t_0)\not\in\mathcal{Q}_{m}^{0}\right\},\eea
			and 
			\be\label{eq:gamma3ngek}\Gamma_3^{n,\ge k}:=\left\{(x_0,t_0)\in \Gamma_3^n:\ \Pi_j(x_0,t_0)\in\mathcal{Q}_{j}^{0}\quad \forall j\in \{2,\ldots,k-1\}\right\}.\ee 
			If $\varphi\in C^\infty(B_1')$, we also set $$\Gamma_3^{n,\infty}:=\bigcap_{k\ge 4}\Gamma_3^{n,\ge k}=\left\{(x_0,t_0)\in \Gamma_3^n:\ \Pi_j(x_0,t_0)\in\mathcal{Q}_j^{0}\quad \forall j\in\N_{\ge 2}\right\}.$$
			\noindent The set of time-like cubic points $\Gamma_{3}^{n}$ can then be decomposed as follows:
			\be\label{eq:splitting-time-like}\Gamma_3^n=\Gamma_3^{n,\ge k}\cup \left(\Gamma_3^{n}\setminus\Gamma_3^{n,\ge k}\right)=\Gamma_3^{n,\ge k}\cup\bigcup_{m=3}^{k-1}\Gamma_{3}^{n,m}.\ee
			For any $m\in \{3,\ldots,k-1\}$ and any $(x_0,t_0)\in \Gamma_{3}^{n,m}$, the spine of 
            the first nonzero term in the expansion of the normal derivative at $(x_{0},t_{0})$ restricted to $t=0$ is denoted by
			\be\label{eq:spinesm} S_{m}^{x_0,t_0}=\left\{\xi\in\R^{n}: \partial_{\xi}\Pi_m(x_0,t_0)(\cdot,0)\equiv0\right\}.\ee 
			Note that since $\Pi_m(x_0,t_0)(\cdot, 0)\not\equiv0$, then 
			$$d_m^{x_0,t_0}:=\text{dim}\,S_{m}^{x_0,t_0}\in \{0,\dots,n-1\}.$$
			For every $d\in\{0,\ldots,n-1\}$, we set 
			$$\Gamma_3^{n,m,d}:=\left\{(x_0,t_0)\in \Gamma_3^{n,m}:\ d_m^{x_0,t_0}=d\right\}.$$ 
			Then, $\Gamma_{3}^{n,m}$ can be further split as follows:  
			\be\label{eq:splitting-nmd}\Gamma_{3}^{n,m}=\bigcup_{d=0}^{n-1} \Gamma_3^{n,m,d}\qquad \forall  m\in \{3,\dots,k-1\}.\ee
		This further stratification is motivated by the following remark, which relates the strata to the order of vanishing of the spatial derivatives of the time graph $T$.
			\begin{remark}\label{remark:derivatives-of-T}
			    The construction of the function $T$ in the proof of \cref{prop:t=T(x)}, together with \eqref{eq:prop-of-theta}, identifies its spatial derivatives on the strata introduced above. More precisely, for every $m\in\{3,\ldots,k-1\}$ and every $(x_0,t_0)\in\Gamma_3^{n,m}$, in the neighborhood where $T$ is defined, we have $T(x'_0,0)=t_0$ and \be\label{eq:derivatives-T-nm1}
                    D_{x'}^\beta T(x'_0,0)=0\qquad\text{for }|\beta|\in\{1,\ldots,m-1\}\ee and \be\label{eq:derivatives-T-nm2}D_{x'}^\beta T(x'_0,0)
                    =-\frac{D_{x'}^\beta \Theta(x'_0,0,t_0)}{\partial_t\Theta(x'_0,0,t_0)}=-\frac{D_{x'}^\beta Q_m^{x_0,t_0}(0,0)}{\partial_tQ_2^{x_0,t_0}(0,0)}\qquad
                    \text{for }|\beta|=m.\ee Similarly, for every $(x_0,t_0)\in\Gamma_3^{n,\ge k}$ in that neighborhood we have
                    \be\label{eq:derivatives-T-ngek} D_{x'}^\beta T(x'_0,0)=0
                    \qquad\text{for }|\beta|\in\{1,\ldots,k-1\}.\ee
                Indeed, by construction, the spatial Taylor polynomial of \,$T$ at $x_0$ coincides with that of the function
                implicitly defined by the parabolic Taylor polynomial of $\Theta$ at $(x_0,t_0)$. Since $\partial_t\Theta(x'_0,0,t_0)\neq0$,
                comparing homogeneous terms and using $Q_j^{x_0,t_0}(\cdot,0)\equiv0$ for $2\le j<m$ gives \eqref{eq:derivatives-T-nm1} and \eqref{eq:derivatives-T-nm2}. At points of $\Gamma_3^{n,\ge k}$, these vanishing conditions hold for every $2\le j\le k-1$, yielding \eqref{eq:derivatives-T-ngek}.
			\end{remark}

			\subsection{Dimensional bound for $\Gamma_{3}^{n}\setminus \Gamma_{3}^{n,\ge k}$.}
			In the following proposition we prove a finer version of \cref{prop:stratification-for-u} for $\Gamma_{3}^{n}\setminus \Gamma_{3}^{n,\ge k}$. As a consequence, we obtain a dimensional bound for its spatial projection.
			
			\begin{proposition}\label{prop:dimensional-bound-time-like}
				Let $u$ be a solution of \eqref{eq:parabolic-signorini-problem}, with obstacle $\varphi$ satisfying \eqref{eq:ostacolo}. Then, for every $\alpha\in(0,\gamma)$, $m\in\{3,\ldots,k-1\}$ and $d\in\{0,\ldots,n-1\}$, 
				$\Gamma_{3}^{n,m,d}$ is locally contained in a $d$-dimensional manifold of class $C^{k-m,\alpha}$ of the form $(x_{d+1},\ldots,x_{n},t)=g(x_1,\ldots,x_d)$, after a suitable rotation in the thin space.
				Consequently, 
				$$\text{dim}_{\HH}(\pi_x(\Gamma_3^{n}\setminus \Gamma_3^{n,\ge k}))\le n-1.$$
			\end{proposition}
			
			\begin{proof} We argue as in the proof of \cref{prop:stratification-for-u}.    
				Let $(x_0,t_0)\in\Gamma_3^{n,m,d}$, for some $m\in\{3,\ldots,k-1\}$ and $d\in\{0,\ldots,n-1\}$. Let $T\in C^{k-1,\alpha}(B_\rho'(x_0))$ be the function given by \cref{prop:t=T(x)}, for some $\rho>0$. For every multiindex $\beta\in \N^n$ with $|\beta|=m-1$, we set $F_\beta(x'):=\partial_{x'}^\beta T(x',0)\in C^{k-m,\alpha}(B_\rho'(x_0))$. 
				Then, by \eqref{eq:derivatives-T-nm1}, 
				\bea\Gamma_{3}^{n,m,d}\cap \mathscr{C}'_\rho(x_0,t_0)\subseteq \{t=T(x',0)\}\cap\bigcap_{|\beta|=m-1}\{F_\beta=0\}.\eea
				Consider the map $F(x')=(F_\beta(x'))_{|\beta|=m-1}$. 
				Arguing as in the proof of \cref{prop:stratification-for-u}, and using \eqref{eq:derivatives-T-nm2}, we get
				$$\ker\left(D_{x'}F(x_0')\right)=S_m^{x_0,t_0},$$ where $S_m^{x_0,t_0}$ is the spine defined in \eqref{eq:spinesm}. Since $d_m^{x_0,t_0}:=\dim S_m^{x_0,t_0}=d$, we have $${\rm rank}\left(D_{x'}F(x_0')\right)=n-d.$$ As a consequence, there exist distinct multiindices $\beta_1,\ldots, \beta_{n-d}$, with $|\beta_{i}|=m-1$, such that 
				the spatial gradients $\nabla_{x'} F_{{\beta_1}}(x_0'),\ldots,\nabla_{x'} F_{{\beta_{n-d}}}(x_0')$ are linearly independent.
				Then, around $(x_{0},t_{0})$, the spatial projection of $\Gamma_3^{n,m,d}$ is contained in the zero set of the function
				\bea\Psi(x'):=\left(F_{\beta_1}(x'),\ldots,F_{\beta_{n-d}}(x')\right),\eea 
				which is locally a $d$-dimensional manifold of class $C^{k-m,\alpha}$. 
                By the implicit function theorem, after a suitable rotation in the thin space, its zero set is locally a $d$-dimensional manifold of class $C^{k-m,\alpha}$ of the form $(x_{d+1},\ldots,x_n)=\widetilde g(x_1,\ldots,x_d)$.
                Setting $g(x_1,\ldots,x_d):=\big(\widetilde g(x_1,\ldots,x_d),T(x_1,\ldots,x_d,\widetilde g(x_1,\ldots,x_d),0)\big),$ we obtain the claimed local containment in a $d$-dimensional manifold of class $C^{k-m,\alpha}$.
				Finally, by a covering argument,
				$$\text{dim}_{\HH}(\pi_x(\Gamma_3^{n,m,d}))\le \text{dim}^{\text{eucl}}_{\HH}(\Gamma_3^{n,m,d})\le d\le n-1,$$ where $\text{dim}^{\text{eucl}}_{\HH}$ is the dimension with respect to the usual Euclidean distance. Recalling the splittings \eqref{eq:splitting-time-like} and \eqref{eq:splitting-nmd}, this estimate concludes the proof. 
			\end{proof}

			\subsection{Cleaning results for $\Gamma_{3}^{n,\ge k}$}
			In the next proposition we improve the cleaning result of \cref{prop:cleaning-odd} for the sets $\Gamma_3^{n,m}$ and $\Gamma_{3}^{n,\ge k}$. We recall the definition of $\text{Clean}$ from \cref{def:cleaned}.

			\begin{proposition}\label{prop:cleaning-time-like}
				Let $u$ be a solution of \eqref{eq:parabolic-signorini-problem}, with obstacle $\varphi$ satisfying \eqref{eq:ostacolo}. Then, for every $m\in\{3,\ldots,k-1\}$ and $\alpha\in(0,\gamma)$, we have
				$$\Gamma_3^{n,m}\in\text{Clean}(m)\qquad\text{and}\qquad \Gamma_3^{n,\ge k}\in \text{Clean}(k+\alpha-1).$$
				Moreover, if $\varphi\in C^\infty(B_1')$, then $$\Gamma_3^{n,\infty}\in \text{Clean}(\infty).$$
			\end{proposition}

			\begin{proof}[Proof of \cref{prop:cleaning-time-like}]
				We only prove the cleaning result for $\Gamma_{3}^{n,\ge k}$, the cleaning for $\Gamma_3^{n,m}$ following similarly.
				In order to prove that $\Gamma_3^{n,\ge k}\in\text{Clean}(k+\alpha-1)$, by the time-monotonicity of the solution, it is enough to show that for every $(x_0,t_0)\in\Gamma_3^{n,\ge k}$ there exists $\rho>0$ and $C>0$ such that 
				$$|t_1-t_0|\le C|x'_1-x'_0|^{k+\alpha-1}\qquad\forall (x_1,t_1)\in\Gamma_3^{n,\ge k}\cap \mathscr{C}'_\rho(x_0,t_0).$$
                Let $T$ be the function from \cref{prop:t=T(x)}. By \eqref{eq:derivatives-T-ngek}, its tangential derivatives at $x_{0}$ vanish up to order $k-1$.
				For every $(x_1,t_1)\in\Gamma_3^{n,\ge k}\cap \mathscr{C}'_\rho(x_0,t_0)$, since $\Gamma_3^n\cap \mathscr{C}'_\rho(x_0,t_0)\subseteq\{t=T(x',0)\}$, we have $t_i=T(x_i',0)$. Therefore, the Taylor expansion of $T$ at $x_{0}$ gives $$|t_1-t_0|=|T(x_1',0)-T(x_0',0)|\le C|x_1'-x_0'|^{k+\alpha-1}.$$
				By the arbitrariness of $(x_1,t_1)$, we conclude the proof in the case $\varphi\in C^{k,\gamma}(B_1')$. When $\varphi$ is smooth, the cleaning result of $\Gamma_3^{n,\infty}$ follows from the cleaning result for $\Gamma_3^{n,\ge k}$ and by $\Gamma_3^{n,\infty}=\bigcap_{k\ge 4}\Gamma_3^{n,\ge k}$.
			\end{proof}
		
			\begin{proof}[Proof of \cref{prop:complete-time-like}] The result follows by combining \cref{prop:dimensional-bound-time-like} and \cref{prop:cleaning-time-like}.
			\end{proof}

			\section{Proof of the main results}\label{sec:proof_main}
			In this final section, we conclude the proof of our main results, \cref{t:thm0}, \cref{t:thm1}, and \cref{thm:n-1andtime-like}. Throughout the section, for any set $E\subseteq (\R^n\times\{0\})\times\R$, we call $\pi_{t}(E):= \{t\in \R: \exists \, x\text{ such that } (x,t)\in E\}$ the time projection. Moreover, for any $t\in \R$, we denote by $E_t:=\{x\in \R^n\times \{0\}:\ (x,t)\in E\}$ the time slice at time $t$. 
			We start by proving that the singular set has parabolic dimension at most $n$.
			
			\begin{proof}[Proof of \cref{t:thm0}]
				We need to prove that $\text{dim}_{\text{par}}(\Sigma)\le n$. By \cref{prop:onlyle3}, we know that $\Sigma=\Gamma_{2}\cup \Gamma_{3}$. Now, on the one hand, we have the trivial bound $\dim_{\mathcal{H}}(\pi_x(\Sigma))\le n$, since $\pi_{x}(\Sigma)\subseteq B_1'$.
				On the other hand, by the cleaning results in \cref{prop:cleaning-gamma2} and \cref{prop:cleaning-odd}, we have $\Gamma_2\in\text{Clean}(2)$ and $\Gamma_3\in\text{Clean}(2+\eta)$ for some $\eta>0$. 
				Then, applying \cref{lemma:par-gmt} separately to $\Gamma_{2}$ and $\Gamma_{3}$, we deduce that both sets have parabolic dimension at most $n$. This concludes the proof.
			\end{proof}
			
			Next, we pass to the proof of the generic regularity result in \cref{t:thm1}. As a byproduct of our analysis, we are actually able to prove the more general result contained in the following proposition.
			\begin{proposition}\label{prop10}
				Let $u$ be a solution of \eqref{eq:parabolic-signorini-problem}, with obstacle $\varphi\in C^{k,\gamma}(B_1')$ satisfying \eqref{eq:ostacolo}.
				\begin{itemize} 
					\item[i)] If $n+1=2$, then $$\dim_{\mathcal{H}}(\pi_t(\Sigma))\le \frac{1}{k+\gamma-1}.$$
					In particular, if $\varphi\in C^\infty(B_1')$, then $\dim_{\mathcal{H}}(\pi_t(\Sigma))=0$.
                    
                    \smallskip
					\item[ii)] If $n+1=3$, then $$dim_{\mathcal{H}}(\pi_t(\Sigma))\le \max\left\{\frac12,\frac{2}{k+\gamma-1}\right\}.$$
					In particular, if $\varphi\in C^\infty(B_1')$, then $\dim_{\mathcal{H}}(\pi_t(\Sigma))\le 1/2$.

                    \smallskip
					\item[iii)] If $n+1=4$, then $$\mathcal{H}^1(\pi_t(\Sigma))=0.$$

                    \smallskip
					\item[iv)] If $n+1\ge 5$, then $$\dim_{\mathcal{H}}(\Sigma_t)\le n-3\qquad\text{for a.e. }t\in (-1,0].$$
					\end{itemize}    
				\end{proposition}   
				
				\begin{proof} 
					By \cref{prop:onlyle3}, \eqref{eq:slitting-gamma3} and \eqref{eq:splitting-time-like}, we may use the splitting \be\label{eq:splitting-final}\Sigma= \Gamma_{2}\cup (\Gamma_{3}\setminus \Gamma_{3}^{n,\ge k}) \cup \Gamma_{3}^{n,\ge k}.\ee For each subset we use one of the GMT lemmas from \cref{geometricmeastheory}. 
					\begin{itemize}
						\item[(1)] \noindent \textit{The set \,$\Gamma_{2}$.} By \cref{corollary:stratification-even}, we know that $\pi_{x}(\Gamma_{2})$ is $(n-1)$-rectifiable and so $\dim_{\mathcal{H}}(\pi_{x}(\Gamma_{2}))\le n-1$. Moreover, by \cref{prop:cleaning-gamma2}, $\Gamma_{2}\in \text{Clean}(2)$. Then, if $n+1\le 3$, point $a)$ in \cref{lemma:gmt} gives $\dim_{\mathcal{H}}(\pi_{t}(\Gamma_{2}))\le (n-1)/2$. If $n+1=4$, \cref{lemma:colombo} gives $\mathcal{H}^{1}(\pi_{t}(\Gamma_{2}))=0$. Finally, if $n+1\ge 5$, point b) in \cref{lemma:gmt} gives $\dim_{\mathcal{H}}((\Gamma_{2})_{t})\le n-3$ for almost every $t\in(-1,0]$. 

                        \smallskip
						\item[(2)] \noindent \textit{The set \,$\Gamma_{3}\setminus \Gamma_{3}^{n,\ge k}$.} By \cref{prop:stratification-for-u} and \cref{prop:complete-time-like}, we have $\dim_{\mathcal{H}}(\pi_{x}(\Gamma_{3}\setminus \Gamma_{3}^{n,\ge k}))\le n-1$. Moreover, by \cref{prop:cleaning-odd}, $\Gamma_{3}\setminus \Gamma_{3}^{n,\ge k}\in \text{Clean}(2+\eta)$ for some $\eta\in (0,1)$. Then, if $n+1\le 4$, point $a)$ in \cref{lemma:gmt} gives $\dim_{\mathcal{H}}(\pi_{t}(\Gamma_{3}\setminus \Gamma_{3}^{n,\ge k}))\le (n-1)/(2+\eta)<1$. If instead $n+1\ge 5$, point b) in \cref{lemma:gmt} gives $\dim_{\mathcal{H}}((\Gamma_{3}\setminus \Gamma_{3}^{n,\ge k})_{t})\le n-3-\eta$ for almost every $t\in(-1,0]$.

                        \smallskip
						\item[(3)] \noindent \textit{The set $\Gamma_{3}^{n,\ge k}$.} We have the trivial dimensional bound $\dim_{\mathcal{H}}(\pi_{x}(\Gamma_{3}^{n,\ge k}))\le n$. Moreover, by \cref{prop:complete-time-like}, $\Gamma_{3}^{n,\ge k}\in \text{Clean}(k+\alpha-1)$, for any $\alpha \in (0,\gamma)$. Then, if $n< k+\gamma-1$, point $a)$ in \cref{lemma:gmt} gives $\dim_{\mathcal{H}}(\pi_{t}(\Gamma_{3}^{n,\ge k}))\le n/(k+\gamma-1)<1$. If instead $n\ge k+\gamma-1$, point b) in \cref{lemma:gmt} gives $\dim_{\mathcal{H}}((\Gamma_{3}^{n,\ge k})_{t})\le n-k-\gamma+1<n-3$ for almost every $t\in(-1,0]$.
					\end{itemize} 
					The conclusion follows by combining all the results obtained for each individual term in the splitting of $\Sigma$.  
				\end{proof}
				
				\begin{proof}[Proof of \cref{t:thm1}]
					The first part of the statement $\mathcal{H}^1(\pi_t(\Sigma))=0$ is obtained as a special case of the more general result proved in \cref{prop10}. By \cite[Theorem 11.3]{dgpt17}, the regular set $\Gamma\setminus \Sigma$ is locally an $n$-dimensional submanifold of $\{x_{n+1}=0\}$ of the form $x_{n}=g(x'', t)$, up to a rotation in the thin space, where $x'':=(x_{1},\ldots,x_{n-1})$. Moreover, by \cite[Theorem 11.6]{dgpt17}, $\nabla_{x''}g \in H^{\alpha,\alpha/2}$ for some $\alpha>0$. 
					In addition, if $\varphi\equiv0$, then $g$ is smooth, by \cite[Theorem 7]{smooth-parabolic-signorini}.
					As a consequence, for any $t_0\in (-1,0]\setminus \pi_{t}(\Sigma)$, the set $\{x\in B_{1}': (x,t_0)\in \Gamma\}= \{x\in B_{1}':(x,t_0)\in \Gamma\setminus \Sigma\}$ is an $(n-1)$-dimensional manifold of class $C^{1,\alpha}$ (or 
					$C^\infty$ if $\varphi\equiv0$) of the form $x_n=g(x'',t_0)$. 
				\end{proof}
				
				Finally, we prove \cref{thm:n-1andtime-like} on the fine splitting of the singular set. We prove below a more general result for $C^{k,\gamma}$ obstacles.

				\begin{proposition}\label{prop:general-kgamma}
					Let $u$ be a solution of \eqref{eq:parabolic-signorini-problem}, with obstacle $\varphi\in C^{k,\gamma}(B_1')$ satisfying \eqref{eq:ostacolo}. Let $\Gamma_3^{n,\ge k}\subseteq\Sigma$ be the set given by \cref{prop:complete-time-like}. Then 
					\be\label{eq:dim-bound-final}\text{dim}_{\text{par}}(\Sigma\setminus \Gamma_3^{n,\ge k})\le n-1.\ee
					Moreover:
					\begin{itemize} 
						\item[i)] if $n+1<k+\gamma$, then $$\text{dim}_{\mathcal{H}}(\pi_t(\Gamma_3^{n,\ge k}))\le \frac{n}{k+\gamma-1};$$ 

                        \smallskip
						\item[ii)] if $n+1\ge k+\gamma$, then $$\text{dim}_{\mathcal{H}}((\Gamma_3^{n,\ge k})_t)\le n-k-\gamma+1 \qquad \text{for a.e. $t\in (-1,0]$}.$$ 
					\end{itemize}
					In particular, if $\varphi\in C^\infty(B_1')$, then, calling $\Gamma_{3}^{n,\infty}:= \bigcap_{k\ge 4}\Gamma_{3}^{n,\ge k}$, we have
					\begin{equation*}
						\text{dim}_{\text{par}}(\Sigma\setminus \Gamma_3^{n,\infty})\le n-1\qquad \text{and}\qquad \text{dim}_{\mathcal{H}}(\pi_t(\Gamma_3^{n,\infty}))=0.
					\end{equation*}
				\end{proposition}
				\begin{proof}
					Let us first prove \eqref{eq:dim-bound-final}.
					By \cref{corollary:stratification-even}, \cref{prop:stratification-for-u} and \cref{prop:complete-time-like}, we know that $\pi_{x}(\Gamma_{2})$ and $\pi_{x}(\Gamma_{3}\setminus \Gamma_{3}^{n,\ge k})$ have both Hausdorff dimension at most $n-1$. In addition, by \cref{prop:cleaning-gamma2} and \cref{prop:cleaning-odd} they can both be cleaned up with rate $2$. Using \cref{lemma:par-gmt} we deduce that $\dim_{\text{par}}(\Gamma_{2}) $ and $\dim_{\text{par}}(\Gamma_{3}\setminus \Gamma_{3}^{n,\ge k})$ are at most $n-1$, which implies \eqref{eq:dim-bound-final} by the splitting \eqref{eq:splitting-final}.
					Finally, points $i)$ and $ii)$ were already proved in point (3) in the proof of \cref{prop10}.   
				\end{proof}
				\begin{proof}[Proof of \cref{thm:n-1andtime-like}]
					We take $\Sigma^{\infty}:=\Gamma_{3}^{n,\infty}$, where $\Gamma_{3}^{n,\infty}$ is the set given by \cref{prop:general-kgamma}. The bounds
					$\text{dim}_{\text{par}}(\Sigma\setminus \Sigma^\infty)\le n-1$ and $\text{dim}_{\mathcal{H}}(\pi_t(\Sigma^\infty))=0$
					follow immediately from \cref{prop:general-kgamma}. Finally, since $\Gamma_{3}^{n,\infty}\subseteq \Gamma_{3}^{n}$, \cref{prop:stratification-for-u} implies that $\Sigma^{\infty}$ is locally contained in a smooth $n$-dimensional manifold. This concludes the proof.  
				\end{proof}
				
				\subsection*{Acknowledgements}
				We would like to thank Bozhidar Velichkov and Xavier Fernández-Real for supporting the visits between the authors, which provided valuable opportunities to develop the collaboration underlying this work.
				M.C. also thanks Giacomo Colombo for a useful discussion about the Stefan problem.
                M.C. is supported by the European Research Council (ERC) via the project ERC FiRM - {\em Fine structure and regularity of stationary and moving free boundaries} (grant agreement No. 101230705). 
				R.C. is supported by the Swiss State Secretariat for Education, Research and Innovation (SERI) under contract number MB22.00034 through the project TENSE, and by the Swiss National Science Foundation (SNF grant PZ00P2\_208930).
				C.T.-L. has received funding from the European Research Council (ERC) under the Grant Agreement No. 862342, from AEI project PID2024-156429NB-I00 (Spain), and from the Grant CEX2023-001347-S funded by MICIU/AEI/10.13039/501100011033 (Spain).

                \subsection*{AI use disclosure}
                The authors used AI tools to proofread the manuscript, identify mathematical and typographical errors, assist with literature searches, and generate \cref{fig:thin-ice-model}. The authors take full responsibility for all content of this article.
				
				\printbibliography
			\end{document}